\documentclass[12pt, reqno]{amsart} 
\usepackage{verbatim}
\usepackage{amssymb,amscd,amsfonts,amsbsy}
\usepackage{latexsym}
\usepackage{exscale}
\usepackage{amsmath,amsthm,amsfonts}
\usepackage{mathrsfs}
\usepackage{xcolor} 
\usepackage[colorlinks=true,linkcolor=blue,citecolor=blue,urlcolor=blue]{hyperref} 
\usepackage{esint} 
\usepackage{amssymb} 
\usepackage{stmaryrd}
\usepackage{cite} 
\usepackage{tikz}
\usepackage{enumerate}
\calclayout

\newcommand{\Q}{{Q_{j}^{k}}}

\allowdisplaybreaks[3] 

\newtheorem{theorem}{Theorem}[section]
\newtheorem{proposition}[theorem]{Proposition}
\newtheorem{lemma}[theorem]{Lemma}
\newtheorem{corollary}[theorem]{Corollary}
\newtheorem{definition}[theorem]{Definition}
\newtheorem{remark}[theorem]{Remark}

\numberwithin{equation}{section}

\makeatletter
\@addtoreset{equation}{section}
\makeatother

\def\rn{{\mathbb R^n}}

\def\cc{{\mathbb C}}

\def\B{{\mathscr B}}
\def\M{{\mathscr M}}
\def\I{{\mathscr I}}
\def\T{{\mathscr T}}
\def\B{{\mathscr B}}
\def\A{{\mathscr A}}
\def\S{{\mathscr S}}

\def\F{{\mathscr F}}

\DeclareMathAlphabet{\mathpzc}{OT1}{pzc}{m}{it}

\def\bqo{C_{\widetilde j_0}B(Q_0)}
\def\bq{C_{\widetilde j_0}B(Q)}
\def\bo{C_{\widetilde j_0}B(Q_0)}
\def\bp{C_{\widetilde j_0}B(P_j)}
\def\cjo{C_{\widetilde j_0}}
\newcommand{\an}{\eta}
\newcommand{\f}{\frac}

\def\dashint{\Xint-}
\def\GTA{\mathcal{A}_{\eta, \mathcal{S}, p_0,\gamma}}

\let \d=\delta
\let \la=\lambda

\def\Xint#1{\mathchoice
   {\XXint\displaystyle\textstyle{#1}}%
   {\XXint\textstyle\scriptstyle{#1}}%
   {\XXint\scriptstyle\scriptscriptstyle{#1}}%
   {\XXint\scriptscriptstyle\scriptscriptstyle{#1}}%
   \!\int}
\def\XXint#1#2#3{{\setbox0=\hbox{$#1{#2#3}{\int}$}
     \vcenter{\hbox{$#2#3$}}\kern-.5\wd0}}

\DeclareMathOperator*{\esssup}{ess\,sup}

\let \d=\delta
\let \la=\lambda

\newcommand{\avf}{{\langle f\rangle}}

\def\d{\mathcal{D}}

\def\F{\mathcal{F}}
\def\N{\mathbb{N}}

\def\R{\mathbb{R}}

\def\S{\mathcal{S}}

\newcommand{\D}{\mathscr{D}}

\begin{document}
\title[The multilinear fractional sparse operator theory ...]
{\bf The multilinear fractional sparse operator theory I: pointwise domination and weighted estimate}

\author[X. Cen]{Xi Cen$^{*}$}
\address{Xi Cen\\
Department of applied mathmatics\\
	School of Mathematics and Physics\\
	Southwest University of Science and Technology\\
	Mianyang 621010 \\
	People's Republic of China}\email{xicenmath@gmail.com} 

\author[Z. Song]{Zichen Song}
\address{Zichen Song\\
	School of Mathematics and Stastics\\
	Xinyang Normal University\\
	Xinyang, 464000\\
	People's Republic of China}\email{zcsong@aliyun.com}

\date{December 31, 2024.}



\thanks{$^{*}$ Corresponding author, Email: xicenmath@gmail.com}

\begin{abstract} 
How to establish some specific quantitative weighted estimates for the generalized commutator of multilinear fractional singular integral operator $\mathcal{T}_{\eta}^{{\bf b}}$ is the focus of this paper, which is defined by  $$\mathcal{T}_{\eta}^{{\bf b}}(\vec{f})(x):= \mathcal{T}_{\eta}\left((b_1(x) - b_1)^{\beta_1}f_1,\ldots,(b_m(x) - b_m)^{\beta_m}f_m\right)(x),$$
where $\mathcal{T}_{\eta}$ is a multilinear fractional singular integral operator, ${\bf b}:=({b_1}, \cdots ,{b_m})$ is a set of symbol functions, and $({\beta _1}, \cdots ,{\beta _m}) \in {\mathbb{N}_0^m}$.

Pointwise dominating the aforementioned commutator leads us to consider a class of higher order multi-symbol multilinear fractional sparse operator ${\mathcal A}_{\eta ,\S,\tau}^\mathbf{b,k,t}$ to achieve this long-cherished wish. Therefore, it suffices to construct its quantitative weighted estimates, which firstly include the characterization of several types of multilinear weighted conditions $A_{\vec p,q}^*$, $W_{\vec p,q}^\infty$, and $H_{\vec p,q}^\infty$. Within the scope of this work, Bloom type estimate for first order multi-symbol multilinear fractional sparse operator is established herein. Moreover, we derive two distinct Bloom type estimates for higher order multi-symbol multilinear fractional sparse operator by using {\tt "maximal weight method"} and {\tt "iterated weight method"} respectively, which not only refines some of Lerner's methods but greatly enhances the generality of our conclusions. 
Endpoint quantitative estimates for multilinear fractional singular integral operators and their first order commutators are also obtained as the last main result.
It is also worthy of highlighting that some important multilinear fractional operators are applicable to our results as applications.

Overall, our results extend the achievements of sparse operator theory in recent ten years under the setting of {\tt "multilinear"} , {\tt "fractional"}, {\tt "higher order commutator"}, and {\tt "space of homogeneous type"}.

\end{abstract}

\maketitle
\tableofcontents
\section{\bf Introduction}\label{Introduction}

\subsection{Background}
~~

In the past decade, sparse operator theory seems to have been greatly expanded, as it can control singular integral operators point by point or in sparse form. However, for the existing results of sparse operator theory, they are always in diagonal form, and non diagonal problems or what we call fractional problems have not been fully answered for a long time. This includes the problems for higher order commutator of multilinear fractional singular integral operators, Bloom type estimates of them, and so on. 

This paper will thoroughly solve these unsolved mysteries. For this, we need to introduce some basic concepts, which will be very helpful for our next. Since the results of this article are discussed under the setting of space of homogeneous type, it is necessary to first introduce some definitions about it. These information may be found in \cite{Yang2019, 2408}.


\begin{definition}
	Let $d: X \times X \rightarrow [0, \infty)$ be a positive function and $X$ be a set, then the quasi-metric space $(X, d)$ satisfies the following conditions:
	\begin{enumerate}
		\item  When $x=y$, $d(x, y)=0$.
		\item $d(x, y)=d(y, x)$ for all $x, y \in X$.
		\item  For all $x, y, z \in X$, there is a constant $A_0 \geq 1$ such that $d(x, y) \leq A_0(d(x, z)+d(z, y))$.
	\end{enumerate}
\end{definition}

\begin{definition}
	Let $\mu$ be a measure of a space $X$. For a quasi-metric ball $B(x,  r)$ and any $r>0$, if $\mu$ satisfies doubling condition, then there exists a doubling constant $C_\mu \geq 1$, such that  
	\begin{align}\label{def_hom}
    0<\mu(B(x, 2 r)) \leq C_\mu \mu(B(x, r))<\infty.
  \end{align}
\end{definition}

\begin{definition}
	For a non-empty set $X$ with a qusi-metric $d$, a triple $(X, d, \mu)$ is said to be a space of homogeneous type if $\mu$ is a regular measure which satisfies doubling condition on the $\sigma$-algebra, generated by open sets and quasi-metric balls.
\end{definition}

For the construction of dyadic cubes in space of homogeneous space, we refer reader to Sect. \ref{Pre.}.  Next we first introduce some important notation.

\medskip 
\medskip 
\medskip 

{\bf{Notation}}
~~

$L_b^{\infty}(X)$ is used to represent the set of $f \in L^{\infty}(X)$ with a bounded support.
A weight \(\omega: X \rightarrow [0, \infty]\) is a locally integrable function that satisfies \(0 < \omega(x) < \infty\) for almost every \(x \in X\). Given a weight \(\omega\), its associated measure is defined as \(d\omega(x) = \omega(x) d\mu(x)\), and we also denote $\omega(E)=\int_E \omega(x)\,d\mu$. 
Given a set $E \subseteq X$, a function $f$, and $\sigma$ is a weight, we denote weighted averages by
\[\avf_{\sigma,E}=\frac{1}{\sigma(E)}\left(\int_E f \sigma \,d\mu\right).\]
If $\sigma =1$, we denote $\avf_{\sigma,E}$ by $\avf_{E}$.

To simplify our discussion, we fix some notation. We always let \(\mathcal{S}\) be a sparse family, and set multi-symbols \(\mathbf{b} = (b_1, \ldots, b_{m}) \in (L_{loc}^1(X))^m\), multi-index \(\mathbf{k} = (k_1, \ldots, k_{m}) \in (\N_0)^m \), and multi-index \(\mathbf{t} = (t_1, \ldots, t_{m}) \in (\N_0)^m \).
We say that a multi-index \(\mathbf{k} = (k_1, \ldots, k_{m}) \in (\N_0)^m \) is bigger than a multi-index \(\mathbf{t} = (t_1, \ldots, t_{m}) \in (\N_0)^m \), if for any $i \in \left\{ {1, \cdots ,m} \right\}$, there holds that $k_i>t_i$. At this point, we always use $\mathbf{t} \prec \mathbf{k}$ or $ \mathbf{k} \succ \mathbf{t}$ to represent this.

We denote \(\tau_m = \{1, \ldots, m\}\). For any subset \(\tau \subseteq \tau_m\), let \(|\tau|\) represent the number of elements in \(\tau\), and define its complement as \(\tau^c = \tau_m \setminus \tau\). 

In order to simplify, we always let $\tau \subseteq \tau_m$ is a strictly monotonically increasing subsequence and write it by
$\tau = \left\{ {\tau (j)} \right\}_{j = 1}^{\left| \tau  \right|}.$
If $|\tau|=l \in \left\{ {1, \cdots ,m} \right\}$, then we sometime write it by \(\tau_{\ell}=\left\{ {\tau (j)} \right\}_{j = 1}^{\ell}\).

Let ${\bf b} \in {\left( {L_{loc}^1\left( X \right)} \right)^m}$. Given a $m$-sublinear operator $G$ and a multi-index \(\mathbf{k} = (k_1, \ldots, k_{m}) \in (\N_0)^m \), its generalized commutator can be defined by 
\begin{align*}
	G_{\tau_{\ell}}^{{\bf b, k}}(\vec{f})(x):=G\left((b_1(x) - b_1)^{\beta_1}f_1,\ldots,(b_m(x) - b_m)^{\beta_m}f_m\right)(x),
\end{align*}

where
\[
\beta_i = 
\begin{cases}
	k_i, & \text{if} \quad i \in \tau_{\ell}, \\
	0, & \text{if} \quad i \in \tau_{\ell}^{c}.
\end{cases}
\] 
When ${\bf k}=\vec{1}$, we denote it by $G_{\tau_{\ell}}^{{\bf b}}$.
When ${\bf k}=0$, the generalized commutator is back to $G$.

\medskip 
\medskip 
\medskip 
We now give some crucial conceptions about the multilinear fractional singular integral operators and their commutators. 

\begin{definition}
Let $\eta  \in [0,m)$, $\B$ is a quasi-Banach space, and $B(\cc,\B)$ is the space of all bounded operators from $\cc$ to $\B$. Set an operator-valued function $Q_\eta:(X^{m+1} \backslash \Delta ) \to B(\cc,\B),$ where $\Delta = \{ (x,\vec y) \in X^{m+1} :x = {y_1} =  \cdots  = {y_m}\}$. For any $\vec f \in {\left( {L_{loc}^1(X)} \right)^m}$, {\tt $\B$-valued $m$-linear fractional singular integral operator} $\mathcal{T}_{\eta}$ is defined by
    \begin{align*}
\mathcal{T}_{\eta}(\vec{f})(x):=\int_{X^m} Q_{\eta}\left(x, \vec y\right) \left(\prod_{j=1}^{m}f_j\left(y_j\right)\right) d \mu(\vec{y}).
    \end{align*}
  \end{definition}

We usually use the following Hörmander kernel, which is more general than multilinear Dini type  Calder\'on-Zygmund kernels.

\begin{definition}
	Let  $\eta  \in [0,m)$, $1 \le r \le\infty$, and $\mathcal D$ is a dydic lattice of $X$. Set an operator-valued function $K:(X^{m+1} \backslash \Delta ) \to B(\cc,\B),$ where $\Delta = \{ (x,\vec y) \in X^{m+1} :x = {y_1} =  \cdots  = {y_m}\}$. We say $K$ satisfies {\tt $\B$-valued $m$-linear  $L^r$-Hörmander condition} if 
	\begin{align*}
		\mathscr{H}_r:= \mathop {\sup }\limits_{Q \in {\mathcal D}} \mathop {\sup }\limits_{x,{x^\prime } \in \frac{1}{2}Q} \sum\limits_{k = 1}^\infty  \mu  {({2^k}Q)^{\frac{m}{r}}}{\left\| {{{\left\| {K(x, \cdot ) - K(x', \cdot )} \right\|}_\B}{\chi _{{{\left( {{2^k}Q} \right)}^m}\backslash {{\left( {{2^{k - 1}}Q} \right)}^m}}}\left(  \cdot  \right)} \right\|_{{L^{r'}}}} < \infty.
	\end{align*}
	We will write $\mathcal{H}_r$ for the class of kernels satisfying the {\tt $\B$-valued $m$-linear $L^r$-Hörmander condition}.
\end{definition}

\begin{remark}\label{HSIO}
Let $\mathcal{T}_{\eta}$ is a $\B$-valued $m$-linear fractional singular integral operator. We say $\mathcal{T}_{\eta}$ is a {\tt $\B$-valued $m$-linear fractional Hörmander type singular integral operator}, if its kernel $Q_\eta \in \mathcal{H}_r$ for some $r \in [1,\infty]$.
\end{remark}

Next, we need to establish several types of multilinear fractional weighted conditions, which are essential for studying  quantitative  weighted estimates.

\begin{definition}\label{vweight3}
	Let  $\mathcal D$ be a dyadic lattice of $X$, $p_i \in (1,\infty)$, $i=1,\cdots,m$, $\eta=\sum\limits_{i = 1}^m {\frac{1}{{{p_i}}}}-\frac{1}{{{q}}} \in [0,m)$, and $(u,\vec{\omega}):=(u,\omega_1,\ldots,\omega_m)$ is a multiple weight. We write $(u,\vec{\omega}) \in A_{\vec{p},q}(X)$ if it satisfies that
	\begin{align*}
		{\left[ u,\vec{\omega} \right]_{{A_{\vec p,q}}(X)}}: = \mathop {\sup }\limits_{Q \in \mathcal D} {\mu(Q)^{\eta - m}}{\left\| {u {\chi _Q}} \right\|_{{L^{q}}}}\prod\limits_{i = 1}^m {{{\left\| {\omega _i^{ - 1}{\chi _Q}} \right\|}_{{L^{{p_i}^\prime}}}}}  < \infty.
	\end{align*}
	Generally, ${A_{\vec p,q}}(X)$ is denoted by ${A_{\vec p,q}}$. While $u=\omega:=\prod\limits_{i = 1}^m {{\omega _i}}$, we often write ${\left[ (u,\vec{\omega}) \right]_{{A_{\vec p,q}}(X)}}$ as ${\left[ \vec{\omega} \right]_{A_{\vec p,q}(X)}}$.
	
	We write $(u,\vec{\omega}) \in A_{\vec{p},q}^{\star}(X)$ if
	\begin{align*}
		[u,\vec \omega]_{A^{\star}_{\vec p,q}(X)} := [u^{\frac{1}{q}},\omega_1^{\frac{1}{p_1}},\ldots,\omega_m^{\frac{1}{p_m}}]^q_{A_{\vec p,q}(X)}= \sup_{Q \in \mathcal{D}} \langle u\rangle_Q \prod_{i=1}^{m}\langle  \omega_i^{1-p_i'}\rangle_Q^{\frac{q}{p'_i}} < \infty.
	\end{align*}
	
	While $u=\omega:=\prod\limits_{i=1}^{m}\omega_i^{\frac{q}{p_i}}$, we always denote $[u,\vec\omega]_{A_{\vec p,q}^{\star}(X)}$ by $[\vec \omega]_{A_{\vec p,q}^{\star}(X)}$ for short. Moreover, if $\sigma_i=\omega_i^{1-p_i'}$, we define
	$
	\|u,\vec \sigma\|_{A^{\star}_{\vec p,q}(X)}:= [u,\vec\omega]_{A_{\vec p,q}^{\star}(X)}.
	$
\end{definition}

We will inevitably need to characterize the properties for each weight of a multiple weight, therefore, we have the following basic conclusion.

	\begin{lemma}[\cite{Xue2010}, Theorem 2.1]\label{vweight4}
		Under the assumption of Definition \ref{vweight3},
		$\vec{\omega} \in A_{\vec{p},q}^{\star}(X)$ implies that
		$$
		\left\{\begin{array}{l}
			\sigma_i \in A_{m p_i^{\prime}}(X) \subseteq A_\infty(X), \quad i=1, \ldots, m, \\
\omega \in A_{m q}(X) \subseteq A_\infty(X).
		\end{array}\right.
		$$
	\end{lemma}

We would also like to introduce some  multilinear fractional $A_\infty$ conditions that will play an indelible role in our main results.

\begin{definition}\label{weight.1}
Let  $\mathcal D$ be a dyadic lattice of $X$, $p_i \in (1,\infty)$, $i=1,\cdots,m$, $\eta:=\sum\limits_{i = 1}^m {\frac{1}{{{p_i}}}}-\frac{1}{{{q}}} \in [0,m)$, and $\vec{\omega}:=(\omega_1,\ldots,\omega_m)$ is a multiple weight. 
	We write $\vec \omega \in W_{\vec p,q}^\infty(X)$ if
	\begin{equation}\label{eq:Fujii_constant}
		[\vec \omega]_{W_{\vec p,q}^\infty(X)}=\sup_{Q \in \mathcal D} \Big(\int_Q\prod^m_{i=1}M(\omega_i \mathbf \chi_Q)^{\frac{q}{p_i}}d \mu\Big)\Big(\int_Q\prod^m_{i=1}\omega_i^{\frac{q}{p_i}} d \mu\Big)^{-1}<\infty.
	\end{equation}
	We write $\vec \omega \in H_{\vec p,q}^\infty(X)$ if
	\begin{equation}\label{eq:Hruscev_constant}
		[\vec \omega]_{H_{\vec p,q}^\infty(X)}:= \sup_{Q \in \mathcal D} \prod_{i=1}^m \langle \omega_i\rangle_Q^{\frac q{p_i}} \exp\Big( \langle\log \omega_i^{-1}\rangle_Q \Big)^{\frac q{p_i}} <\infty.
	\end{equation}
\end{definition}

When $\eta = 0$, the space $W_{\vec{p}, q}^{\infty}$ simplifies to $W_{\vec{p}}^{\infty}$. Introduced by Fujii in \cite{Li18.9} and revisited by J.~M. Wilson in \cite{Li18.30}, the class $W_{\vec{p}}^{\infty}$ serves as a multilinear counterpart to the classical $A_{\infty}$ constant. Chen further explored the weight properties within $W_{\vec{p}}^{\infty}$ in \cite{L18.4}. 
Similarly, setting $\eta = 0$ for $H_{\vec{p}, q}^{\infty}$ results in a more natural multilinear $A_{\infty}$ constant, denoted as $H_{\vec{p}}^{\infty}$. This extension builds on the classical Hruscev $A_{\infty}$ constant introduced in \cite{Li18.11}.

\begin{remark}
Definitions \ref{vweight3} and \ref{weight.1} are based on dyadic lattice of $X$. In fact, according to the structure of dyadic cubes in space of homogeneous type, we will find that these characteristic constants given on dyadic lattice are completely equivalent to those on all ball of $X$. Ones may refer to \cite[Lemma 2.17]{2408} and \cite[Lemma 2.3]{Li2015_2}.
\end{remark}

In order to be able to  dominate the generalized commutators of multilinear singular integral operators pointwisely, we need to define the following sparse operators.

\begin{definition}\label{cla.sp.}
	Let $\eta  \in [0,m)$, $p_0\ge 1$, and $\gamma>0$. Given a $\delta$--sparse family $\S \subseteq \mathcal{D}$, we define $\GTA$ as follows,
	\begin{align*}
		\GTA(\vec{f})(x):=\left(\sum_{Q\in S} \left[ \mu(Q)^{\an} \prod_{i=1}^m \langle f_i \rangle_{Q,p_0}\right]^\gamma \mathbf \chi_{Q}(x)\right)^{1/\gamma},
	\end{align*}
	where for any cube $Q$,
	$$
	\langle f \rangle_{Q,p_0} := \left(\frac{1}{\mu(Q)} \int_Q |f(x)|^{p_0}d \mu \right)^{\frac{1}{p_0}}.
	$$
	In particular, when $\eta = 0$ and $p_0,\gamma=1$  we denote $\GTA  $ by $\mathcal{A}_{\mathcal{S}}$; when $\eta = 0$ and $p_0=1$  we denote $\GTA  $ by $\mathcal{A}_{\gamma, \mathcal{S}}$.
\end{definition}
 Moreover, when the natural regression is linear for \( m = 1 \), we denote the operators defined above by $A_{\eta, p_0, \gamma, \mathcal{S}}$.

\begin{definition}\label{new.sp.}
Let $\eta  \in [0,m)$, $ r \in [1, \infty)$, \(\mathcal{S}\) be a sparse family, and multi-symbols \(\mathbf{b} = (b_1, \ldots, b_{m}) \in (L_{loc}^1(X))^m\). Suppose that $\mathbf{t}$ and $\mathbf{k}$ are both multi-indexs with \(\mathbf{t} \prec \mathbf{k}\).
{\tt higher order multi-symbol multilinear fractional sparse operator} is defined by
\begin{align*}
 \quad {\mathcal A}_{\eta ,\S,\tau,r}^\mathbf{b,k,t} (\vec f)(x)&=\sum_{Q \in \mathcal{S}} \mu(Q)^{\an \cdot \frac{1}{r}}\\
 &\times \prod_{i \in \tau}  \left|b_i(x) - b_{i,Q}\right|^{k_i - t_i} \langle\left|f_i (b_i - b_{i,Q})^{t_i}\right|\rangle_{Q,r} \prod_{j \in \tau^c} \langle\left|f_j\right|\rangle_{Q,r} \chi_Q(x),\\
\end{align*}
  and {\tt first order multi-symbol multilinear fractional sparse operator} is defined by
  \begin{align*}
    {\mathcal A}_{\eta ,\S,\tau,r}^\mathbf{b} (\vec f)(x)& = \sum_{Q \in \S} \mu(Q)^{\an \cdot \frac{1}{r}}\\
    &\times \prod_{i \in \tau}\left|b_i(x)-b_{i, Q}\right|\left\langle f_i\right\rangle_{Q, r}  \prod_{j \in \tau_{\ell} \backslash \tau}\left\langle\left(b_j-b_{j, Q}\right) f_j\right\rangle_{Q, r}  \prod_{k \in \tau_m \backslash \tau_{\ell}}\left\langle f_k\right\rangle_{Q, r} \chi_Q(x).
  \end{align*}
  To illustrate the endpoint behavior, let $\tau \in \tau_m$, we define
\begin{align*}
  \mathcal{A}_{\eta,\mathcal{S}, L(\log L)^r}^{\tau^{c}}(\vec{f})(x) & :=\sum_{Q \in \mathcal{S}} \mu(Q)^{\an \cdot \frac{1}{r}}\prod_{i \in \tau}\left\langle f_i\right\rangle_{Q, r} \times \prod_{j \in \tau^{c}}\left\|f_j^r\right\|_{L(\log L)^r, Q}^{1 / r} \chi_Q(x). \\
  \mathcal{M}_{\eta,L(\log L)^r}^{\tau^{c}}(\vec{f})(x) & :=\sup _{Q \in \d} \mu(Q)^{\an \cdot \frac{1}{r}} \prod_{i \in \tau}\left\langle f_i\right\rangle_Q \times \prod_{j \in \tau^{c}}\left\|f_j\right\|_{L(\log L)^r, Q} \chi_Q(x).
  \end{align*}
 \end{definition}

\begin{remark}\label{order.rela.}
If $\mathbf{k}=\mathbf{t}=\vec{1}$, it is not difficult to see that ${\mathcal A}_{\eta ,\S,\tau,r}^\mathbf{b,k,t}$ corresponds with ${\mathcal A}_{\eta ,\S,\tau,r}^\mathbf{b}$. 
Regarding some other situations of ${\mathcal A}_{\eta ,\S,\tau,r}^\mathbf{b}$, we have the following:  
\begin{itemize}
    \item When $\tau = \tau_{\ell} = \{j\}$ with $j \in \tau_m$, we define 
    \begin{align*}
        {\mathcal A}_{\eta ,\S,r}^{b_j} (\vec f)(x) := {\mathcal A}_{\eta ,\S,\tau,r}^\mathbf{b} (\vec f)(x) = \sum_{Q \in \mathcal{S}}\mu(Q)^{\an \cdot \frac{1}{r}}\left|b_j(x)-b_{j, Q}\right|\langle | f_j| \rangle_{Q,r} \prod_{i \neq j}\langle | f_i| \rangle_{Q,r} \chi_Q(x).
    \end{align*}
    \item When $\tau_{\ell} = \{j\}$ with $j \in \tau_m$ and $\tau = \phi$, we define
    \begin{align*}
        {\mathcal A}_{\eta ,\S,r}^{\star b_j} (\vec f)(x) := {\mathcal A}_{\eta ,\S,\tau,r}^\mathbf{b} (\vec f)(x) = \sum_{Q \in \mathcal{S}}\mu(Q)^{\an \cdot \frac{1}{r}}\langle |\left(b_j-b_{j, Q}\right) f_j| \rangle_{Q,r} \prod_{i \neq j}\langle | f_i| \rangle_{Q,r} \chi_Q(x) .
    \end{align*}
\end{itemize}
Their linear versions can be given naturally,
\begin{itemize}
    \item The linear version of ${\mathcal A}_{\eta ,\S,r}^{b_j}$,
    \begin{align*}
			A^b_{\eta,\mathcal{S},r} f(x)=\sum_{Q \in \mathcal{S}}\mu(Q)^{\eta \cdot \frac{1}{r}}\left|b(x)-b_Q\right| \langle f \rangle_{Q,r} \chi_Q(x) ,
		\end{align*}
    \item The linear version of ${\mathcal A}_{\eta ,\S,r}^{\star b_j}$,
    \begin{align*}
A^{\star,b}_{\eta,\mathcal{S},r}  f(x)=\sum_{Q \in \mathcal{S}}\mu(Q)^{\eta \cdot \frac{1}{r}}\langle\left|b-b_Q\right| f\rangle_{Q,r} \chi_Q(x),
		\end{align*}
\end{itemize}
Analogously, one can define  the linear version of ${\mathcal A}_{\eta ,\S,\tau,r}^\mathbf{b,k,t}$ as follows,
\begin{align*}
  A_{\eta ,\S,\tau,r}^{b,k,t}f(x) = \sum_{Q\in \mathcal S}\mu(Q)^{\eta \cdot \frac{1}{r}} |b(x)-b_Q|^{k-t}
\langle |b(z)-b_Q|^t |f(z)| \rangle_{Q,r}\chi_Q(x).
\end{align*}
\end{remark}

{

}

	
The sparse bound theory of operators was first introduced by Lerner \cite{Lersmall} in 2016, where he employed the sparse operator \(A_{\S}\) to control Calderón-Zygmund operators. 
	\, \hspace{-20pt}{\bf Theorem $A_1$}(\cite{Lersmall}). {\it\
		Let $T$ be an $\omega$-Calderón-Zygmund operator with $\omega$ satisfying the Dini condition. Then, for every compactly supported $f \in L^1\left(\mathbb{R}^n\right)$, there exists a sparse family $\mathcal{S}$ such that for a.e. $x \in \mathbb{R}^n$,
		$$
		|T f(x)| \leq c_n\left(\|T\|_{L^2 \rightarrow L^2}+C_K+\|\omega\|_{\text {Dini }}\right) A_{\mathcal{S}}|f|(x) .
		$$
	}

Building on this, Lerner \cite{Ler2017} in 2017 extended the sparse domination theory to the commutators of Calderón-Zygmund operators. 

Before proceeding, we recall a fundamental definition. Let $T$ be a linear operator and $b$ a locally integrable function. The commutator of $T$ and $b$, denoted by $[b, T]$, is defined as
\[
[b, T]f(x) = b(x)T(f)(x) - T(bf)(x).
\]
For each integer $m \in \mathbb{N}$, the iterated commutators $T_b^m$ are defined recursively by
\begin{equation}\label{diedai_m}
    T_b^m f = [b, T_b^{m-1}]f \quad \text{with} \quad T_b^1 f = [b, T]f.
\end{equation}

	\, \hspace{-20pt}{\bf Theorem $A_2$}(\cite{Ler2017}). {\it\
		Let $T$ be an $\omega$-Calderón-Zygmund operator with $\omega$ satisfying the Dini condition, and let $b \in L_{\text {loc }}^1$. For every compactly supported $f \in L^{\infty}\left(\mathbb{R}^n\right)$, there exist $3^n$ dyadic lattices $\D^{(j)}$ and $\frac{1}{2.9^n}$-sparse families $\mathcal{S}_j \subseteq \D^{(j)}$ such that for a.e. $x \in \mathbb{R}^n$,
		\begin{align*}
			|[b, T] f(x)| \leq c_n C_T \sum_{j=1}^{3^n}\left(A^b_{0,\eta,\mathcal{S}_j,1} |f|(x)+A^{\star,b}_{0,\eta,\mathcal{S}_j,1}|f|(x)\right) .
		\end{align*}
	}

Further, Li \cite{Li2018} and Cao et al. \cite{Cao2018}  respectively generalized the above results to the multilinear Hörmander type singular integral operator and multilinear Littlewood-Paley square operator in 2017 .

	\, \hspace{-20pt}{\bf Theorem $A_3$}(\cite{Ler2017}). {\it\
		Assume that $T$ is bounded from $L^q \times \cdots \times L^q$ to $L^{q / m, \infty}$ and $\mathcal{M}_T$ is bounded from $L^r \times \cdots \times L^r$ to $L^{r / m, \infty}$, where $1 \leq q \leq r<\infty$. Then, for compactly supported functions $f_i \in L^r\left(\mathbb{R}^n\right), i=1, \ldots, m$, there exists a sparse family $\mathcal{S}$ such that for a.e. $x \in \mathbb{R}^n$,
		
		\begin{align*}
			& \left|T\left(f_1, \ldots, f_m\right)(x)\right| \\
			& \leq c_{n, q, r}\left(\|T\|_{L^q \times \cdots \times L^q \rightarrow L^{q / m, \infty}}+\left\|\mathcal{M}_T\right\|_{L^r \times \cdots \times L^r \rightarrow L^{r / m, \infty}}\right) \\
			& \quad \times \mathcal{A}_{0, \mathcal{S}, r, 1}(\vec f)(x).
		\end{align*}
	}
    
	\, \hspace{-20pt}{\bf Theorem $A_4$}(\cite{Cao2018}). {\it\
		Let $1 \leq r<\infty$ and $1 \leq \ell \leq m$. Assume that $\mathfrak{g}$ is bounded from $L^r \times \cdots \times L^r$ to $L^{r / m, \infty}$ and its kernel $\left\{K_t\right\}$ satisfies the m-linear $L^r$-Hörmander condition. Then, for any compactly supported functions $f_i \in L^r\left(\mathbb{R}^n\right), i=1, \ldots, m$, there exist $3^n$ sparse families $\mathcal{S}_j$ such that for a.e. $x \in \mathbb{R}^n$,
		$$
		\mathfrak{g}_b(\vec{f})(x) \leq c_{m, n, r} \mathfrak{C} \sum_{j=1}^{3^n} \sum_{\tau \subseteq \tau_{\ell}} {\mathcal A}_{0 ,\S,\tau,r}^\mathbf{b} (\vec f)(x).
		$$
	}

Also in 2017, Lerner et al. \cite{Lerner2018} resolved the problem of higher order commutators. They used a complete set of methods to solve the Bloom type estimates for higher order commutators of singular integral operators.

\, \hspace{-20pt}{\bf Theorem $A_5$}(\cite{Lerner2018}). {\it\
Let $\mu,\lambda\in A_{p}$, $1<p<\infty$. Further, let $\nu=\left(\frac{\mu}{\lambda}\right)^{\frac{1}{p}}$ and $m\in {\mathbb N}$.
 If $b\in {\mathrm{BMO}}_{\nu^{1/m}}$, then for every $\omega$-Calder\'on-Zygmund
operator $T$ on ${\mathbb R}^n$ with $\omega$ satisfying the Dini condition,
\begin{equation}\label{itcase}
\|T_{b}^{m}f\|_{L^{p}(\lambda)}\leq c_{n,m,T}\|b\|_{{\mathrm {BMO}}_{\nu^{1/m}}}^{m}\left([\lambda]_{A_{p}}[\mu]_{A_{p}}\right)^{\frac{m+1}{2}\max\left\{ 1,\frac{1}{p-1}\right\} }\|f\|_{L^{p}(\mu)}.
\end{equation}
}

Meanwhile, Rivera-Ríos et al. \cite{Rivera2017} fully extended the results of \cite{Lerner2018} to the fractional case and pioneered the higher order commutators Bloom type estimates.

\, \hspace{-20pt}{\bf Theorem $A_6$}(\cite{Rivera2017} ). {\it\
Let $0<\alpha<n$. Let $m$ be a non-negative integer. For every  $f\in\mathcal{C}_{c}^{\infty}(\mathbb{R}^{n})$ and $b\in L_{\text{loc }}^{m}(\mathbb{R}^{n})$,
there exist a family $\{\mathcal{D}_j\}_{j=1}^{3^{n}}$ of dyadic lattices and a family $\{\mathcal{S}_j\}_{j=1}^{3^n}$ of sparse
families such that $\mathcal{S}_{j} \subseteq \mathcal{D}_{j}$, for each $j$, and 
\[
|(I_{\alpha})_{b}^{m}f(x)|\leq c_{n,m,\alpha}\sum_{j=1}^{3^{n}}\sum_{h=0}^{m}\binom{m}{h}\mathcal{A}_{\alpha,\mathcal{S}_{j}}^{m,h}(b,f)(x),\qquad \text{a.e. } x\in\mathbb{R}^n,
\]
where, for a sparse family $\mathcal{S}$, $\mathcal{A}_{\alpha,\mathcal{S}}^{m,h}(b,\cdot)$ is the sparse operator given by
\[
\mathcal{A}_{\alpha,\mathcal{S}}^{m,h}(b,f)(x)=\sum_{Q\in\mathcal{S}}|b(x)-b_{Q}|^{m-h}|Q|^{\frac{\alpha}{n}}|f(b-b_{Q})^{h}|_{Q}\chi_{Q}(x).
\]
}
	
From 2018 to 2019, Yang et al. \cite{Yang2019} show us how to solve the sparse dominate problem over space of homogeneous type $(X,d,\mu)$. Especially, the relationship between the dyadic cubes of $X$ and the ball is well handled.

\, \hspace{-20pt}{\bf Theorem $A_7$}(\cite{Yang2019} ). {\it\
Let $T$ be the Calder\'on--Zygmund operator as in Definition 1.1 of \cite{Yang2019} and let $b\in L^1_{loc}(X)$. For every  $f\in L^\infty(X)$ with bounded support,
there exist $\mathcal{K}$ dyadic systems $\mathcal D^{\mathfrak{k}}, \mathfrak{k}=1,2,\ldots,\mathcal{K}$ and $\eta$-sparse families $\mathcal S_\mathfrak{k} \subseteq \mathcal D^\mathfrak{k}$ such that
for a.e. $x\in X$,
\begin{align}\label{sparse domination high}
|T_b^m(f)(x)|\leq C \sum_{\mathfrak{k} = 1}^{\mathcal{K}}\sum_{t=0}^k C_k^t A_{0,\S,\tau,1}^{b,k,t}f(x),
\end{align}
where $C_m^k:= {  m! \over (m-k)! \cdot k! }$.
}
	
In 2020, Cao et al. \cite{Cao2020} established the maximal domination and sparse domination for multilinear pseudo--differential operators and their multilinear commutators, and thus gave their quantitative weighted estimates. Let us first recall the following definition:

Let $\sigma$ be a symbol. For $\vec f \in {\mathscr{S}^m}$, $m$-linear pseudodifferential operator $T_\sigma$ is given by
$$
T_\sigma(\vec{f})(x) = \int_{(\mathbb{R}^n)^m} \sigma(x, \vec{\xi})\, e^{2\pi i x \cdot (\xi_1 + \cdots + \xi_m)}\, \widehat{f}_1(\xi_1) \cdots \widehat{f}_m(\xi_m)\, d\vec{\xi},
$$
where the Fourier transform $\widehat{f}$ of a function $f$ is defined by
$$
\widehat{f}(\xi) = \int_{\mathbb{R}^n} e^{-2\pi i x \cdot \xi}\, f(x)\, dx.
$$
Consider the $m$-linear operator $T_\sigma$ and a vector of locally integrable functions $\mathbf{b} = \left(b_1, \ldots, b_m\right)$. We define the $m$-linear commutator of $T_\sigma$ with respect to $\mathbf{b}$ as follows:

For each $j = 1, \ldots, m$, define the $j$-th commutator by
$$
\left[b_j, T_\sigma\right]_j(\vec{f})(x) = b_j(x)\, T_\sigma(\vec{f})(x) - T_\sigma\left(f_1, \ldots, b_j f_j, \ldots, f_m\right)(x).
$$
Then, the $m$-linear commutator is given by
$$
T_{\sigma, \Sigma \mathbf{b}}(\vec{f})(x) = \sum_{j=1}^m \left[b_j, T_\sigma\right]_j(\vec{f})(x).
$$

\, \hspace{-20pt}{\bf Theorem $A_{8}$}(\cite{Cao2020}). {\it\
Let $\sigma \in S_{\rho, \delta}^r(n, m)$ with $0 \leq \rho, \delta \leq 1$, and $r<m n(\rho-1)$. Then, for every compactly supported functions $f_i, i=1, \ldots, m$, there exist $3^n+1$ sparse collections $\mathcal{S}$ and $\left\{\mathcal{S}_i\right\}_{i=1}^{3^n}$ such that
$$
\left|T_\sigma(\vec{f})(x)\right| \lesssim \mathcal{A}_{\mathcal{S}}(\vec{f})(x), \text { a.e. } x \in \mathbb{R}^n,
$$
and
$$
\left|T_{\sigma, \Sigma \mathbf{b}}(\vec{f})(x)\right| \lesssim \sum_{i=1}^{3^n} \sum_{j=1}^m\left({\mathcal A}_{0,\S,1}^{ b_j}(\vec{f})(x)+{\mathcal A}_{0,\S,1}^{\star b_j}(\vec{f})(x)\right), \text { a.e. } x \in \mathbb{R}^n
$$
}

\, \hspace{-20pt}{\bf Theorem $A_{8}$}(\cite{Cao2020}). {\it\
Let $0<\gamma<\min \{\epsilon, 1 / m\}, 0 \leq \rho, \delta \leq 1, r<m n(\rho-1)$, and $\sigma \in$ $S_{\rho, \delta}^r(n, m)$. Then there holds that
$$
M_\gamma^{\sharp}\left(T_{\sigma, \Sigma \mathbf{b}}(\vec{f})\right)(x) \leq C_{\epsilon, \gamma}\|\mathbf{b}\|_{B M O}\left(\sum_{j=1}^m \mathcal{M}_{L(\log L)}^j(\vec{f})(x)+M_\epsilon\left(T_\sigma(\vec{f})\right)(x)\right) .
$$
}

In 2021, Zhang \cite{zhang2023} considered sparse bound for the first order 1-symbol commutator of multilinear maximal singular integral operators with Dini-type kernels, and then gaven their Bloom type estimates.

\, \hspace{-20pt}{\bf Theorem $A_9$}(\cite{zhang2023}). {\it\
Let $T$ be an $\omega$-Calderón-Zygmund operator on $\mathbb{R}^n$ with $\omega$ satisfying the Dini condition, $b \in L_{{loc}}^1$. Then, for every $f \in L^{\infty}\left(\mathbb{R}^n\right)$ with compact support, there exist $3^n$ dyadic lattices $\mathscr{D}^{(j)}$ and $\frac{1}{2 \cdot 9^n}$-sparse families $\mathcal{S}_j \subseteq \mathscr{D}^{(j)}, j 
 =1, \ldots, 3^n$, such that for a.e. $x \in \mathbb{R}^n$, we have
$$
T_{b, k}^* f(x) \leq c_n C_T \sum_{j=0}^{3^n} \sum_{t=0}^kC_k^t A_{0,\S,\tau,1}^{b,k,t}f(x).
$$

Moreover, let $\mu, \lambda \in A_p, 1<p<\infty, \nu=\left(\frac{\mu}{\lambda}\right)^{\frac{1}{p}}$, and $k \in \mathbb{N}$.
If $b \in \mathrm{BMO}_{\nu^{1 / k}}$, and $\omega$ satisfies the Dini condition, then for every $\omega$-Calderón-Zygmund operator $T$ on $\mathbb{R}^n$, we have
$$
\left\|T_{b, k}^* f\right\|_{L^p(\lambda)} \leq c_{n, k, T}\|b\|_{\mathrm{BMO}_{\nu^{1 / k}}}\left([\lambda]_{A_p}[\mu]_{A_p}\right)^{\frac{k+1}{2} \max \left(1, \frac{1}{p-1}\right)}\|f\|_{L^p(\mu)}.
$$
}

In 2022, Moen et al. \cite{Moen2022} bulit first order multi-symbol linear fractional sparse operators and then deduced the sparse bound for commutors of fractional integral.

\, \hspace{-20pt}{\bf Theorem $A_{10}$}(\cite{Moen2022}). {\it\
Let $0<\alpha<n, I_\alpha$ be fractional integral operators and $\tau_m=\{1, \cdots, m\}$. Given ${\bf b}(x)=\left(b_1(x), \cdots, b_m(x)\right)$ with $b_i \in L_{{loc}}^1\left(\mathbb{R}^n\right)$, there exists a constant $C=$ $C(n, \alpha)$ so that for any $f \in L_c^{\infty}\left(\mathbb{R}^n\right)$, there exists $3^n$ sparse families $\mathcal{S}_j$ of dyadic cubes such that
$$
\left|I_{\alpha, b} f(x)\right| \leq C \sum_{j=1}^{3^n} \sum_{\tau \subseteq \tau_m} T_{\mathcal{S}_j, {\bf b}}^{\alpha, \tau}(f)(x),
$$
where
\begin{align*}
    T_{\mathcal{S}, \mathbf{b}}^{\alpha, \tau} f(x)=\sum_{Q \in \mathcal{S}}|Q|^{\frac{\alpha}{n}}\left(\prod_{i \in \tau}\left|b_i(x)-\left(b_i\right)_Q\right| f_{Q_l} \prod_{l \in \tau^c}\left|b_l(y)-\left(b_l\right)_Q\right||f(y)| d y\right) \chi_Q(x).
\end{align*}
}

In 2024,  Lorist et al. \cite{Lorist2024} constructed a specific sparse form for sparse bound, which further broadens the Bloom type estimates for higher order commutators.

\, \hspace{-20pt}{\bf Theorem $A_{11}$}(\cite{Lorist2024}). {\it\
Let $1\le r<s\le\infty, m\in {\mathbb N}$ and let $T$ be a sublinear operator.
Assume that $T$ and ${\mathcal M}^{\#}_{T,s}$ are locally weak $L^r$-bounded. Then there exist $C_{m,n}>1$ and $\la_{m,n}<1$ so that,
for any $f,g\in L^{\infty}_c({\mathbb R}^n)$ and $b\in L^1_{\text{loc}}({\mathbb R}^n)$, there is a $\frac{1}{2\cdot 3^n}$-sparse collection of cubes
${\mathcal S}$ such that
\begin{align*}
\int_{{\mathbb R}^n}|T_b^mf||g|&\le C\Big(\sum_{Q\in {\mathcal S}}\big\langle |b-\langle b\rangle_Q|^{m}|f|\big\rangle_{r,Q}\big\langle|g|\big\rangle_{s',Q}|Q|\\
&\hspace{1cm}+\sum_{Q\in {\mathcal S}}\big\langle |f|\big\rangle_{r,Q}\big\langle |b-\langle b\rangle_Q|^{m}|g|\big\rangle_{s',Q}|Q|\Big),
\end{align*}
where
\begin{equation}\label{constant}
C:=C_{m, n}\left(\varphi_{T, r}\left(\lambda_{m, n}\right)+\varphi_{\mathcal{M}_{T, s}^*, r}\left(\lambda_{m, n}\right)\right).
\end{equation}

}


Inspired by these excellent predecessors, this paper will thoroughly solve the problem of pointwise sparse domination for $\B$-valued multilinear fractional Hörmander type singular integral operators and their generalized commutators over space of homogeneous type $(X,d,\mu)$. That is to say, we will use the {\tt higher order multi-symbol multilinear fractional sparse operators} for sparse bound, followed by quantitative weighted estimates of such sparse operators.

From now on, we default that $(X,d,\mu)$ is space of homogeneous type, $\mathcal D$ is a dyadic lattice of $X$,  $\B$ is a quasi-Banach space, and $\mathcal{T}_{\eta}$ is a {\tt $\B$-valued $m$-linear fractional Hörmander type singular integral operator} and its kernel $Q_\eta \in \mathcal{H}_r$ for some $r \in [1,\infty]$ (see Remark \ref{HSIO} for its definition).
In the proof for our main results, we sometimes use the exponent $p$, which is always defined by the H\"older relation, which means that $\frac{1}{p}: = \sum\limits_{i = 1}^m {\frac{1}{{{p_i}}}}$. Now, we introduce the main results of this paper.


\subsection{Multilinear fractional sparse domination principle}
~~

The multilinear fractional sparse domination principle we have long dreamed of is given as follows.

\begin{theorem}\label{Sparse.to.Cb}
Let $0<{\tilde{r}}<\infty$, $1 \leq r<\infty$, and $1 \leq \ell \leq m$. Assume that $\mathcal{T}_{\eta}$ is bounded from $L^r(X) \times \cdots \times L^r(X)$ to $L^{\tilde{r}, \infty}(X,\B)$ with $\an:=\frac{m}{r}-\frac{1}{\tilde{r}} \in [0,m)$. Then, for any $\vec f \in {\left( {L_b^\infty (X)} \right)^m}$,  there exist $\mathcal{K}$ dyadic lattices $\mathscr{D}^\mathfrak{k}, \mathfrak{k}=1,2, \ldots, \mathcal{K}$  and $\delta$--sparse families $\mathcal{S}_\mathfrak{k} \subseteq \mathscr{D}^\mathfrak{k}$ such that for $\mu$-almost every $x \in X$,
\begin{align}\label{zhang:th1.6}
  \|\mathcal{T}_{\eta,\tau_{\ell}}^{{\bf b, k}}(\vec{f})(x)\|_{\B} \lesssim C_{\mathcal{T}_\eta} \sum_{{\mathfrak{k}}=1}^{\mathcal{K}}\sum_{\tau \subseteq \tau_{\ell}} \sum_{{\bf k, t}} {\mathcal A}_{\eta ,\S_\mathfrak{k},\tau,r}^\mathbf{b,k,t}(\vec f)(x),
  \end{align}
   where $$C_{\mathcal{T}_{\eta}}=\|\mathcal{T}_\eta\|_{L^t \times \cdots \times L^t \rightarrow L^{\tilde{t}, \infty}}+\left\|\mathcal{M}_{\mathcal{T}_{\eta}}\right\|_{L^r \times \cdots \times L^r \rightarrow L^{\tilde{r}, \infty}},$$
   \begin{align*}
    \sum_{{\bf k, t}} :=\sum_{t_{\tau(1)}=0}^{k_{\tau(1)}} \sum_{t_{\tau(2)}=0}^{k_{\tau(2)}} \cdots \sum_{t_{\tau({\ell})}=0}^{k_{\tau({|\tau|})}}\Big(\prod_{i \in \tau_{|\tau|}}C_{k_i}^{t_i}\Big),
   \end{align*}
    and  
    $$C_{k_i}^{t_i}=\frac{t_{i}!}{k_{i}!\cdot\left(t_i-k_i\right)!}.$$
\end{theorem}

Combining Remark~\ref{order.rela.}, we have following result to control $\mathcal{T}_{\eta}$ and $\mathcal{T}_{\eta,\tau_{\ell}}^{\bf b}$.
\begin{corollary}\label{tui1}
    Under the assumption of Theorem \ref{Sparse.to.Cb}, 
    for any $\vec f \in {\left( {L_b^\infty (X)} \right)^m}$,  there exist $\mathcal{K}$ dyadic lattices $\mathscr{D}^\mathfrak{k}, \mathfrak{k}=1,2, \ldots, \mathcal{K}$  and $\delta$--sparse families $\mathcal{S}_\mathfrak{k} \subseteq \mathscr{D}^\mathfrak{k}$ such that for $\mu$-almost every $x \in X$,
\begin{align} \label{Sparse.to.g}
  &\big\|\mathcal{T}_{\eta}(\vec{f})(x)\big\|_{\B} \lesssim C_{\mathcal{T}_\eta} \sum_{Q \in \mathcal{S}} {\mathcal A}_{\eta,\S,r}(\vec f)(x),\\ \label{2.7_1}
&\|\mathcal{T}_{\eta,\tau_{\ell}}^{\bf b}(\vec{f})(x)\|_{\B} \lesssim C_{\mathcal{T}_\eta}\sum_{{\mathfrak{k}}=1}^{\mathcal{K}} \sum_{\tau \subseteq \tau_{\ell}} {\mathcal A}_{\eta ,\S_\mathfrak{k},\tau,r}^\mathbf{b}(\vec f)(x),
  \end{align}
   where $$C_{\mathcal{T}_{\eta}}=\|\mathcal{T}_\eta\|_{L^t \times \cdots \times L^t \rightarrow L^{\tilde{t}, \infty}}+\left\|\mathcal{M}_{\mathcal{T}_{\eta}}\right\|_{L^r \times \cdots \times L^r \rightarrow L^{\tilde{r}, \infty}}.$$
\end{corollary}

\subsection{Multilinear weighted estimates}
~~

To begin with, we present the quantitative weighted estimates of \(\mathcal A_{\eta,\mathcal S,p_0,\gamma}\), as detailed in Theorems \ref{Horm.pro}--\ref{Thm:3}.


\begin{theorem}[$A^{\star}_{\vec{p},q}$ estimate]\label{Horm.pro}
Let $\mathcal D$ be a dyadic lattice of $X$, $\mathcal{S} \subseteq \d$ be a sparse family, $\gamma  > 0$, $0<q<\infty$, $1 < p_i < \infty$, $i=1,\cdots,m$, and $\eta:=\sum\limits_{i = 1}^m {\frac{1}{{{p_i}}}}-\frac{1}{{{q}}} \in [0,m)$
	If $\gamma \geq 1$ and $\vec{\omega} \in A^{\star}_{\vec{p},q}(X)$, then
\begin{align}\label{classical_A}
    \sup_{\mathcal{S} \subseteq \d}\|\mathcal A_{\eta,\mathcal S,p_0,\gamma}(\vec f)\|_{{L^{p_1}(X,\omega_1)}\times \cdots \times {L^{p_m}(X,\omega_m)}\rightarrow L^{q}(X,\omega)} \lesssim \mathcal [\vec{\omega}]_{A^{\star}_{\vec{p},q}(X)}^{\max\{\f{1}{\gamma},\tfrac{p_1'}{q},\dotsc,\tfrac{p_m'}{q}\} \times \frac{1}{p_0}}.
\end{align}
\end{theorem}
The quantitative characterization of $\mathcal A_{\eta,\mathcal S,p_0,\gamma}(\vec f)$ has been previously addressed in the following cases.
\begin{itemize}
    \item When $\eta = 0$, $p_0=1$, $\gamma=1$, the above result is established in \cite{LMS} by Moen et al.
    \item When $\eta = 0$, $p_0=1$, $\gamma \geq 1$ the above result is established in \cite{Hormozi2017} by Hormozi et al.
\end{itemize}
\begin{theorem}[$A^{\star}_{\vec{p},q}-A_\infty$ estimate]\label{Thm:1}
Let $\mathcal D$ be a dyadic lattice of $X$, $\mathcal{S} \subseteq \d$ be a sparse family, $\gamma  > 0$, $0<q<\infty$, $1 \leq p_0 < p_i < \infty$,  $i=1,\cdots,m$, and $\eta:=\sum\limits_{i = 1}^m {\frac{1}{{{p_i}}}}-\frac{1}{{{q}}} \in [0,m)$.
Set $\|u,\vec \sigma\|_{A_{\vec p/p_0,q/p_0}(X)}<\infty$ and $u, \sigma_i \in A_\infty(X)$, for $i=1,\cdots, m$. If $\gamma \ge p_0$, then
\begin{align*}
 &\quad \sup_{\mathcal{S} \subseteq \d} \|\mathcal{A}_{\eta,\mathcal S,p_0,\gamma} (\vec f)\|_{L^{p_1}(X,\omega_1)\times \cdots \times L^{p_m}(X,\omega_m)\rightarrow L^q(X,u)} \\
& \lesssim \|u,\vec \sigma\|_{A^{\star}_{\vec{p}/p_0,q/p_0}(X)}^{\frac 1q}\Big( \prod_{i=1}^m [\sigma_i]_{A_\infty(X)}^{\frac 1{p_i}} +[u]_{A_\infty(X)}^{(\frac 1\gamma-\frac 1q)_+}\sum_{j=1}^{m} \prod_{i\neq j}[\sigma_i]_{A_\infty(X)}^{\frac 1{p_i}} \Big) ,
\end{align*}
where
$
\left(\frac 1\gamma-\frac 1q\right)_+:=\max\left\{\frac 1\gamma-\frac 1q, 0\right\}.
$

If $0<\gamma < p_0$, then the above result still holds for all $q>\gamma $.
\end{theorem}
%
\begin{theorem}[${A^{\star}_{\vec{p},q}}-W_{\vec p,q}^\infty$ estimate]\label{Horm.pro_2}
Let $\mathcal D$ be a dyadic lattice of $X$, $\mathcal{S} \subseteq \d$ be a sparse family, $\gamma  > 0$, $0<q<\infty$, $1 \leq p_0 < p_i < \infty$, $i=1,\cdots,m$, $\eta:=\sum\limits_{i = 1}^m {\frac{1}{{{p_i}}}}-\frac{1}{{{q}}} \in [0,m)$, and $t:=q/\gamma$. Set $u$ and $\vec\sigma$ be weights satisfying that $\|u,\vec \sigma\|_{A^{\star}_{\vec p,q}(X)}<\infty$. If $p_0 \le \gamma < \mathop {\min }\limits_{1 \le i \le m} {p_i}$, then
\begin{align}
  \label{eq:mixedbound}\notag
& \quad \sup_{\mathcal{S} \subseteq \d} \|\mathcal{A}_{\eta,\mathcal S, p_0,\gamma}(\vec f )\|_{L^{p_1}(X,\omega_1)\times \cdots \times L^{p_m}(X,\omega_m)\rightarrow L^q(X,u)} \\
&\le  \|u,\vec \sigma\|_{A^{\star}_{\vec{p}/p_0,q/p_0}(X)}^{1/q}\Big([\vec \sigma]_{W_{\vec p,q}^\infty(X)}^{1/q}+ \sum_{i=1}^m [\vec \sigma^i]_{W_{ (\vec p,q)^i}^\infty(X)}^{1/{\gamma(\frac{p_i}\gamma)'}}\Big),
\end{align}
where $[\vec \sigma^i]_{W^{\infty}_{ (\vec p,q)^i}(X)}=1$ if $q\le \gamma$, and otherwise,
\begin{align*}
[\vec \sigma^i]_{W^{\infty}_{(\vec p,q)^i}(X)}=\sup_{Q\in\d} \Big(  \int_Q M &(\mathbf \chi_Q u)^{\frac{(p_i/\gamma)'}{t'}} \prod_{j\neq i} M (\mathbf \chi_Q\sigma_j)^{\frac{(p_i/\gamma)'}{p_j/\gamma}} d \mu\Big)\Big(\int_Q u^{\frac{(p_i/\gamma)'}{t'}}\prod_{j\neq i} \sigma_j^{\frac{(p_i/\gamma)'}{p_j/\gamma}}d \mu \Big)^{-1}.
\end{align*}
If  $0<\gamma < p_0$, then the above result still holds for all $q>\gamma $.
\end{theorem}
\begin{theorem}[${A^{\star}_{\vec{p},q}}-{H_{\vec{p},q}^\infty}$ estimate]\label{Thm:3}
Let $\mathcal D$ be a dyadic lattice of $X$, $\mathcal{S} \subseteq \d$ be a sparse family, $\gamma  > 0$, $0<q<\infty$, $1 \leq p_0 < p_i < \infty$, $i=1,\cdots,m$, $\eta:=\sum\limits_{i = 1}^m {\frac{1}{{{p_i}}}}-\frac{1}{{{q}}} \in [0,m)$, and $t=q/\gamma$. Set $u$ and $\vec\sigma$ be weights satisfying that $\|u,\vec \sigma\|_{A^{\star}_{\vec p,q}(X)}<\infty$. If $p_0 \le \gamma < \mathop {\min }\limits_{1 \le i \le m} {p_i}$, then
  \begin{align}
    \label{eq:mixedboundH}\notag
  &\quad \sup_{\mathcal{S} \subseteq \d} \|\mathcal{A}_{\eta, \mathcal S,p_0,\gamma}(\vec f )\|_{L^{p_1}(X,\omega_1)\times \cdots \times L^{p_m}(X,\omega_m)\rightarrow L^q(X,u)} \\
  &\le  \|u,\vec \sigma\|_{A^{\star}_{\vec{p}/p_0,q/p_0}(X)}^{1/q}\Big( [\vec{\sigma}]_{H_{\vec{p},q}^\infty(X)}^{\frac{1}{q}} + \sum_{i=1}^m [\sigma_i]_{H_{(\vec{p},q)^i}^\infty(X)}^{\frac{1}{p_i'}} \Big),
  \end{align}
  where $[\vec \sigma^i]_{H_{(\vec p,q)^i}^\infty(X)}=1$ if $q\le \gamma$, and otherwise,
  \begin{equation}\begin{split}
  [\vec \sigma^i]_{H_{(\vec p,q)^i}^\infty(X)}=\sup_{Q\in\d} &\langle u \rangle_{Q}^{p_i'(\frac 1\gamma-\frac 1q)_+} \exp{\left(\dashint_Q \log u^{-1}\right)}^{p_i'(\frac 1\gamma-\frac 1q)_+} \\ &\times\prod_{j\neq i} \langle \sigma_i \rangle_{Q}^{p_i'/p_j} \exp{\left(\dashint_Q \log\sigma_i^{-1}\right)}^{p_i'/p_j}.
  \end{split}\end{equation}
  If  $0<\gamma < p_0$, then the above result still holds for all $q>\gamma $.
  \end{theorem}


  

\begin{theorem}\label{Caopro_1}
Let $\mathcal D$ be a dyadic lattice of $X$, $0<q < \infty$, $1 \leq r < p_i < \infty$, and $\eta:=\sum\limits_{i = 1}^m {\frac{1}{{{p_i}}}}-\frac{1}{{{q}}} \in [0,m)$. Set $u, \sigma_j \in A_\infty(X)$, for each $j \in \tau^c$, and $\omega_i=\sigma_i^{1-p_i / r}$, $i=1, \ldots, m$. If ${\bf b} \in \mathrm{B M O}(X)^m$, then for any $\tau \subseteq \tau_m$ and a sparse family $\mathcal{S}$,
  \begin{align}\label{cao.e0}
    & \left\|{\mathcal A}_{\eta ,\S,\tau,r}^\mathbf{b}\right\|_{\prod\limits_{i=1}^{m} L^{p_i}\left(X,\omega_i\right) \rightarrow L^{q}(X,u)} \lesssim \mathcal{C}_0  \left\|{\mathcal A}_{\eta ,\S}\right\|_{\prod\limits_{i=1}^{m} L^{p_i}\left(X,\omega_i\right) \rightarrow L^{q}(X,u)},
    \end{align}
  where 
  \begin{align*}
    \mathcal{C}_0 = [u]_{A_{\infty}}^{|\tau|} \prod_{j \in \tau^c}[\sigma_j]_{A_{\infty}} \prod_{i=1}^m\|b_i\|_{\mathrm{B M O}(X)}.
  \end{align*}
  \end{theorem}
\begin{remark}
  Under the assumption of Theorem \ref{Caopro_1}, by setting different weight conditions, we can obtain the following estimates:
\begin{itemize}
  \item If $\vec \omega \in A^{\star}_{\vec p,q}(X)$, it follows from Theorem \ref{Horm.pro} that
  \begin{align}\label{cao.e1}
    &\sup_{\mathcal{S} \subseteq \d}\left\|{\mathcal A}_{\eta ,\S,\tau,r}^\mathbf{b}\right\|_{\prod\limits_{i=1}^{m} L^{p_i}\left(X,\omega_i\right) \rightarrow L^{q}(X,u)} \overset{\text{Theorem \ref{Horm.pro}}}{\lesssim} \mathcal{C}_1,
    \end{align}
where 
\begin{align*}
  \mathcal{C}_1= \mathcal{C}_0 \times [\vec{\omega}]_{A^{\star}_{\vec{p} / r,q/r}(X)}^{\max\limits_i\left\{\frac{1}{r}, \frac{1}{q}\left(\frac{p_i}{r}\right)^{\prime}\right\}}.
\end{align*}
\item   Suppose that $u$ and $\vec\sigma$ be weights satisfying that $\|u,\vec\sigma\|_{A^{\star}_{\vec p,q}(X)}<\infty$ and $u, \sigma_i \in A_\infty(X)$, for $i=1 \cdots m$.
Then, we can obtain various quantitative estimates as follows, which are \eqref{cao.e2}, \eqref{cao.e3}, and \eqref{cao.e4}.
\begin{list}{\rm (\theenumi)}{\usecounter{enumi}\leftmargin=1cm \labelwidth=1cm \itemsep=0.2cm \topsep=.2cm \renewcommand{\theenumi}{\roman{enumi}}}
  \item \begin{align}\label{cao.e2}
    &\quad \sup_{\mathcal{S} \subseteq \d}  \left\|{\mathcal A}_{\eta ,\S,\tau,r}^\mathbf{b}\right\|_{\prod\limits_{i=1}^{m} L^{p_i}\left(X,\omega_i\right) \rightarrow L^{q}(X,u)} \overset{\text{Theorem \ref{Thm:1}}}{\lesssim} \mathcal{C}_{2},
    \end{align}
  \begin{align*}
    \mathcal{C}_2= \mathcal{C}_0 \times \|u,\vec \sigma\|_{A^{\star}_{\vec{p},q}(X)}^{\frac 1q}\Big( \prod_{i=1}^m [\sigma_i]_{A_\infty(X)}^{\frac 1{p_i}} +[u]_{A_\infty(X)}^{\max\left\{1-\frac 1q, 0\right\}}\sum_{j=1}^{m} \prod_{i\neq j}[\sigma_i]_{A_\infty(X)}^{\frac 1{p_i}} \Big).
  \end{align*}
    \item \begin{align}\label{cao.e3}
    &\quad \sup_{\mathcal{S} \subseteq \d} \left\|{\mathcal A}_{\eta ,\S,\tau,r}^\mathbf{b}\right\|_{\prod\limits_{i=1}^{m} L^{p_i}\left(X,\omega_i\right) \rightarrow L^{q}(X,u)} \overset{\text{Theorem \ref{Horm.pro_2}}}{\lesssim} \mathcal{C}_{3},
    \end{align}
    where 
    \begin{align*}
      \mathcal{C}_3=\mathcal{C}_0 \times \|u,\vec{\sigma}\|_{A^{\star}_{\vec{p},q}(X)}^{1/q}\Big([\vec \sigma]_{W_{\vec p,q}^\infty(X)}^{1/q}+ \sum_{i=1}^m [\vec \sigma^i]_{W_{ (\vec p,q)^i}^\infty(X)}^{1/{({p_i})'}}\Big),
    \end{align*}
     $[\vec \sigma^i]_{W^{\infty}_{ (\vec p,q)^i}(X)}=1$ if $q\le 1$, and otherwise,
    \begin{align*}
    [\vec \sigma^i]_{W^{\infty}_{(\vec p,q)^i}(X)}=\sup_{Q \in \d} \Big(  \int_Q M &(\mathbf \chi_Q u)^{\frac{(p_i)'}{q'}} \prod_{j\neq i} M (\mathbf \chi_Q\sigma_j)^{\frac{(p_i)'}{p_j}} d \mu\Big) \\ &\times\Big(\int_Q u^{\frac{(p_i)'}{q'}}\prod_{j\neq i} \sigma_j^{\frac{(p_i)'}{p_j}}d \mu \Big)^{-1}.
    \end{align*}
    \item \begin{align}\label{cao.e4}
    &\quad \sup_{\mathcal{S} \subseteq \d} \left\|{\mathcal A}_{\eta ,\S,\tau,r}^\mathbf{b}\right\|_{\prod\limits_{i=1}^{m} L^{p_i}\left(X,\omega_i\right) \rightarrow L^{q}(X,u)} \overset{\text{Theorem \ref{Thm:3}}}{\lesssim} \mathcal{C}_{4},
        \end{align}
  \begin{align*}
  \mathcal{C}_4=\mathcal{C}_0 
  \times \|u,\vec{\sigma}\|_{A^{\star}_{\vec{p},q}(X)}^{1/q}\Big( [\vec{\sigma}]_{H_{\vec{p},q}^\infty(X)}^{\frac{1}{q}} + \sum_{i=1}^m [\sigma_i]_{H_{(\vec{p},q)^i}^\infty(X)}^{\frac{1}{p_i'}} \Big),
\end{align*}
where $[\vec \sigma^i]_{H_{(\vec p,q)^i}^\infty(X)}=1$ if $q\le 1$ and otherwise,
  \begin{equation}\begin{split}
  [\vec \sigma^i]_{H_{(\vec p,q)^i}^\infty(X)}=\sup_{Q \in \d} &\langle u \rangle_{Q}^{p_i'{\max\left\{1-\frac 1q, 0\right\}}} \exp{\left(\dashint_Q \log u^{-1}\right)}^{p_i'{\max\left\{1-\frac 1q, 0\right\}}} \\ &\times\prod_{j\neq i} \langle \sigma_i \rangle_{Q}^{p_i'/p_j} \exp{\left(\dashint_Q \log\sigma_i^{-1}\right)}^{p_i'/p_j}.
  \end{split}\end{equation}
  \end{list}
\end{itemize}
\end{remark}


\subsection{Bloom type estimate for first order multi-symbol multilinear fractional sparse operator}
~~

For Theorem \ref{Caopro_1}, by setting \( r = 1 \), we obtain a new method to estimate, which combines $\mathrm{BMO}(X)$ with Bloom type estimate.
\begin{theorem}\label{est.new.}
  Let $\mathcal D$ be a dyadic lattice of $X$, $\mathcal{S} \subseteq \d$ be a sparse family, $0 < q < \infty$, $1 < p_i < \infty$, $i=1,\cdots,m$, with $\eta:=\sum\limits_{i = 1}^m {\frac{1}{{{p_i}}}}-\frac{1}{{{q}}} \in [0,m)$. If $\nu \in A_q$, \(b_j \in {\mathrm{BMO}}_{\nu_j}(X)\), and $\omega_j \in A_{p_j}(X) $ for each \( j \in \tau^c\), then there exists a sparse family $\tilde{\mathcal{S}}$ satisfying ${\mathcal{S}} \subseteq \tilde{\mathcal{S}} \subseteq \d$ such that
    \begin{align}\label{mix2,0}
      &\quad \left\|{\mathcal A}_{\eta ,\S,\tau,r}^\mathbf{b}\right\|_{\prod\limits_{i=1}^{m} L^{p_i}\left(X,\omega_i\right) \rightarrow L^{q}(X,\nu)} \lesssim \mathcal{N}_{0} \left\|\mathcal{A}_{\eta,\tilde{\mathcal{S}}}\right\|_{\prod\limits_{i=1}^{m} L^{p_i}(X,\varpi_i)\rightarrow L^{q}(X,\nu)},
      \end{align}
where $\varpi_i:= \begin{cases}\omega_i, & i \in \tau \\ \lambda_i, & i \notin \tau\end{cases}$, $\nu_j:=\lambda_j^{-\frac{1}{p_j}} \omega_j^{\frac{1}{p_j}} $ for each $ j \in \tau^c$, and 
\begin{align*}
 {\mathcal{N}_{0}}=   \prod_{i \in \tau}\|b_i\|_{\mathrm{B M O}(X)} \prod_{j \in \tau^c}\|b_j\|_{{\mathrm{BMO}}_{\nu_j}(X)}
\times [\nu]_{A_{\infty}}^{|\tau|} 
\prod_{j \in \tau^c} \left[\omega_j\right]_{A_{p_j}(X)}^{\max \big\{1, \frac{p_j^{\prime}}{p_j}\big\}} .
\end{align*}
\end{theorem}
\begin{remark}\label{mix.remark}
  Under the assumption of Theorem \ref{est.new.}, by setting different weight conditions, we can obtain the following estimates:
\begin{itemize}
  \item If $({({\omega _i})_{i \in \tau }},{({\lambda _j})_{j \in {\tau ^c}}}) \in A^{\star}_{\vec p,q}(X)$ with $\nu = \prod\limits_{i \in \tau}\omega_i^{\frac{q}{p_i}}\prod\limits_{j \in \tau^c}\lambda_j^{\frac{q}{p_j}}$. It follows from Theorem \ref{Horm.pro} that
  \begin{align}\label{nm1}
    &\quad \sup_{\mathcal{S} \subseteq \d} \left\|{\mathcal A}_{\eta ,\S,\tau,r}^\mathbf{b}\right\|_{\prod\limits_{i=1}^{m} L^{p_i}\left(X,\omega_i\right) \rightarrow L^{q}(X,\nu)} \overset{\text{Theorem \ref{Horm.pro}}}{\lesssim} \mathcal{N}_{1},
    \end{align}
    where 
    \begin{align*}
      \mathcal{N}_{1}=\mathcal{N}_{0} \times [((\omega_i)_{i \in \tau},(\lambda_j)_{j \in \tau^c})]_{A^{\star}_{\vec p,q}(X)}^{\max\big\{1,\tfrac{p_1'}{q},\dotsc,\tfrac{p_m'}{q}\big\}}.
    \end{align*}
  \item 
  Let the weights \(\nu\) and \(\vec{\sigma}\) satisfy \( \|\nu, \vec{\sigma}\|_{A^{\star}_{\vec{p}, q}(X)} < \infty \), where $
  \nu = \prod\limits_{i \in \tau} \omega_i^{\frac{q}{p_i}} \prod\limits_{j \in \tau^c} \lambda_j^{\frac{q}{p_j}}$ with $
  \sigma_i = \begin{cases}
  \omega_i^{1 - p_i'}, & \text{if } i \in \tau, \\
  \lambda_i^{1 - p_i'}, & \text{if } i \in \tau^c,
  \end{cases} $.
Suppose that  \( \nu, \sigma_i \in A_\infty(X) \) for all \( i = 1, \dotsc, m \).
  Then, we can obtain various quantitative estimates as follows, which are \eqref{nm2}, \eqref{nm3}, and \eqref{nm4}.
  \begin{list}{\rm (\theenumi)}{\usecounter{enumi}\leftmargin=1cm \labelwidth=1cm \itemsep=0.2cm \topsep=.2cm \renewcommand{\theenumi}{\roman{enumi}}}
  \item \begin{align}\label{nm2}
    &\quad \sup_{\mathcal{S} \subseteq \d} \left\|{\mathcal A}_{\eta ,\S,\tau,r}^\mathbf{b}\right\|_{\prod\limits_{i=1}^{m} L^{p_i}\left(X,\omega_i\right) \rightarrow L^{q}(X,\nu)} \overset{\text{Theorem \ref{Thm:1}}}{\lesssim} \mathcal{N}_{2},
    \end{align}
    where
\begin{align*}
  \mathcal{N}_2&=\mathcal{N}_0 \times \|\nu,\vec \sigma\|_{A^{\star}_{\vec{p},q}(X)}^{\frac 1q}\Big( \prod_{i=1}^m [\sigma_i]_{A_\infty(X)}^{\frac 1{p_i}} +[\nu]_{A_\infty(X)}^{\max\left\{1-\frac 1q, 0\right\}}\sum_{j=1}^{m} \prod_{i\neq j}[\sigma_i]_{A_\infty(X)}^{\frac 1{p_i}} \Big),
    \end{align*}
  \item   \begin{align}\label{nm3}
    &\quad \sup_{\mathcal{S} \subseteq \d} \left\|{\mathcal A}_{\eta ,\S,\tau,r}^\mathbf{b}\right\|_{\prod\limits_{i=1}^{m} L^{p_i}\left(X,\omega_i\right) \rightarrow L^{q}(X,\nu)} \overset{\text{Theorem \ref{Horm.pro_2}}}{\lesssim} \mathcal{N}_{3},
    \end{align}
    where 
    \begin{align*}
      \mathcal{N}_3=\mathcal{N}_0 \times \|\nu,\vec{\sigma}\|_{A^{\star}_{\vec{p},q}(X)}^{1/q}\Big([\vec \sigma]_{W_{\vec p,q}^\infty(X)}^{1/q}+ \sum_{i=1}^m [\vec \sigma^i]_{W_{ (\vec p,q)^i}^\infty(X)}^{1/{({p_i})'}}\Big),
    \end{align*}
     $[\vec \sigma^i]_{W^{\infty}_{ (\vec p,q)^i}(X)}=1$ if $q\le 1$, and otherwise,
    \begin{align*}
    [\vec \sigma^i]_{W^{\infty}_{(\vec p,q)^i}(X)}=\sup_{Q \in \d} \Big(  \int_Q M &(\mathbf \chi_Q \nu)^{\frac{(p_i)'}{q'}} \prod_{j\neq i} M (\mathbf \chi_Q\sigma_j)^{\frac{(p_i)'}{p_j}} d \mu\Big)\Big(\int_Q \nu^{\frac{(p_i)'}{q'}}\prod_{j\neq i} \sigma_j^{\frac{(p_i)'}{p_j}}d \mu \Big)^{-1}.
    \end{align*}
  \item   \begin{align}\label{nm4}
    &\quad \sup_{\mathcal{S} \subseteq \d} \left\|{\mathcal A}_{\eta ,\S,\tau,r}^\mathbf{b}\right\|_{\prod\limits_{i=1}^{m} L^{p_i}\left(X,\omega_i\right) \rightarrow L^{q}(X,\nu)} \overset{\text{Theorem \ref{Thm:3}}}{\lesssim} \mathcal{N}_{4},
      \end{align}
\begin{align*}
  \mathcal{N}_4=\mathcal{N}_0 
  \times \|\nu,\vec{\sigma}\|_{A^{\star}_{\vec{p},q}(X)}^{1/q}\Big( [\vec{\sigma}]_{H_{\vec{p},q}^\infty(X)}^{\frac{1}{q}} + \sum_{i=1}^m [\sigma_i]_{H_{(\vec{p},q)^i}^\infty(X)}^{\frac{1}{p_i'}} \Big),
\end{align*}
where $\|\vec \sigma^i\|_{H_{(\vec p,q)^i}^\infty(X)}=1$ if $q\le 1$ and otherwise,
  \begin{equation}\begin{split}
  [\vec \sigma^i]_{H_{(\vec p,q)^i}^\infty(X)}=\sup_{Q\in\d} &\langle \nu \rangle_{Q}^{p_i'{\max\left\{1-\frac 1q, 0\right\}}} \exp{\left(\dashint_Q \log \nu^{-1}\right)}^{p_i'{\max\left\{1-\frac 1q, 0\right\}}} \\ &\times\prod_{j\neq i} \langle \sigma_i \rangle_{Q}^{p_i'/p_j} \exp{\left(\dashint_Q \log\sigma_i^{-1}\right)}^{p_i'/p_j}.
  \end{split}\end{equation}

  \end{list}
\end{itemize}

\end{remark}

\subsection{Bloom type estimates for higher order multi-symbol multilinear fractional sparse operator}
~~

  Next, we perform the Bloom type estimation of the sparse operators that we focus on in this paper, see Theorem \ref{zhang2018:th1.7} and \ref{path2}. Moreover, we introduce new  characterization based on this, see Remark  \ref{remarkb1} and \ref{remarkb2}.  
\subsection*{\tt Bloom type estimates by maximal weight method }
~~

\begin{theorem}\label{zhang2018:th1.7}
Let $\mathcal D$ be a dyadic lattice of $X$, $\mathcal{S} \subseteq \d$ be a sparse family, \(1 < q < \infty\), \(1 < p_i < \infty\) for each \(i = 1, \ldots, m\), and
$\eta := \sum\limits_{i=1}^m \frac{1}{p_i} - \frac{1}{q} \in [0, m)$.
Set \(\mathbf{t}\) and \(\mathbf{k}\) be multi-indices satisfying \(\mathbf{t} \prec \mathbf{k}\), and
$t := \sum\limits_{i \in \tau} k_i - \sum\limits_{i \in \tau} t_i - 1$.
Suppose that \(\mu_0, \lambda \in A_q(X)\) and \(\mu_i, v_i \in A_{p_i}(X)\), for each  \(i \in \tau\). Define the weights
$\eta_i := \left(\frac{\mu_i}{v_i}\right)^{\frac{1}{t_i p_i}}$ and the maximal weight $\eta_0(x) := \max\limits_{i \in \tau} \{\eta_i(x)\}$.
If \(b_i \in \mathrm{BMO}_{\eta_i}(X)\) for all \(i \in \tau\) and
$\eta_0 = \left(\frac{\mu_0}{\lambda}\right)^{\frac{1}{(t+1) q}}$,
then there exists a sparse family $\tilde{\mathcal{S}}$ satisfying ${\mathcal{S}} \subseteq \tilde{\mathcal{S}} \subseteq \d$ such that
  \begin{align}\label{max.weight_}
    &\quad  \left\|{\mathcal A}_{\eta ,\S,\tau}^\mathbf{b,k,t}\right\|_{\prod\limits_{i=1}^{m} L^{p_i}\left(X,\mu_i\right) \rightarrow L^{q}(X,\lambda)} \lesssim   \mathcal{W}_{0} \left\|\mathcal{A}_{\eta,\tilde{\mathcal{S}}}\right\|_{\prod\limits_{i=1}^{m} L^{p_i}(X,\varpi_i)\rightarrow L^{q}(X,\mu_0)},
    \end{align}
  where $\varpi_i= \begin{cases}v_i, & i \in \tau \\ \mu_i, & i \notin \tau\end{cases}$, and
  \begin{align*}
    \mathcal{W}_{0}
    &=\prod\limits_{i \in \tau} \left(\left([\mu_i]^{\frac{t_i+1}{2}}_{A_{p_i}(X)} \left[v_i\right]^{\frac{t_i-1}{2}}_{A_{p_i}(X)}\right)^{\max \left\{1, \frac{1}{p_i-1}\right\}}\left([\lambda]_{A_{q}(X)} 
    [\mu_0]_{A_{q}(X)}\right)^{\frac{t+1}{2}\max\left\{ 1,\frac{1}{q-1}\right\} }\right)\\
    &\quad \times \prod_{i \in \tau} \|b_i\|_{\mathrm{BMO}_{\eta_0}(X)}^{k_i-t_i} \|b_i\|_{\mathrm{BMO}_{\eta_i}(X)}^{t_i}.\\
  \end{align*}
  \end{theorem}
  \begin{remark}\label{remarkb1}
  Under the assumption of Theorem \ref{zhang2018:th1.7}, by setting different weighted conditions, we can obtain the following estimates:
  \begin{itemize}
    \item If $({({v_i})_{i \in \tau }},{({\mu _j})_{j \in {\tau ^c}}}) \in A^{\star}_{\vec p,q}(X)$ with $\mu_0 = \prod\limits_{i \in \tau} v_i^{\frac{q}{p_i}} \prod\limits_{j \in \tau^c} \mu_j^{\frac{q}{p_j}}$, it follows from Theorem \ref{Horm.pro} that
    \begin{align}\label{bloom1_1}
  \quad \sup_{\mathcal{S} \subseteq \d}\left\|{\mathcal A}_{\eta ,\S,\tau}^\mathbf{b,k,t}\right\|_{\prod\limits_{i=1}^{m} L^{p_i}\left(X,\mu_i\right) \rightarrow L^{q}(X,\lambda)}\overset{\text{Theorem \ref{Horm.pro}}}{\lesssim} {\mathcal{W}_1},
  \end{align}
  where
  \begin{align*}
    {\mathcal{W}_1}&=\mathcal{W}_0 \times [((v_i)_{i \in \tau},(\mu_j)_{j \in \tau^c})]^{\max\limits_{i \in \{1,\ldots,m\}}\{1,\tfrac{p_i'}{q}\}}_{A^{\star}_{\vec p,q}(X)}.
  \end{align*}
  \item  If $\mu_0$ and $\vec\sigma$ be weights satisfying that $\|\mu_0,\vec\sigma\|_{A^{\star}_{\vec p,q}(X)}<\infty$ with $\mu_0 = \prod\limits_{i \in \tau} v_i^{\frac{q}{p_i}} \prod\limits_{j \in \tau'} \mu_j^{\frac{q}{p_j}}$, and $\mu_0, \sigma_i \in A_\infty(X)$ for $i=1 \cdots m$, where $\sigma_i= \begin{cases}v_i^{1-p_i'}, & i \in \tau \\ \mu_i^{1-p_i'}, & i \in \tau^c\end{cases}$.
  
  Then, we can obtain various quantitative estimates as follows, which are \eqref{bloom1_2}, \eqref{bloom1_3}, and \eqref{bloom1_4}.
  \begin{list}{\rm (\theenumi)}{\usecounter{enumi}\leftmargin=1cm \labelwidth=1cm \itemsep=0.2cm \topsep=.2cm \renewcommand{\theenumi}{\roman{enumi}}}
    \item \begin{align}\label{bloom1_2}
      \quad \sup_{\mathcal{S} \subseteq \d} \left\|{\mathcal A}_{\eta ,\S,\tau}^\mathbf{b,k,t}\right\|_{\prod\limits_{i=1}^{m} L^{p_i}\left(X,\mu_i\right) \rightarrow L^{q}(X,\lambda)}\overset{\text{Theorem \ref{Thm:1}}}{\lesssim} {\mathcal{W}_2},
      \end{align}
    where
  \begin{align*}
      {\mathcal{W}_2}
      &={\mathcal{W}_0} \times\|\mu_0,\vec \sigma\|_{A^{\star}_{\vec{p},q}(X)}^{\frac 1q}\left( \prod_{i=1}^m [\sigma_i]_{A_\infty(X)}^{\frac 1{p_i}} +[\mu_0]_{A_\infty(X)}^{{\max\left\{1-\frac 1q, 0\right\}}}\sum_{j=1}^{m} \prod_{i\neq j}[\sigma_i]_{A_\infty(X)}^{\frac 1{p_i}} \right).
  \end{align*}
  \item 
  \begin{align}\label{bloom1_3}
      \quad \sup_{\mathcal{S} \subseteq \d} \left\|{\mathcal A}_{\eta ,\S,\tau}^\mathbf{b,k,t}\right\|_{\prod\limits_{i=1}^{m} L^{p_i}\left(X,\mu_i\right) \rightarrow L^{q}(X,\lambda)}\overset{\text{Theorem \ref{Horm.pro_2}}}{\lesssim} {\mathcal{W}_3},
      \end{align}
  where
  \begin{align*}
      {\mathcal{W}_3}&={\mathcal{W}_0} \times \|\mu_0,\vec{\sigma}\|_{A^{\star}_{\vec p,q}(X)}^{1/q}\left([\vec \sigma]_{W_{\vec p,q}^\infty(X)}^{1/q}+ \sum_{i=1}^m [\vec \sigma^i]_{W_{ (\vec p,q)^i}^\infty(X)}^{1/{(p_i)'}}\right).
  \end{align*}
  and
  $[\vec \sigma^i]_{W_{\vec (\vec p,q)^i}(X)}=1$ if $q\le 1$ and otherwise,
    \begin{align*}
    [\vec \sigma^i]_{W_{\vec (\vec p,q)^i}(X)}=\sup_{Q\in\d} \Big(  \int_Q M &(\mathbf \chi_Q \mu_0)^{\frac{(p_i)'}{q'}} \prod_{j\neq i} M (\mathbf \chi_Q\sigma_j)^{\frac{(p_i)'}{p_j}} d \mu\Big)\Big(\int_Q \mu_0^{\frac{(p_i)'}{q'}}\prod_{j\neq i} \sigma_j^{\frac{(p_i)'}{p_j}}d \mu \Big)^{-1}.
    \end{align*}
    \item   
    \begin{align}\label{bloom1_4}
        \quad \sup_{\mathcal{S} \subseteq \d} \left\|{\mathcal A}_{\eta ,\S,\tau}^\mathbf{b,k,t}\right\|_{\prod\limits_{i=1}^{m} L^{p_i}\left(X,\mu_i\right) \rightarrow L^{q}(X,\lambda)}\overset{\text{Theorem \ref{Thm:3}}}{\lesssim} {\mathcal{W}_4},
        \end{align}
    where
    \begin{align*}
        {\mathcal{W}_4}&=\mathcal{W}_0 \times \|\mu_0,\vec\sigma\|_{{A^{\star}_{\vec p,q}(X)}}^{\frac 1q} \left([\vec \sigma]_{H_{\vec p,q}^\infty(X)}^{1/q}+  \sum_{i=1}^m [\vec \sigma^i]_{H_{(\vec p,q)^i}^\infty(X)}^{1/{p_i'}}\right).
    \end{align*}
    and $[\vec \sigma^i]_{H_{(\vec p,q)^i}^\infty(X)}=1$ if $q\le 1$ and otherwise,
    \begin{equation*}\begin{split}
    [\vec \sigma^i]_{H_{(\vec p,q)^i}^\infty(X)}=\sup_{Q \in \d} &\langle \mu_0 \rangle_{Q}^{p_i'( 1-\frac 1q)_+} \exp{\left(\dashint_Q \log \mu_0^{-1}\right)}^{p_i'( 1-\frac 1q)_+}\prod_{j\neq i} \langle \sigma_i \rangle_{Q}^{p_i'/p_j} \exp{\left(\dashint_Q \log\sigma_i^{-1}\right)}^{p_i'/p_j}.
    \end{split}\end{equation*}
  \end{list} 
  \end{itemize}
  \end{remark}

  \subsection*{\tt Bloom type estimates by iterated weight method}
~~

Distinguished from the method above enables us to obtain a new bloom type estimate, whose technique is called the iterated weight method. 

  \begin{theorem}\label{path2}
Let $\mathcal D$ be a dyadic lattice of $X$, $\mathcal{S} \subseteq \d$ be a sparse family, \(1 < q < \infty\), $1 < p_i < \infty$, $i=1,\ldots,m$, and $\an:=\sum\limits_{i = 1}^m {\frac{1}{{{p_i}}}}-\frac{1}{{{q}}} \in [0,m)$. Set $\mathbf{t}$ and $\mathbf{k}$ are both multi-indexs with \(\mathbf{t} \prec \mathbf{k}\).  Suppose that $\zeta$ is a weight, $\zeta_{\mathbf s}, \vartheta_{\mathbf s} \in A_{p_{\mathbf s}}(X)$,  and \(\lambda, \xi_{\mathbf s} \in A_q(X)\), for every ${\mathbf s} \in \tau:=(\tau(1),\cdots,\tau(|\tau|))$. For $i \in \left\{ {1, \cdots ,\left| \tau  \right|} \right\}$, we define the iterated weights $\eta_{\tau(i)}$ by
$$\eta_{\tau(i)}:= \begin{cases}(\zeta/\xi_{\tau(1)})^{1/q}, & i ={1},\\ (\xi_{\tau(i)}/\xi_{\tau(i+1)})^{{1}/{(k_{\tau(i)} - t_{\tau(i)})q}}, &2 \le i \le |\tau|-1,\\
(\xi_{\tau(|\tau|)}/\lambda)^{1/(k_{\tau(|\tau|)} - t_{\tau(|\tau|)})q}, & i =|\tau|.\\
\end{cases}$$
If \(b_{\mathbf s} \in \mathrm{BMO}_{\eta_{\mathbf s}}(X)\), for each \({\mathbf s} \in \tau\), and \(\eta_{\tau{(|\tau|)}} \in A_q(X)\),  then there exists a sparse family $\tilde{\mathcal{S}}$ satisfying ${\mathcal{S}} \subseteq \tilde{\mathcal{S}} \subseteq \d$ such that
  \begin{align}\label{iter.weight_}
 \left\|{\mathcal A}_{\eta ,\S,\tau}^\mathbf{b,k,t}\right\|_{\prod\limits_{i=1}^{m} L^{p_i}\left(X,\zeta_i\right) \rightarrow L^q\left(X,\lambda\right)}
 \lesssim \mathcal{I}_0 \left\|\mathcal{A}_{\eta,\tilde{\mathcal{S}}}\right\|_{  \prod\limits_{i=1}^{m} L^{p_i}(X,\varpi_i)\rightarrow L^{q}(X,\zeta)},
\end{align}
where $\varpi_i= \begin{cases}\vartheta_i, & i \in \tau \\ \zeta_i, & i \notin \tau\end{cases}$.  We elaborate the composition of $\mathcal{I}_0 = \mathcal{I}_{I} \times \mathcal{I}_{II} \times \mathcal{I}_{III} \times  \prod\limits_{i \in \tau} \left([\zeta_i]^{\frac{t_i + 1}{2}}_{A_{p_i}(X)} \left[\vartheta_i\right]^{\frac{t_i-1}{2}}_{A_{p_i}(X)}\right)^{\max \left\{1, \frac{1}{p_i-1}\right\}}$ in detail as follows.
\begin{align*}
	\mathcal{I}_{I}& :=[\lambda]_{A_q(X)}[\eta_{\tau({|\tau|})}]_{A_q(X)}^{\frac{k_{\tau(|\tau|)} - t_{\tau(|\tau|)} + 1}{2}\max\{1,\frac{1}{q-1}\}},\\
	\mathcal{I}_{II}&:=\prod_{i = 2}^{|\tau|-1}\big([\xi_{\tau(i+1)}]^{\frac{k_{\tau(i)} - t_{\tau(i)} + 1}{2}}_{A_{q}(X)} \left[\xi_{\tau(i)}\right]^{\frac{k_{\tau(i)} - t_{\tau(i)} - 1}{2}}_{A_{q}(X)}\big) ^{\max \left\{1, \frac{1}{q-1}\right\}},\\
\mathcal{I}_{III}&:=\big([\xi_{\tau(2)}]^{\frac{k_{\tau(1)} - t_{\tau(1)} - 2}{2}}_{A_{q}(X)} \left[\xi_{\tau(1)}\right]^{\frac{k_{\tau(1)} - t_{\tau(1)}}{2}}_{A_{q}(X)}\big) ^{\max \left\{1, \frac{1}{q-1}\right\}}.
\end{align*}
\end{theorem}

  \begin{remark}\label{remarkb2}
     Under the assumption of Theorem \ref{path2}, if $\eta_{\bf s} = (\zeta_{\bf s}/\vartheta_{\bf s})^{1/t_{\bf s} p_{\bf s}}$ for each ${\bf s} \in \tau$,
then by setting different weight conditions, we can obtain the following estimates:
  \begin{itemize}
      \item  If $({({\vartheta _i})_{i \in \tau }},{({\zeta _j})_{j \in {\tau ^c}}}) \in A^{\star}_{\vec p,q}(X)$. It follows from Theorem \ref{Horm.pro} that
  \begin{align}\label{bloom2_1}
      \quad \sup_{\mathcal{S} \subseteq \d}  \|{\mathcal A}_{\eta ,\S,\tau}^\mathbf{b,k,t}\|_{\prod\limits_{i=1}^{m} L^{p_i}\left(X,\zeta_i\right) \rightarrow L^q\left(X,\lambda\right)}\overset{\text{Theorem \ref{Horm.pro}}}{\lesssim} {\mathcal{I}_1},
      \end{align}
  where
  \begin{align*}
      {\mathcal{I}_1}
      &= \mathcal{I}_0 \times [((\vartheta_i)_{i \in \tau},(\zeta_j)_{j \in \tau^c})]^{\max\limits_{i \in \tau_{m}}\{1,\tfrac{p_i'}{q}\}}_{A^{\star}_{\vec p,q}(X)}.
  \end{align*}
  \item If $\zeta$ and $\vec\sigma$ be weights satisfying that $\|\zeta,\vec\sigma\|_{A^{\star}_{\vec p,q}(X)}<\infty$ with $\zeta = \prod\limits_{i \in \tau} \vartheta_i^{\frac{q}{p_i}} \prod\limits_{j \in \tau^c} \zeta_j^{\frac{q}{p_j}}$, and $\zeta, \sigma_i \in A_\infty(X)$ for $i=1,\ldots,m$, where $\sigma_i= \begin{cases}\vartheta_i^{1-p_i'}, & i \in \tau \\ \zeta_i^{1-p_i'}, & i \in \tau^c \end{cases}$. 
  Then, we can obtain various quantitative estimates as follows, which are \eqref{bloom2_2}, \eqref{bloom2_3}, and \eqref{bloom2_4}.
  \begin{list}{\rm (\theenumi)}{\usecounter{enumi}\leftmargin=1cm \labelwidth=1cm \itemsep=0.2cm \topsep=.2cm \renewcommand{\theenumi}{\roman{enumi}}}
  \item
  \begin{align}\label{bloom2_2}
  \quad \sup_{\mathcal{S} \subseteq \d}  \|{\mathcal A}_{\eta ,\S,\tau}^\mathbf{b,k,t}\|_{\prod\limits_{i=1}^{m} L^{p_i}\left(X,\zeta_i\right) \rightarrow L^q\left(X,\lambda\right)}\overset{\text{Theorem \ref{Thm:1}}}{\lesssim}{\mathcal{I}_2},
  \end{align}
  where
  \begin{align*}
      {\mathcal{I}_2}=
      \mathcal{I}_0 \times \|\zeta,\vec \sigma\|_{A^{\star}_{\vec{p},q}(X)}^{\frac 1q}
        \Big( \prod_{i \in \tau} [\sigma_i]_{A_\infty(X)}^{\frac 1{p_i}} +[\zeta]_{A_\infty(X)}^{(1  - \frac 1q)_+}\sum_{j=1}^{m} \prod_{i\neq j}[\sigma_i]_{A_\infty(X)}^{\frac 1{p_i}} \Big).
  \end{align*}
  \item \begin{align}\label{bloom2_3}
    \quad \sup_{\mathcal{S} \subseteq \d}  \|{\mathcal A}_{\eta ,\S,\tau}^\mathbf{b,k,t}\|_{\prod\limits_{i=1}^{m} L^{p_i}\left(X,\zeta_i\right) \rightarrow L^q\left(X,\lambda\right)}\overset{\text{Theorem \ref{Horm.pro_2}}}{\lesssim}{\mathcal{I}_3},
    \end{align}
    where
    \begin{align*}
        {\mathcal{I}_3}&= \mathcal{I}_0 \times \|\zeta,\vec{\sigma}\|_{A^{\star}_{\vec p,q}(X)}^{1/q}\Big([\vec \sigma]_{W_{\vec p,q}^\infty(X)}^{1/q}+ \sum_{i=1}^m[\vec \sigma^i]_{W_{ (\vec p,q)^i}^\infty(X)}^{1/{(p_i)'}}\Big).
    \end{align*}
    and
    $[\vec \sigma^i]_{W^{\infty}_{\vec (\vec p,q)^i}(X)}=1$ if $q\le 1$ and otherwise,
      \begin{align*}
      \|\vec \sigma^i\|_{W^{\infty}_{\vec (\vec p,q)^i}(X)}=\sup_{Q \in \d} \Big(  \int_Q M (\mathbf \chi_Q \zeta)^{\frac{(p_i)'}{q'}} \prod_{j\neq i} M (\mathbf \chi_Q\sigma_j)^{\frac{(p_i)'}{p_j}} d \mu\Big) \times\Big(\int_Q \zeta^{\frac{(p_i)'}{q'}}\prod_{j\neq i} \sigma_j^{\frac{(p_i)'}{p_j}}d \mu \Big)^{-1}.
      \end{align*}
  \item \begin{align}\label{bloom2_4}
      \quad \sup_{\mathcal{S} \subseteq \d}  \|{\mathcal A}_{\eta ,\S,\tau}^\mathbf{b,k,t}\|_{\prod\limits_{i=1}^{m} L^{p_i}\left(X,\zeta_i\right) \rightarrow L^q\left(X,\lambda\right)}\overset{\text{Theorem \ref{Thm:3}}}{\lesssim} {\mathcal{I}_4},
    \end{align}
    where
    \begin{align*}
        {\mathcal{I}_4}&=\mathcal{I}_0 \times \|\zeta,\vec\sigma\|_{{A^{\star}_{\vec p,q}(X)}}^{\frac 1q} \Big([\vec \sigma]_{H_{\vec p,q}^\infty(X)}^{1/q}+  \sum_{i=1}^m [\vec \sigma^i]_{H_{(\vec p,q)^i}^\infty(X)}^{1/{p_i'}}\Big).
    \end{align*}
    and $[\vec \sigma^i]_{H_{(\vec p,q)^i}^\infty(X)}=1$ if $q\le 1$ and otherwise,
    \begin{equation*}\begin{split}
    [\vec \sigma^i]_{H_{(\vec p,q)^i}^\infty(X)}=\sup_{Q \in \d}\langle \zeta \rangle_{B}^{p_i'( 1-\frac 1q)_+} \exp{\left(\dashint_Q \log \zeta^{-1}\right)}^{p_i'( 1-\frac 1q)_+}\prod_{j\neq i} \langle \sigma_i \rangle_{B}^{p_i'/p_j} \exp{\left(\dashint_Q \log\sigma_i^{-1}\right)}^{p_i'/p_j}.
    \end{split}\end{equation*}
  \end{list}
  \end{itemize}
  \end{remark}

\subsection{Endpoint quantitative estimates}
~~

We present endpoint quantitative estimates for $\mathcal{T}_{\eta}$  as the last main result.

\begin{theorem}\label{endingBound}
Let $\mathcal D$ be a dyadic lattice of $X$, $\mathcal{S} \subseteq \d$ be a sparse family, $0 < \tilde{r} < \infty$, $1 \leq r < \infty$, $1 \leq \ell \leq m$, and ${\tau _\ell} \subseteq {\tau _m}$. Assume that the operator $\mathcal{T}_{\eta}$ is bounded from $L^r(X) \times \cdots \times L^r(X)$ to $L^{\tilde{r}, \infty}(X,\B)$ with $\eta := \frac{m}{r} - \frac{1}{\tilde{r}} \in [0, m)$. If $\vec{\omega} \in A^{\star}_{\vec{1}, q_0}(X)$, then for any $\lambda > 0$, we have
    \begin{align}\label{eq:LlogL_1}
\sup_{\mathcal{S} \subseteq \d} \omega\left(\left\{ x \in X \;\bigg|\; {\mathcal A}_{\eta,\S,r}(\vec f)(x) > \lambda^m \right\}\right) \lesssim [\vec{\omega}]_{A^{\star}_{\vec{1}, q_0}(X)} \prod_{i=1}^m \left( \int_{X} \frac{|f_i(x)|^r}{\lambda^r} \omega_i \, d\mu \right)^{q_0}.
\end{align}
Furthermore, if ${\bf b} = (b_{1}, \ldots, b_{m}) \in \mathrm{BMO}^{m}(X)$ and $\vec{\omega} \in A^{\star}_{\vec{1}, q_0}(X)$, then for any $\lambda > 0$, the following estimate holds:
\begin{align}\label{eq:LlogL_2}
\sup_{\mathcal{S} \subseteq \d} \omega\left(\left\{ x \in X \;\bigg|\;   {\mathcal A}_{\eta ,\S,\tau,r}^\mathbf{b}(\vec f)(x) > \lambda^m \right\}\right) \lesssim_{\bf b} [\vec{\omega}]_{A^{\star}_{\vec{1}, q_0}(X)} \prod_{i=1}^m \left( \int_{X} \Phi_{r, \ell} \left( \frac{|f_i(x)|}{\lambda} \right) \omega_i \, d\mu \right)^{q_0},
\end{align}
where $\Phi_{r, \ell}(t) = t^r \left(1 + (\log^+ t)^{r \ell}\right)$ and $q_0 := \frac{1}{m - \eta}$. Moreover, the exponent $\ell$ is optimal in the sense that it cannot be replaced by any nonnegative number $k$ with $k < \ell$.
\end{theorem}
Leveraging the sparse bounds Corollary \ref{tui1}, we obtain the following results.
\begin{corollary}
Under the assumption of Theorem \ref{endingBound}.
If $\vec{\omega} \in A^{\star}_{\vec{1}, q_0}(X)$, then for any $\lambda > 0$, we have
\begin{align*}
\omega\left(\left\{ x \in X \;\bigg|\; \|\mathcal{T}_{\eta}(\vec{f})(x)\|_{\B} > \lambda^m \right\}\right) \lesssim [\vec{\omega}]_{A^{\star}_{\vec{1}, q_0}(X)} \prod_{i=1}^m \left( \int_{X} \frac{|f_i(x)|^r}{\lambda^r} \omega_i \, d\mu \right)^{q_0}.
\end{align*}
Furthermore, if ${\bf b} = (b_{1}, \ldots, b_{m}) \in \mathrm{BMO}^{m}(X)$ and $\vec{\omega} \in A^{\star}_{\vec{1}, q_0}(X)$, then for any $\lambda > 0$, the following estimate holds:
\begin{align*}
\omega\left(\left\{ x \in X \;\bigg|\; \big\|\mathcal{T}_{\eta,\tau_{\ell}}^{\bf b}(\vec{f})(x)\big\|_{\B} > \lambda^m \right\}\right) \lesssim_{\bf b} [\vec{\omega}]_{A^{\star}_{\vec{1}, q_0}(X)} \prod_{i=1}^m \left( \int_{X} \Phi_{r, \ell} \left( \frac{|f_i(x)|}{\lambda} \right) \omega_i \, d\mu \right)^{q_0},
\end{align*}
where $\Phi_{r, \ell}(t) = t^r \left(1 + (\log^+ t)^{r \ell}\right)$ and $q_0 := \frac{1}{m - \eta}$. Moreover, the exponent $\ell$ is optimal in the sense that it cannot be replaced by any nonnegative number $k$ with $k < \ell$.
\end{corollary}

\subsection{Applications}
~~

In this subsection, 
We want to apply our main results to some important multilinear fractional operators and their commutators.
Let us start with some definitions.

For $\eta  \in \left[ {0,m} \right)$, every ball $B \subseteq X$, we define the multilinear fractional averaging operator by
$${\A_{\eta,B}}(\vec f)(x): = {\mu(B)^{\eta}}\left( {\prod\limits_{i = 1}^m {{{\left\langle {{f_i}} \right\rangle }_B}} } \right){\chi _B}(x),$$
The multilinear fractional maximal operator $\M_\eta$ on the spaces of homogeneous type is defined as
$$
\M_\eta (\vec{f})(x) : =\sup _{B \subseteq X}{\A_{\eta,B}}(|f_1|,\cdots,|f_m|)(x).
$$
If $X=\rn$, we take $\eta=\frac{\alpha}{n}$ and denote $\M_\eta$ by $\M_\alpha$. If $m=1$, we denote $\M_\eta$ by $M_\eta$. If $\eta=0$, we deonte $\M_\eta$ by $\M$.

\begin{definition}
Let $\eta  \in [0,m)$, $\B$ is a quasi-Banach space, and $B(\cc,\B)$ is the space of all bounded operators from $\cc$ to $\B$. Set an operator-valued function $Q_\eta:(X^{m+1} \backslash \Delta ) \to B(\cc,\B),$ where $\Delta = \{ (x,\vec y) \in X^{m+1} :x = {y_1} =  \cdots  = {y_m}\}$. Suppose that $\T_{\eta}$ is a $\B$-valued $m$-linear fractional singular integral operator. We say $\T_\eta$ is a {\tt $\B$-valued multilinear fractional singular integral operator with Dini type kernel}, if for any $\vec f \in {(L_b^\infty(X))^m}$ and each $x \notin \mathop  \cap \limits_{i = 1}^m {supp}{f_i}$,

	\begin{align*}
\T_{\eta}(\vec{f})(x):=\int_{X^m} Q_{\eta}\left(x, \vec y\right) \left(\prod_{j=1}^{m}f_j\left(y_j\right)\right) d\mu (\vec y),
	\end{align*}
	
where Dini type kernel $Q_\eta$ satisfies that for 

\begin{itemize}
\item (1) ${\left\| {{Q_\eta }\left( {x,\vec y} \right)} \right\|_\B} \lesssim {\left( {\sum\limits_{i = 1}^m {\mu (B(x,d(x,{y_i})))} } \right)^{\eta  - m}}$,\\
		
\item (2)  For each $j \in \left\{ {0, \cdots ,m} \right\}$, whenever $d\left( {{{y_j'}},{y_j}} \right) \le \frac{1}{2}\max \left\{ {d\left( {{y_0},{y_i}} \right):i = 1, \cdots ,m} \right\}$,
\begin{align*}
&{\left\| {{Q_\eta }\left( {{y_0}, \cdots ,{{y_j'}}, \cdots ,{y_m}} \right) - {Q_\eta }\left( {{y_0}, \cdots ,{y_j}, \cdots ,{y_m}} \right)} \right\|_\B}, \\
\lesssim & {\left( {\sum\limits_{i = 1}^m {\mu (B({y_0},d({y_0},{y_i})))} } \right)^{\eta  - m}}w\left( {\frac{{d\left( {{{y'}_j},{y_j}} \right)}}{{\sum\limits_{i = 1}^m {d\left( {{y_0},{y_j}} \right)} }}} \right),
\end{align*}
where $w$ is increasing, $w(0)=0$, and ${\left[ w \right]_{Dini}} = \int_0^1 {\frac{{w\left( t \right)}}{t}dt}  < \infty$.
\end{itemize}

If there exist some points $\left\{ {{s_1}, \cdots ,{s_m},{\tilde s}} \right\}$ with $\eta : = \sum\limits_{i = 1}^m {\frac{1}{{{s_i}}} - \frac{1}{{\tilde s}}} \in [0,m)$ such that $T_\eta$ is bounded from $\prod\limits_{i = 1}^m {{L^{{s_i}}}(X)}$ to ${L^{\tilde s,\infty }}(X,\B)$, then we say that $\T_\eta$ is a {\tt $\B$-valued multilinear fractional Dini type Calder\'on-Zygmund operator}.
\end{definition}

It is worth noting that {\tt multilinear fractional integral operator} $\I_\eta$ is a special case of {\tt $\B$-valued multilinear fractional Dini type Calder\'on-Zygmund operator}, which is defined by 
\begin{align*}
	{\I_\eta }(\vec f)(x) = \int_{{X^m}} {{{\left( {\sum\limits_{i = 1}^m {\mu (B(x,d(x,{y_i})))} } \right)}^{\eta  - m}}\prod\limits_{i = 1}^m {{f_i}({y_i})} } d\mu (\vec y),
\end{align*}
where $d\mu (\vec y) = d\mu ({y_1}) \cdots d\mu ({y_m})$.

This fact that 
$${\M_\eta }(\vec f)(x) \le {m^{\eta  - m}}{\I_\eta }(|\vec f|)(x).$$
is also valid under the space of homogeneous type setting, which follows from that
\[{\I_\eta }(|\vec f|)(x) \ge {m^{\eta  - m}}\mu {(B(x,r))^{\eta  - m}}\int_{d(x,{y_1}) \le r} { \cdots \int_{d(x,{y_m}) \le r} {\prod\limits_{i = 1}^m {|{f_i}({y_i})|} d\mu (\vec y)} }.\]

Similar to the proof of \cite[Proposition 3.1]{Li2018_s}, by the corresponding method, we can obtain that 

\begin{proposition}\label{app.1}
Under the same assumption of above, if {\tt $\T_\eta$ is a $\B$-valued multilinear fractional singular integral operator with Dini type kernel}, then $\T_\eta$ is a {\tt $\B$-valued multilinear fractional Hörmander type singular integral operator}.
\end{proposition}

It follows immediately from  Proposition \ref{app.1} and Theorems \ref{Sparse.to.Cb}, \ref{endingBound} that 

\begin{theorem}\label{app.2}
Under the same assumption of Theorem \ref{Sparse.to.Cb} and Theorem \ref{endingBound} respectively, the {\tt $\B$-valued multilinear fractional Dini type Calder\'on-Zygmund operator} $\T_\eta$,  {\tt multilinear fractional integral operator} $\I_\eta$, {\tt multilinear fractional maximal operator $\M_\eta$}, and their gerneralized commutators can enjoy the  multilinear fractional sparse domination principle: Theorem \ref{Sparse.to.Cb} and endpoint estimates: Theorem \ref{endingBound}.
\end{theorem}

\begin{remark}
(1) By Theorem \ref{app.2}, we can get the domination estimates and endpoint estimates for {\tt $\B$-valued multilinear fractional Dini type Calder\'on-Zygmund operator} $\T_\eta$. It is worth mentioning that the Hörmander type parameter $r$ can be given in $[1,\infty)$ arbitrarily, since the multilinear fractional  Dini type Calder\'on-Zygmund theory has been set up in \cite{ZhangWu2023}.\\
(2) If we take $B$ as a specific $L^2(\Omega,dv)$, such as $L^2(\R_ + ^{n + 1},\frac{{dzdt}}{{{t^{n + 1}}}})$,  then we can obtain the new multilinear fractional Littlewood-Paley sparse domination theory, which straightway extend the results by Cao et al. in \cite{Cao2018}. So how to select these $L^2(\Omega,dv)$? Ones may refer to \cite{Cao2018, Si2021, Xue2021,Xue2015}.
\end{remark}

\subsection{Organization}
~~

The structure of rest of the paper is as follows. In Sect. \ref{Pre.}, we build and give some fundamental lemmas, which play a important roles in our proof. In Sect. \ref{SDP}, the multilinear fractional sparse domination principle will be proven. Finally, we will give all of the proofs of quantitative weighted estimates in Sect. \ref{weight.estimate}.

\section{\bf  Preliminaries }\label{Pre.}


\subsection{Dyadic lattices}
~~
In $(X, d, \mu)$, a countable family $\d:=\bigcup\limits_{k \in \mathbb{Z}} \d_k$, where $\d_k:=\left\{Q_\alpha^k: \alpha \in \mathscr{A}_k\right\}$, is called a system of dyadic cubes with parameters $\delta \in(0,1)$ and $0<a_1 \leq A_1<\infty$ if it satisfies the following properties:
\begin{itemize}
  \item For each integer \( k \), the set \( X \) can be written as a disjoint union of the cubes \( Q_\alpha^k \) indexed by \( \alpha \in \mathscr{A}_k \):
$
  X = \bigcup\limits_{\alpha \in \mathscr{A}_k} Q_\alpha^k$;
  \item If $\ell \geq k$, then either $Q_\beta^{\ell} \subseteq Q_\alpha^k$ or $Q_\alpha^k \cap Q_\beta^{\ell} = \emptyset$;
  \item For each $(k, \alpha)$ and for each $\ell \leq k$, there exists a unique $\beta$ such that $Q_\alpha^k \subseteq Q_\beta^{\ell}$;
  \item For each $(k, \alpha)$, there exists at most $M$ indices $\beta$ such that $Q_\beta^{k+1} \subseteq Q_\alpha^k$, and$
  Q_\alpha^k = \bigcup\limits_{Q \in \d_{k+1}, Q \subseteq Q_\alpha^k} Q$;
  \item
\begin{equation}\label{eq:contain}
  B\left(x_\alpha^k, a_1 \delta^k\right) \subseteq Q_\alpha^k \subseteq B\left(x_\alpha^k, A_1 \delta^k\right)=: B\left(Q_\alpha^k\right);
\end{equation}
\item 
If $\ell \geq k$ and $Q_\beta^{\ell} \subseteq Q_\alpha^k$, then $B\left(Q_\beta^{\ell}\right) \subseteq B\left(Q_\alpha^k\right)$.
\end{itemize}
The set $Q_\alpha^k$ is a dyadic cube of generation $k$, centered at point $x_\alpha^k \in Q_\alpha^k$ with sidelength $\delta^k$.

From the properties of the dyadic system and doubling measure, we conclude that there exists a constant \( C_{\mu, 0} \), depending only on \( C_\mu \) as defined in \eqref{def_hom} and the earlier constants \( a_1, A_1 \), such that for any cubes \( Q_\alpha^k \) and \( Q_\beta^{k+1} \) with \( Q_\beta^{k+1} \subseteq Q_\alpha^k \),
\begin{align}\label{Cmu0}
  \mu(Q^{k+1}_\beta) \leq \mu(Q^k_\alpha) \leq C_{\mu,0} \mu(Q^{k+1}_\beta).
\end{align}
\begin{proposition}
    [\cite{Yang2019}]\label{theorem dyadic cubes}
On the metric measure space $(X, d, \mu)$, there exists a dyadic lattices  with parameters $0 < \delta \leq (12A_0^3)^{-1}$, $a_1 = (3A_0^2)^{-1}$, and $A_1 = 2A_0$. This construction relies on a fixed countable set of center points $x^k_\alpha$ satisfying
\[
d(x^k_\alpha, x^k_\beta) \geq \delta^k \quad \text{for } \alpha \neq \beta,
\]
\[
\min_\alpha d(x, x^k_\alpha) < \delta^k \quad \forall x \in X,
\]
and a partial order $\leq$ on index pairs $(k,\alpha)$. Furthermore, the system can be constructed such that
\[
\overline{Q}^k_\alpha = \overline{\{x^\ell_\beta \mid (\ell,\beta) \leq (k,\alpha)\}},
\]
\[
\widetilde{Q}^k_\alpha = \operatorname{int} \overline{Q}^k_\alpha = \left( \bigcup_{\gamma \neq \alpha} \overline{Q}^k_\gamma \right)^c,
\]
\[
\widetilde{Q}^k_\alpha \subseteq Q^k_\alpha \subseteq \overline{Q}^k_\alpha,
\]
where $Q^k_\alpha$ are derived from the closed sets $\overline{Q}^k_\alpha$ and the open sets $\widetilde{Q}^k_\alpha$ through finitely many set operations.
\end{proposition}

We also recall the following remark from \cite[Section 2.3]{KLPW}.
\begin{remark}\label{p2.3}
For a fixed point $x_0 \in X$ and each $k \in \mathbb{Z}$, there exists an index $\alpha$ such that $x_0 = x_\alpha^k$, the center of $Q_\alpha^k \in \d_k$.
\end{remark}
\begin{definition}\label{MWbmo}
Let $\omega \in A_\infty$. A locally integrable function $b \in L^1_{\text{loc}}(X)$ belongs to ${ BMO}_\omega(X)$ provided that
\[
\|b\|_{{ BMO}_\omega(X)} := \sup_{B} \frac{1}{\omega(B)} \int_B |b(x) - b_B| \, d\mu(x) < \infty,
\]
where
\[
b_B := \frac{1}{\mu(B)} \int_B b(x) \, d\mu(x)
\]
and the supremum is taken over all balls $B \subseteq X$.
  \end{definition}
\subsection{Adjacent Systems of Dyadic Cubes}
~~

For a space of homogeneous type $(X, d, \mu)$, a finite collection $\{\mathscr{D}^\mathfrak{k}\}_{\mathfrak{k}=1}^{\mathcal{K}}$ of dyadic families is called a \textit{collection of adjacent dyadic systems} with parameters $\delta \in (0,1)$, $0 < a_1 \leq A_1 < \infty$, and $1 \leq C_{adj} < \infty$ provided that:

\begin{list}{\rm (\theenumi)}{\usecounter{enumi}\leftmargin=1cm \labelwidth=1cm \itemsep=0.2cm \topsep=.2cm \renewcommand{\theenumi}{\roman{enumi}}}
  \item \textbf{Individual Systems:} Each $\mathscr{D}^\mathfrak{k}$ is a dyadic system with parameters $\delta$, $a_1$, and $A_1$.
  \item \textbf{Adjacency Condition:}
  For every ball $B(x, r) \subseteq X$ with $\delta^{k+3} < r \leq \delta^{k+2}$ for some $k \in \mathbb{Z}$, there exists $\mathfrak{k} \in \{1, 2, \ldots, \mathcal{K}\}$ and a cube $Q \in \mathscr{D}^\mathfrak{k}$ of generation $k$ centered at ${}^\mathfrak{k} x^k_\alpha$ such that
  \[
  d(x, {}^\mathfrak{k} x^k_\alpha) < 2A_0 \delta^{k}
  \]
  and
  \begin{equation}\label{eq:ball;included}
    B(x,r)\subseteq Q\subseteq B(x,C_{adj}r).
\end{equation}
\end{list}

From \cite{HK}, we recall the following construction.

\begin{proposition}\label{thm:existence2}
    Let $(X, d, \mu)$ be a space of homogeneous type. Then there exists a finite collection $\{\mathscr{D}^\mathfrak{k}\}_{t=1}^{\mathcal{K}}$ of adjacent dyadic systems with parameters $\delta \in \left(0, \frac{1}{96 A_0^6}\right)$, $a_1 = \frac{1}{12 A_0^4}$, $A_1 = 4 A_0^2$, and $C = 8 A_0^3 \delta^{-3}$. The center points ${}^\mathfrak{k}x^k_\alpha$ of cubes $Q \in \mathscr{D}^\mathfrak{k}_k$ satisfy, for each $\mathfrak{k} \in \{1, 2, \ldots, \mathcal{K}\}$,
\[
d({}^\mathfrak{k}x^k_\alpha, {}^\mathfrak{k}x^k_\beta) \geq \frac{\delta^k}{4 A_0^2} \quad (\alpha \neq \beta)
\]
and
\[
\min_{\alpha} d(x, {}^\mathfrak{k}x^k_\alpha) < 2 A_0 \delta^k \quad \text{for all } x \in X.
\]
Moreover, these adjacent systems can be constructed so that each $\mathscr{D}^\mathfrak{k}$ satisfies the distinguished center point property  in Remark \ref{p2.3}.
\end{proposition}

We recall from \cite[Remark 2.8]{KLPW} that the number $\mathcal{K}$ of adjacent dyadic systems satisfies
\[
\mathcal{K} = \mathcal{K}(A_0, \widetilde{A}_1, \delta) \leq \widetilde{A}_1^6 \left(\frac{A_0^4}{\delta}\right)^{\log_2 \widetilde{A}_1},
\]
where $\widetilde{A}_1$ is the geometrically doubling constant (see \cite[Section 2]{KLPW}).

\begin{definition}\label{D:sparse}
  For $0 < \delta < 1$, a collection $\mathcal{S} \subseteq \mathcal{D}$ of dyadic cubes is \textit{$\delta$-sparse} if for each $Q \in \mathcal{S}$, there exists a measurable subset $E_Q \subseteq Q$ with $\mu(E_Q) \geq \delta \mu(Q)$, and the family $\{E_Q\}_{Q \in \mathcal{S}}$ are pairwise disjoint.
\end{definition}

Let  $\mathcal D$ be a dyadic lattice of $X$, the dyadic maximal function is defined by
\[
  M_{\mathcal{D}}^{\sigma}(f) = \sup_{Q \in \mathcal{D}} \frac{1}{\sigma(Q)} \int_Q |f(x)| \sigma \, d\mu.
\]
It is well-known that
\begin{equation}\label{dyadicmaximal}
  \|M_{\mathcal{D}}^{\sigma} f\|_{L^p(X,\sigma)} \leq p' \|f\|_{L^p(X,\sigma)}, \quad 1 < p < \infty.
\end{equation}

\begin{lemma}[\cite{Yang2019}]\label{zhang:6.1}
  Let $0 < \gamma < 1$, $\mathcal{D}$ be a dyadic lattice in $X$, and $\mathcal{S} \subseteq \mathcal{D}$ a $\gamma$-sparse family. For $b \in L^1_{\text{loc}}(X)$, there exists a $\frac{\gamma}{2(\gamma+1)}$-sparse family $\tilde{\mathcal{S}} \subseteq \mathcal{D}$ containing $\mathcal{S}$ such that for each $Q \in \tilde{\mathcal{S}}$,
  \[
    |b(x) - b_Q| \leq C \sum_{\substack{R \in \tilde{\mathcal{S}} \\ R \subseteq Q}} \langle\left|b(x)-b_R\right| \rangle_{R} \chi_R(x) \quad \text{a.e. } x \in Q.
  \]
\end{lemma}

\section{\bf Proof for multilinear fractional sparse domination principle}\label{SDP}
In this section, we will illustrate the sparse domination theorem by introducing the grand maximal truncated operator.

We define the grand maximal truncated operator \(\mathcal{M}_{\mathcal{T}_{\eta}}\) as
\begin{align*}
\mathcal{M}_{\mathcal{T}_{\eta}}(\vec{f})(x) := \sup_{B \ni x} \esssup_{\xi \in B} \big\| \mathcal{T}_{\eta}\big(\vec{f} \chi_{X \setminus C_{\widetilde{j}_0} B}\big)(\xi) \big\|_{\B},
\end{align*}
where the supremum is over all balls \(B \subseteq X\) containing \(x\). The integer \(\widetilde{j}_0\) is the smallest integer satisfying
\begin{align}\label{widetildej0}
2^{\widetilde{j}_0} > \max\{3A_0,\, 2A_0 \cdot C_{{adj}}\},
\end{align}
and we set \(C_{\widetilde{j}_0} := 2^{\widetilde{j}_0 + 2} A_0\). Here, \(C_{{adj}}\) is a constant mentioned in preliminaries.

Given a ball \(B_0 \subseteq X\) and a point \(x \in B_0\), we define the local grand maximal truncated operator \(\mathcal{M}_{\mathcal{T}_{\eta}, B_0}\) by
\begin{align*}
\mathcal{M}_{\mathcal{T}_{\eta}, B_0}(\vec{f})(x) := \sup_{B \ni x,\, B \subseteq B_0} \esssup_{\xi \in B} \big\| T\big(\vec{f} \chi_{C_{\widetilde{j}_0} B_0 \setminus C_{\widetilde{j}_0} B}\big)(\xi) \big\|_{\B}.
\end{align*}

Then, we claim two Lerner type lemmas that are critical tools for the proof in this section involving the Calderón-Zygmund decomposition.

  \begin{lemma}\label{Condi.Mg}
Let $1 \leq r<\infty$ and $\frac{m}{r}-\frac{1}{\tilde{r}}=\an \in [0,m)$. Under the assumption of Theorem \ref{Sparse.to.Cb}, there holds that
    \begin{align}\label{Condi.Mg_1}
      \left\|\mathcal{M}_{\mathcal{T}_{\eta}}\right\|_{L^r \times \cdots \times L^r \rightarrow L^{\tilde{r}, \infty}} \lesssim 1.
    \end{align}
    \end{lemma}
Since the proof of this lemma closely follows \cite[Lemma 4.1]{Cao2018}, we will omit it here and proceed to the formal proof.

\begin{lemma}\label{Lemma:preLi2018}
Let $1 \leq r<\infty$, $1 \leq t<\infty$, and $\frac{m}{t}-\frac{1}{\tilde{t}}=\an \in [0,m)$. Suppose that $\mathcal{T}_\eta$ be bounded from $L^t(X) \times \cdots \times L^t(X) \rightarrow L^{\tilde{t}, \infty}(X,\B)$. Then for a.e. $x \in B_0$,
  \begin{align}\label{theoremA_11}
    \big\|\mathcal{T}_{\eta}(\vec{f} \cdot \chi_{\cjo B_0})(x)\big\|_{\B} \lesssim \|\mathcal{T}_{\eta}\|_{L^t \times \cdots \times L^t \rightarrow L^{\tilde{t}, \infty}} \mu(\cjo B_0)^{\an \cdot \frac{1}{r}}\prod_{i=1}^m\big|f_i(x)\big|+\mathcal{M}_{\mathcal{T}_{\eta}, B_0}(\vec{f})(x).
  \end{align}
  \end{lemma}
  \begin{proof}
    Without loss generality, we only need to prove the case $m=2$. That is,
    \begin{align*}
      & \big\|\mathcal{T}_{\eta}\big(f_1 \chi_{C_{\widetilde j_0}B_0}, f_2 \chi_{C_{\widetilde j_0}B_0}\big)(x)\big\|_{\B} \\
      & \quad \leq \|\mathcal{T}_{\eta}\|_{L^t \times L^t \rightarrow L^{\tilde{t}, \infty}}\mu(C_{\widetilde j_0}B_0)^{\an \cdot \frac{1}{r}}\prod_{i=1}^{2}\big|f_i(x)\big|+\mathcal{M}_{\mathcal{T}_{\eta}, B_0}\big(f_1, f_2\big)(x).
  \end{align*}
  
  Suppose that $x$ be a point of approximate continuity of 
  $\mathcal{T}_{\eta}\big(f_1 \chi_{C_{\widetilde j_0}B_0}, f_2 \chi_{C_{\widetilde j_0}B_0}\big)$, and $x \in \operatorname{int} B_0$.
  Then for every $\epsilon>0$, the sets
  \begin{align*}
    &E_S(x):=\quad \big\{y \in B(x, s):\big\|\mathcal{T}_{\eta}\big(f_1 \chi_{C_{\widetilde j_0}B_0}, f_2 \chi_{C_{\widetilde j_0}B_0}\big)(y)-\mathcal{T}_{\eta}\big(f_1 \chi_{C_{\widetilde j_0}B_0}, f_2 \chi_{C_{\widetilde j_0}B_0}\big)(x)\big\|_{\B}<\epsilon\big\}.
  \end{align*}
  satisfy $\lim\limits_{s \rightarrow 0} \frac{\mu(E_s(x))}{\mu(B(x, s))}=1$. 
  
  Let $s>0$ be so small that $B(x, s) \subseteq B_0$.
   Then for a.e. $y \in E_S(x)$,
   \begin{align*}
    \big\|\mathcal{T}_{\eta}\big(f_1 \chi_{C_{\widetilde j_0}B_0}, f_2 \chi_{C_{\widetilde j_0}B_0}\big)(x)\big\|_{\B} &<\big\|\mathcal{T}_{\eta}\big(f_1 \chi_{C_{\widetilde j_0}B_0}, f_2 \chi_{C_{\widetilde j_0}B_0}\big)(y)\big\|_{\B}+\epsilon \\
    & \leq\big\|\mathcal{T}_{\eta}\big(f_1 \chi_{C_{\widetilde j_0}B(x, s)}, f_2 \chi_{C_{\widetilde j_0}B(x, s)}\big)(y)\big\|_{\B} \\
    &\quad+\mathcal{M}_{\mathcal{T}_{\eta}, B_0}(f_1, f_2)(x)+\epsilon.
    \end{align*}  
  It follows that
  \begin{align*}
    & \big\|\mathcal{T}_{\eta}\big(f_1 \chi_{C_{\widetilde j_0}B_0}, f_2 \chi_{C_{\widetilde j_0}B_0}\big)(x)\big\|_{\B} \\
    & \leq \underset{y \in E_S(x)}{\operatorname{essinf}}\big\|\mathcal{T}_{\eta}\big(f_1 \chi_{C_{\widetilde j_0}B(x, s)}, f_2 \chi_{C_{\widetilde j_0}B(x, s)}\big)(y)\big\|_{\B} \\
    & \quad+\mathcal{M}_{\mathcal{T}_{\eta}, B_0}\big(f_1, f_2\big)(x)+\epsilon \\
    & \leq \mu(E_S(x))^{-\frac{1}{\tilde{t}}}\big\|\mathcal{T}_{\eta}\big(f_1 \chi_{C_{\widetilde j_0}B(x, s)}, f_2 \chi_{C_{\widetilde j_0}B(x, s)}\big)\big\|_{L^{\tilde{t}, \infty}(\B)} \\
    & \quad+\mathcal{M}_{\mathcal{T}_{\eta}, B_0}\big(f_1, f_2\big)(x)+\epsilon \\
    & \leq\|\mathcal{T}_{\eta}\|_{L^t \times L^t \rightarrow L^{\tilde{t}, \infty}} \frac{1}{\mu(E_S(x))^{2 /{t}}} \cdot \mu(E_S(x))^{\an } \prod_{i=1}^2\bigg(\int_{C_{\widetilde j_0}B(x, s)}\big|f_i(y)\big|^t \mathrm{~d} y\bigg)^{\frac{1}{{t}}} \\
    & \quad+\mathcal{M}_{\mathcal{T}_{\eta}, B_0}\big(f_1, f_2\big)(x)+\epsilon \\
    \intertext{Assuming additionally that $x$ is a Lebesgue point of $\big|f_1\big|^t$ and $\big|f_2\big|^t$ and letting subsequently $s \rightarrow 0$ and $\epsilon \rightarrow 0$. Since $\mu(E_S(x))^{\an } \rightarrow 0 \,\, (s \rightarrow 0)$, the above}
    & \lesssim \|\mathcal{T}_{\eta}\|_{L^t \times L^t \rightarrow L^{\tilde{t}, \infty}}\mu(\bo)^{\an \cdot \frac{1}{r}}\prod_{i=1}^{2}\big|f_i(x)\big|+\mathcal{M}_{\mathcal{T}_{\eta}, B_0}\big(f_1, f_2\big)(x).
    \end{align*}  
  Then we obtain our desired.
  \end{proof}

\begin{proof}[\bf Proof of Theorem \ref{Sparse.to.Cb}]
~

Without loss of generality, we consider $m=3$. We will apply the methods from \cite[Theorem 1.1]{Li2018_s} and \cite[Theorem 3.1]{Lersmall} to establish \eqref{zhang:th1.6}.

\item[\textbullet] \textbf{Step 1} Firstly, we point out that starting from any fixed ball $B_0$, we can find a way to cover the entire homogeneous space $X$.

We define the annuli $U_j:= 2^{j+1}B_0\backslash 2^j B_0$, $j\geq0$ and we choose $j_0$ to be the smallest integer such that
\begin{align}\label{j0}
j_0>\widetilde j_0\quad{\rm and}\quad 2^{j_0}> 4A_0.
\end{align}
Next, for each $U_j$, we choose the balls 
\begin{align}\label{Bjl}
\{ \widetilde B_{j,\ell} \}_{\ell=1}^{L_j}
\end{align}
 centred in $U_j$ and with radius $2^{j-\widetilde j_0}r$ to cover $U_j$. From the geometric doubling property \cite[p. 67]{CW1}, it is direct to see that 
\begin{align}\label{CA0mu}
 \sup _j L_j \leq C_{A_0,\mu,\widetilde j_0},
\end{align}  
 where $C_{A_0,\mu,\widetilde j_0}$ is an absolute constant depending only on $A_0$, $\widetilde j_0$ and $C_\mu$.

We now examine the properties of the balls \( \widetilde{B}_{j,\ell} = B(x_{j,\ell}, 2^{j - \widetilde{j}_0} r) \) and their enlarged versions \( C_{{adj}} \widetilde{B}_{j,\ell} = B(x_{j,\ell}, C_{{adj}} 2^{j - \widetilde{j}_0} r) \), where \( C_{{adj}} \) is a constant. We claim that
\begin{align}\label{claim1}
C_{{adj}} \widetilde{B}_{j,\ell} \cap U_{j + j_0} = \emptyset, \quad \text{for all } j \geq 0 \text{ and } \ell = 1, \dotsc, L_j,
\end{align}
and
\begin{align}\label{claim2}
C_{{adj}} \widetilde{B}_{j,\ell} \cap U_{j - j_0} = \emptyset, \quad \text{for all } j \geq j_0 \text{ and } \ell = 1, \dotsc, L_j.
\end{align}
Proofs of these claims are provided in Yang et al.\,\cite{Yang2019}, which are omitted here.

These properties show that each enlarged ball \( C_{{adj}} \widetilde{B}_{j,\ell} \) intersects at most \( 2j_0 + 1 \) annuli \( U_j \). Additionally, for every \( j \) and \( \ell \), the ball \( C_{\widetilde{j}_0} \widetilde{B}_{j,\ell} \) contains \( B_0 \).
Given the ball \( B_0 \), equation \eqref{eq:ball;included} guarantees the existence of an integer \( t_0 \in \{1, 2, \dotsc, \mathcal{K}\} \) and a cube \( Q_0 \in \mathscr{D}^{t_0} \) such that 
\( B_0 \subseteq Q_0 \subseteq C_{{adj}} B_0\).

Let \( B(Q_0) \) be the ball containing \( Q_0 \) with measure comparable to \( \mu(Q_0) \) as described in \eqref{eq:contain}. Then, \( B(Q_0) \) includes \( B_0 \) and satisfies
$\mu(B(Q_0)) \lesssim \mu(B_0)$,
where the implicit constant depends only on \( C_{{adj}} \), \( C_\mu \), and \( A_1 \).

\vspace{1cm}

  \item[\textbullet] \textbf{Step 2} Secondly, based on the previous discussion, we will focus our proof on the local version of \eqref{2.7_1}. Without loss of generality, we only consider the case where $\tau_{\ell}=\{1,2,3\}$ and $\tau=\{1,2\}$.
  
  We show that there exists a {${1\over 2}$}--sparse family $\mathcal F^{t_0} \subseteq  \mathscr{D}^{t_0}(Q_0)$, such that for a.e.
  $x\in B_0$,
  \begin{align}\notag
    & \quad \|\mathcal{T}_{\eta,\tau_{\ell}}^{{\bf b, k}}(\vec{f}\chi_{\bqo})(x)\|_{\B} \\ \notag
    & \leq C_{\mathcal{T}_\eta}\sum_{{\bf k, t}} \sum_{Q \in \mathcal{F}^{t_0}}\Bigg( \mu(\bq)^{\an \cdot \frac{1}{r}} \prod_{i =1}^2 \left|b_i(x) - b_{i,R_Q}\right|^{k_i - t_i} \\ \label{1015_1}
    &\quad \quad \quad\quad \quad \quad \quad \quad \quad  \times \langle\left|f_i (b_i - b_{i,R_Q})^{t_i}\right|^r\rangle^{\frac{1}{r}}_{\bq} \langle\left|f_3\right|^r\rangle^{\frac{1}{r}}_{\bq} \Bigg) \chi_Q(x), 
 \end{align}
 where the operator notation $\sum\limits_{{\bf k, t}}$ means $\sum\limits_{t_1=0}^{k_1} \sum\limits_{t_2=0}^{k_2} \Big( \prod\limits_{i=1}^{2} C_{k_i}^{t_i} \Big)$. Setting $R_Q$ is the dyadic cube in $\mathscr{D}^{\mathfrak{k}}$ for some $\mathfrak{k} \in \{1, 2, \ldots, \mathcal{K}\}$ such that
$C_{\widetilde j_0} B(Q)\subseteq R_Q \subseteq C_{adj}\cdot C_{\widetilde j_0} B(Q)$, where $B(Q)$ is defined as in \eqref{eq:contain}, $j_0$ defined as in \eqref{j0} and $\widetilde j_0$  defined as in \eqref{widetildej0}.

    It suffices to show that we can choose a family of pairwise disjoint cubes $\{P_j\}\subseteq \mathscr{D}^{t_0}(Q_0)$ with $\sum\limits_j\mu(P_j)\leq {1\over 2}\mu(Q_0)$ and such that for a.e. $x\in B_0$,
    \begin{align}\notag
    & \quad \big\|\mathcal{T}_{\eta,\tau_{\ell}}^{{\bf b, k}}(f_1\chi_{\bq},f_2\chi_{\bq},f_3\chi_{\bq})(x)\big\|_{\B}\chi_{Q_0}(x) \\\notag
    & \lesssim C_{\mathcal{T}_{\eta}} \sum_{{\bf k, t}}\Bigg(\mu(\bq)^{\an \cdot \frac{1}{r}}\Big(\prod_{i=1}^{2}|b_i(x)-b_{i,R_{Q_0}}|^{k_i-t_i}\\\notag
    &\quad \quad \quad  \times\langle|f_i(b_i-b_{i,R_{Q_0}})^{t_i}|^r\rangle^{\frac{1}{r}}_{\bqo}\Big)  \langle\left|f_3\right|^r\rangle^{\frac{1}{r}}_{\bqo} \Bigg) \chi_{Q_0}(x) \\\label{eq:suff.to.A_}
    & \quad\quad \quad \quad+\sum_j \big\|\mathcal{T}_{\eta,\tau_{\ell}}^{{\bf b, k}} ( f_1 \chi_{\bp},f_2 \chi_{\bp},f \chi_{\bp})(x)\big\|_{\B} \chi_{P_j}(x).
    \end{align}

Indeed, let $\left\{P_{j_0}\right\}=\left\{Q_0\right\},\left\{P_{j_1}\right\}=\left\{P_j\right\}$ and $\left\{P_{j_1,\cdots,j_k}\right\}$ are the cubes obtained at the $k$-th stage of the iterative process. By above setting, on the one hand, 
\begin{align*}
&\quad\|\mathcal{T}_{\eta,\tau_{\ell}}^{{\bf b, k}}(f_1\chi_{\cjo B(P_{j_0})},f_2\chi_{\cjo B(P_{j_0})},f_3\chi_{\cjo B(P_{j_0})})\|_{\B}\chi_{Q_0}(x) \\
  &  \lesssim  C_{\mathcal{T}_{\eta}} \sum_{{\bf k, t}}\Bigg(\mu(\cjo B(P_{j_0}))^{\an \cdot \frac{1}{r}}\Big(\prod_{i=1}^{2}|b_i(x)-b_{i,R_{B(P_{j_0})}}|^{k_i-t_i}\\
  &\quad \quad \quad \quad \quad\times \langle|f_i(b_i-b_{i,R_{B(P_{j_0})}})^{t_i}|^r\rangle^{\frac{1}{r}}_{\cjo B(P_{j_0})} \Big) \langle\left|f_3\right|^r\rangle^{\frac{1}{r}}_{\cjo B(P_{j_0})}\Bigg) \chi_{P_{j_0}} 
  \\
  & \quad \quad \quad \quad\quad \quad+  \sum_{j_1} \big\|\mathcal{T}_{\eta,\tau_{\ell}}^{{\bf b, k}}(f_1\chi_{\cjo B(P_{j_1})},f_2\chi_{\cjo B(P_{j_1})},f_3\chi_{\cjo B(P_{j_1})})\|_{\B}\chi_{P_{j_1}}\\
  & \lesssim C_{\mathcal{T}_{\eta}} \sum_{{\bf k, t}} \sum_{j_0,j_1}  \Bigg( \mu(\cjo B(P_{j_0,j_1}))^{\an \cdot \frac{1}{r}} \Big( \prod_{i=1}^{2}|b_i(x)-b_{i,R_{B(P_{j_0,j_1})}}|^{k_i-t_i}\\
  & \quad \quad \quad \quad \quad \times \langle|f_i(b_i-b_{i,R_{B(P_{j_0,j_1})}})^{t_i}|^r\rangle^{\frac{1}{r}}_{\cjo B(P_{j_0,j_1})} \Big) \langle\left|f_3\right|^r\rangle^{\frac{1}{r}}_{\cjo B(P_{j_0,j_1})} \Bigg)\chi_{P_{j_0,j_1}} \\
  &  \quad \quad \quad \quad\quad \quad + \sum_{j_1,j_2} \big\|\mathcal{T}_{\eta,\tau_{\ell}}^{{\bf b, k}}(f_1\chi_{\cjo B(P_{j_1,j_2})},f_2\chi_{\cjo B(P_{j_1,j_2})},f_3\chi_{\cjo B(P_{j_1,j_2})})\big\|_{\B}\chi_{P_{j_1,j_2}}\\
  & \cdots \\
  & \lesssim C_{\mathcal{T}_{\eta}} \sum_{{\bf k, t}} \sum_{j_0,\cdots,j_k}  \Bigg( \mu(\cjo B(P_{j_0, \cdots,j_k}))^{\an \cdot \frac{1}{r}}\Big( \prod_{i=1}^{2}|b_i(x)-b_{i,R_{B(P_{j_0, \cdots,j_k})}}|^{k_i-t_i}\\
  & \quad \quad \quad \quad \quad \times \langle|f_i(b_i-b_{i,R_{B(P_{j_0, \cdots,j_k})}})^{t_i}|^r\rangle^{\frac{1}{r}}_{\cjo B(P_{j_0, \cdots,j_k})}\Big)  \langle\left|f_3\right|^r\rangle^{\frac{1}{r}}_{\cjo B(P_{j_0, \cdots,j_k})} \Bigg) \chi_{P_{j_0, \cdots,j_k}} \\
  &  \quad \quad \quad \quad\quad \quad + \sum_{j_1,\cdots,j_k} \big\|\mathcal{T}_{\eta,\tau_{\ell}}^{{\bf b, k}}(f_1\chi_{\cjo B(P_{j_1, \cdots,j_k})},f_2\chi_{\cjo B(P_{j_1, \cdots,j_k})},f_3\chi_{\cjo B(P_{j_1, \cdots,j_k})})\big\|_{\B}\chi_{P_{j_1, \cdots,j_k}}.\\
\end{align*}
Due to the fact that 
\begin{align*}
  \sum_{j_1,\cdots,j_k}\mu(P_{j_1,\cdots,j_k}) \leq \sum_{j_1,\cdots,j_{k-1}}\frac{1}{2}\mu(P_{j_1,\cdots,j_{k-1}})\lesssim \frac{1}{2^{k}}\mu(B_0).
\end{align*}

On the other hand,
\begin{align*}
  &\lim_{k \rightarrow \infty} \sum_{j_1,\cdots,j_k} \big\|\mathcal{T}_{\eta,\tau_{\ell}}^{{\bf b, k}}(f_1\chi_{\cjo B(P_{j_1, \cdots,j_k})},f_2\chi_{\cjo B(P_{j_1, \cdots,j_k})},f_3\chi_{\cjo B(P_{j_1, \cdots,j_k})}\big\|_{\B}\chi_{P_{j_1, \cdots,j_k}} \\
  &\quad \quad \quad \leq \big\|\mathcal{T}_{\eta,\tau_{\ell}}^{{\bf b, k}}(f_1\chi_{\cjo B(P_{j_1, \cdots,j_k})},f_2\chi_{\cjo B(P_{j_1, \cdots,j_k})},f_3\chi_{\cjo B(P_{j_1, \cdots,j_k})}\big\|_{\B}\chi_{\lim\limits_{k \rightarrow \infty}\bigcup\limits_{j_1, \cdots,j_k}P_{j_1, \cdots,j_k}} = 0.\\
\end{align*}

Therefore, it follows from two aspects that
\begin{align*}
  &\quad \big\|\mathcal{T}_{\eta,\tau_{\ell}}^{{\bf b, k}}(f_1\chi_{\bqo},f_2\chi_{\bqo},f_3\chi_{\bqo})\big\|_{\B}\chi_{Q_0} \\
  & \lesssim C_{\mathcal{T}_{\eta}} \lim_{k \rightarrow \infty} \bigg( \sum_{{\bf k, t}}\sum_{j_0,\cdots,j_k}   \mu(\cjo B(P_{j_0, \cdots,j_k}))^{\an \cdot \frac{1}{r}} |b_j(x)-b_{j,R_{B(P_{j_0, \cdots,j_k})}}|^{k_i-t_i}\\
  & \quad \quad \quad \quad \quad \times \langle|f(b_j-b_{j,R_{B(P_{j_0, \cdots,j_k})}})^{t_i}|^r\rangle^{\frac{1}{r}}_{\cjo B(P_{j_0, \cdots,j_k})}  \langle\left|f_3\right|^r\rangle^{\frac{1}{r}}_{\cjo B(P_{j_0, \cdots,j_k})} \chi_{P_{j_0, \cdots,j_k}} \\
  &  \quad \quad \quad \quad\quad \quad + \sum_{j_1,\cdots,j_k} \big\|\mathcal{T}_{\eta,\tau_{\ell}}^{{\bf b, k}}(f_1\chi_{\cjo B(P_{j_1, \cdots,j_k})},f_2\chi_{\cjo B(P_{j_1, \cdots,j_k})},f_3\chi_{\cjo B(P_{j_1, \cdots,j_k})})\big\|_{\B}\chi_{P_{j_1, \cdots,j_k}}\bigg)\\
  & \lesssim C_{\mathcal{T}_{\eta}} \lim_{k \rightarrow \infty} \sum_{{\bf k, t}} \sum_{j_0,\cdots,j_k}  \Bigg( \mu(\cjo B(P_{j_0, \cdots,j_k}))^{\an \cdot \frac{1}{r}} \Big(\prod_{i=1}^{2}|b_i(x)-b_{i,R_{B(P_{j_0, \cdots,j_k})}}|^{k_i-t_i}\\
  & \quad \quad \quad \quad \quad \times \langle|f_i(b_i-b_{i,R_{B(P_{j_0, \cdots,j_k})}})^{t_i}|^r\rangle^{\frac{1}{r}}_{\cjo B(P_{j_0, \cdots,j_k})}\Big)  \langle\left|f_3\right|^r\rangle^{\frac{1}{r}}_{\cjo B(P_{j_0, \cdots,j_k})}\Bigg) \chi_{P_{j_0, \cdots,j_k}} \\
  &=: \sum_{{\bf k, t}} \Bigg(\sum\limits_{Q \in \F^{t_0}}  \mu(\cjo B(Q))^{\an \cdot \frac{1}{r}}\Big(\prod_{i=1}^{2} |b_i(x)-b_{i,R_{B(Q)}}|^{k_i-t_i}\\
  & \quad \quad \quad \quad \quad \times \langle|f_i(b_i-b_{i,R_{B(Q)}})^{t_i}|^r\rangle^{\frac{1}{r}}_{\cjo B(Q)}\Big)  \langle\left|f_3\right|^r\rangle^{\frac{1}{r}}_{\cjo B(Q)}\Bigg)  \chi_{Q} .\\
\end{align*}
Consequently, \eqref{1015_1} is valid, so we just need to prove \eqref{eq:suff.to.A_} next.

We observe that for any arbitrary family of disjoint cubes $\{P_j\}\subseteq \mathscr{D}^{t_0}(Q_0)$, we have that
\begin{align*}
  & \quad \big\|\mathcal{T}_{\eta,\tau_{\ell}}^{{\bf b, k}}\big(f_1 \chi_{\bqo},f_2 \chi_{\bqo},f_3 \chi_{\bqo}\big)(x)\big\|_{\B} \chi_{Q_0} \\
  & =\big\|\mathcal{T}_{\eta,\tau_{\ell}}^{{\bf b, k}}\big(f_1 \chi_{\bqo},f_2 \chi_{\bqo},f_3 \chi_{\bqo}\big)(x)\big\|_{\B} \chi_{Q_0 \backslash \cup_j P_j}\\
  &\quad \quad  +\sum_j \|\mathcal{T}_{\eta,\tau_{\ell}}^{{\bf b, k}}\big(f_1 \chi_{\bqo},f_2 \chi_{\bqo},f_3 \chi_{\bqo}\big)(x)\|_{\B} \chi_{P_j} \\
  & \leq \big\|\mathcal{T}_{\eta,\tau_{\ell}}^{{\bf b, k}}\big(f_1 \chi_{\bqo},f_2 \chi_{\bqo},f_3 \chi_{\bqo}\big)(x)\big\|_{\B} \chi_{Q_0 \backslash \cup_j P_j}\\
  & \quad \quad +\sum_j \big\|\mathcal{T}_{\eta,\tau_{\ell}}^{{\bf b, k}}\big(f_1 \chi_{\bqo \backslash \bp},f_2 \chi_{\bqo \backslash \bp},f_3 \chi_{\bqo \backslash \bp}\big)(x)\big\|_{\B} \chi_{P_j}\\
  &\quad \quad + \sum_j \|\mathcal{T}_{\eta,\tau_{\ell}}^{{\bf b, k}}\big(f_1 \chi_{\bp},f_2 \chi_{\bp},f_3 \chi_{\bp}\big)(x)\|_{\B} \chi_{P_j}\\
  &:=I_1 +I_2 + \sum_j \|\mathcal{T}_{\eta,\tau_{\ell}}^{{\bf b, k}}\big(f_1 \chi_{\bp},f_2 \chi_{\bp},f_3 \chi_{\bp}\big)(x)\|_{\B} \chi_{P_j}.
  \end{align*}
Suppose that
\begin{align}
      I_1 + I_2 &  \lesssim C_{\mathcal{T}_{\eta}} \sum_{{\bf k, t}}\Bigg(\mu(\bq)^{\an \cdot \frac{1}{r}}\Big(\prod_{i=1}^{2}|b_i(x)-b_{i,R_{Q_0}}|^{k_i-t_i} \notag\\
      &\quad \quad \quad  \times\langle|f_i(b_i-b_{i,R_{Q_0}})^{t_i}|^r\rangle^{\frac{1}{r}}_{\bqo}\Big)  \langle\left|f_3\right|^r\rangle^{\frac{1}{r}}_{\bqo} \Bigg) \chi_{Q_0}(x). \label{CZresults_3}
      \end{align}
Then \eqref{eq:suff.to.A_} follows directly from \eqref{CZresults_3}, concluding that \eqref{1015_1}. We will now prove \eqref{CZresults_3} as follows.

We will use Calderón--Zygmund decomposition.
Note that 
\begin{align*}
  &\quad \prod_{i=1}^{2} (b_i(x) - b_i(y_i))^{k_i} \\
  &= \sum_{t_1=0}^{k_1} \sum_{t_2=0}^{k_2} \Big( \prod_{i=1}^{2} C_{k_i}^{t_i} (-1)^{t_i} [b_i(x) - b_{i,R_{Q_0}}]^{k_i - t_i} [b_i(y_i) - b_{i,R_{Q_0}}]^{t_i} \Big).
\end{align*}
Combining with the definition of $\mathcal{T}_{\eta,\tau_{\ell}}^{{\bf b, k}}$, we conclude that 
\begin{align*}
&\quad I_1+I_2\\
&\leq \sum_{{\bf k, t}}\prod_{i=1}^{2}\big|b_i(x)-b_{i,R_{Q_0}}\big|^{k_i-t_i} \\
  &  \times \big\|\mathcal{T}_{\eta}^{}\big(\big(b_1-b_{1,R_{Q_0}}\big)^{t_1}f_1\chi_{\bqo},\big(b_2-b_{2,R_{Q_0}}\big)^{t_2}f_2\chi_{\bqo},f_3 \chi_{\bqo}\big)(x)\big\|_{\B} \chi_{Q_0 \backslash \cup_j P_j} \\
  &+\sum_{{\bf k, t}} \prod_{i=1}^{2}\big|b_i(x)-b_{i,R_{Q_0}}\big|^{k_i-t_i} \big\| \mathcal{T}_{\eta}^{}\big(\big(b_1-b_{1,R_{Q_0}}\big)^{t_1}f_1\chi_{\bqo \backslash \bp},\\
  &\quad \quad \quad \quad \quad \big(b_2-b_{2,R_{Q_0}}\big)^{t_2}f_2\chi_{\bqo \backslash \bp},f_3 \chi_{\bqo \backslash \bp}\big)(x)\|_{\B} \chi_{P_j}.
  \end{align*}  

Now, we write
\begin{align*}
H_1(x):=&\big\|\mathcal{T}_{\eta}^{}\big(\big(b_1-b_{1,R_{Q_0}}\big)^{t_1}f_1\chi_{\bqo},\big(b_2-b_{2,R_{Q_0}}\big)^{t_2}f_2\chi_{\bqo},f_3 \chi_{\bqo}\big)(x)\big\|_{\B} \chi_{Q_0 \backslash \cup_j P_j},\\
H_2(x):=&\big\| \mathcal{T}_{\eta}^{}\big(\big(b_1-b_{1,R_{Q_0}}\big)^{t_1}f_1\chi_{\bqo \backslash \bp},\\
&\big(b_2-b_{2,R_{Q_0}}\big)^{t_2}f_2\chi_{\bqo \backslash \bp},f_3 \chi_{\bqo \backslash \bp}\big)(x)\|_{\B} \chi_{P_j}.
\end{align*}
In order to verify \eqref{CZresults_3}, it suffices to estimate $H_1$ and $H_2$.

Denote $E=\bigcup\limits_{\tau \subseteq \tau_\ell} \bigcup\limits_{t_i =0}^{k_i}\big(E_{t_i,\tau}^1 \cup E_{t_i,\tau}^2\big)$, where
\begin{align*}
  E_{t_i,\tau}^{1} &:= \big\{ x \in B_0:\prod_{i =1}^2\big|(b_i(x)-b_{R_{i,Q_0}})^{t_i} f_i\big|  |f_3| \big. \\
  &\quad \quad \quad  \big. >\alpha \prod_{i =1}^2 \langle\big|\big(b_i-b_{i,R_{Q_0}}\big)^{t_i} f_i\big|^r\rangle^{\frac{1}{r}}_{\bqo}\langle\big|f_3\big|^r\rangle^{\frac{1}{r}}_{\bqo} \big\},\\
  E_{t_i,\tau}^{2} &:= \big\{ x \in B_{0} : \mathcal{M}_{\mathcal{T}_{\eta}^{}, B_0}\big(\big(b_1-b_{1,R_{Q_0}}\big)^{t_1}f_1\chi_{\bqo},\big(b_2-b_{2,R_{Q_0}}\big)^{t_2}f_2\chi_{\bqo},f_3 \chi_{\bqo}\big)(x)\big.\\
  &\quad \quad \quad  \big. > \alpha \prod_{i =1}^2 \mu(\bqo)^{\an \cdot \frac{1}{r}}\langle\big|\big(b_i-b_{i,R_{Q_0}}\big)^{t_i} f_i\big|^r\rangle^{\frac{1}{r}}_{\bqo} \langle\big|f_3\big|^r\rangle^{\frac{1}{r}}_{\bqo} \big\}.
  \end{align*}
    
  By Lemma \ref{Condi.Mg}, choosing \( \alpha \) sufficiently large (depending on \( C_{\widetilde{j}_0} \), \( C_{{adj}} \), \( C_\mu \), and \( A_1 \) from \eqref{eq:contain}), we have
  \begin{align}\label{CZeq_1}
    \mu(E) \leq \frac{1}{4 C_{\mu, 0}} \mu\left(B_0\right),
  \end{align}
  where \( C_{\mu,0} \) is defined in \eqref{Cmu0}.
  
  Next, we apply the Calderón--Zygmund decomposition to \( \chi_E \) on \( B_0 \) at the height $
  \lambda = \frac{1}{2 C_{\mu,0}}$ and 
  obtain pairwise disjoint cubes \( \{P_j\} \subseteq \mathscr{D}^{t_0}(Q_0) \) such that
  \begin{align}\label{CZeq_2}
    \frac{1}{2 C_{\mu,0}} \mu(P_j) \leq \mu(P_j \cap E) \leq \frac{1}{2} \mu(P_j),
  \end{align}
  and
  \[
  \mu\Big( E \setminus \bigcup_j P_j \Big) = 0.
  \]
  Consequently,
  \begin{align}\label{CZeq_3}
    \sum_j \mu(P_j) \leq \frac{1}{2} \mu(B_0) \quad \text{with} \quad P_j \cap E^c \neq \emptyset \text{ for each } j.
  \end{align}
    
  In the following we proceed to the estimation of $H_1$ and $H_2$.
By the definition of $E_{t_i,\tau}^{2}$ and ${P_j} \cap E^c \ne \emptyset $, for $x \in P_j$, we have
  \begin{align*}
   &\quad H_2(x)\\
    &\le \underset{x\in P_j}{\operatorname{ess} \inf } \big\|\mathcal{M}_{\mathcal{T}_{\eta}^{}, B_0}\big(\big(b_1-b_{1,R_{Q_0}}\big)^{t_1}f_1\chi_{\bqo},\big(b_2-b_{2,R_{Q_0}}\big)^{t_2}f_2\chi_{\bqo},f_3 \chi_{\bqo}\big)(x)\big\|_{\B}\\
    &  \leq C_{\mathcal{T}_\eta} \prod_{i =1}^2 \mu(\bqo)^{\an \cdot \frac{1}{r}}\prod_{i=1}^{2}\langle\big|\big(b_i-b_{i,R_{Q_0}}\big)^{t_i} f_i\big|^r\rangle^{\frac{1}{r}}_{\bqo} \langle|f_3|^r\rangle^{\frac{1}{r}}_{\bqo},
  \end{align*} 

Last but not least, $H_1$ will be bounded as follows. It follows from Lemma \ref{Lemma:preLi2018} and the construction of $E$ that for a.e. $x \in B_0\backslash\bigcup\limits_j P_j$, 
    \begin{align*}
 H_1(x) &\le \|\mathcal{T}_{\eta}\|_{L^t \times \cdots \times L^t \rightarrow L^{\tilde{t}, \infty}} \mu(\cjo B_0)^{\an \cdot \frac{1}{r}} \prod_{i =1}^2 \left|b_i(x) - b_{i,R_Q}\right|^{k_i - t_i} \prod_{i=1}^3\big|f_i(x)\big|+\mathcal{M}_{\mathcal{T}_{\eta}, B_0}(\vec{f})(x)\\
&\lesssim C_{\mathcal{T}_\eta} \mu(\bqo)^{\an \cdot \frac{1}{r}} \prod_{i=1}^{2}\langle\big|\big(b_i-b_{i,R_{Q_0}}\big)^{t_i} f_i\big|^r\rangle^{\frac{1}{r}}_{\bqo}\langle|f_3|^r\rangle^{\frac{1}{r}}_{\bqo}.
  \end{align*}

Hence, \eqref{CZresults_3}, \eqref{eq:suff.to.A_}, and \eqref{1015_1} are valid.
    
\item[\textbullet] \textbf{Step 3} In this section, we will extend the local results \eqref{eq:suff.to.A_} to $X$.  
Assume that $f_i$, for $i=1,\cdots,m$, are supported in a ball $B_0 \subseteq X$. Then, we have \( X = \bigcup\limits_{j=0}^\infty 2^j B_0 \).

For each $j \geq 0$, define the annulus $U_j = 2^{j+1} B_0 \setminus 2^j B_0$ and cover it with balls $\{\widetilde B_{j,\ell}\}_{\ell=1}^{L_j}$ as in {\bf Step 1}. 
For each ball $\widetilde B_{j,\ell}$, there exist $t_{j,\ell} \in \{1, 2, \ldots, \mathcal K\}$ and $\widetilde Q_{j,\ell} \in \mathscr{D}^{t_{j,\ell}}$ such that 
$\widetilde B_{j,\ell} \subseteq \widetilde Q_{j,\ell} \subseteq C_{{adj}} \widetilde B_{j,\ell}$.

Furthermore, since $C_{\widetilde j_0} \widetilde B_{j,\ell}$ covers $B_0$, the enlargement $C_{\widetilde j_0} B(\widetilde Q_{j,\ell})$ also covers $B_0$. 
Applying \eqref{1015_1} to each $\widetilde B_{j,\ell}$ yields a $\frac{1}{2}$-sparse family $\mathcal{\widetilde F}_{j,\ell} \subseteq \mathscr{D}^{t_{j,\ell}}(\widetilde Q_{j,\ell})$ such that \eqref{1015_1} holds for almost every $x \in \widetilde B_{j,\ell}$.

Now we set $\mathcal F := \bigcup\limits_{j,\ell} \mathcal{\tilde F}_{j,\ell}$.
Note that the balls $C_{adj}\widetilde B_{j,\ell}$ are overlapping at most $C_{A_0,\mu,\widetilde j_0}(2j_0+1)$ times, where $C_{A_0,\mu,\widetilde j_0}$ is the constant in {\bf Step 1}.   Then we obtain that
$\mathcal F $ is a ${1\over 2C_{A_0,\mu,\widetilde j_0}(2j_0+1)} $--sparse family and for a.e. $x\in X$,
    \begin{align}\notag
      & \quad \big\|\mathcal{T}_{\eta,\tau_{\ell}}^{{\bf b, k}}(\vec{f})(x)\big\|_{\B} \\ \notag
      & \lesssim C_{\mathcal{T}_\eta}\sum_{\tau \subseteq \tau_{\ell}}\sum_{{\bf k, t}} \sum_{Q \in \mathcal{F}^{t_0}} \left(\mu(\bq)^{\an \cdot \frac{1}{r}} \prod_{i \in \tau} \left|b_i(x) - b_{i,R_Q}\right|^{k_i - t_i}\right. \\ \notag
      &\quad \quad \quad\quad \quad \quad \quad \quad \quad \left. \times \langle\left|f_i (b_i - b_{i,R_Q})^{t_i}\right|^r\rangle^{\frac{1}{r}}_{\bq} \prod_{j \in \tau'} \langle\left|f_i\right|^r\rangle^{\frac{1}{r}}_{\bq} \right) \chi_Q(x).
      \end{align}

Note that $\overline C'\mu(C_{\widetilde j_0}B(Q)) \leq \mu(R_Q)\leq \overline C\mu(C_{\widetilde j_0}B(Q))$, where $\overline C'$ and $\overline C$ depends only on $C_\mu$ and $C_{adj}$. Then we obtain
\begin{align}\notag
  & \quad \|\mathcal{T}_{\eta,\tau_{\ell}}^{{\bf b, k}}(\vec{f})(x)\|_{\B} \\ \notag
  & \lesssim C_{\mathcal{T}_\eta}\sum_{\tau \subseteq \tau_{\ell}}\left(\sum_{{\bf k, t}} \sum_{Q \in \mathcal{F}^{t_0}} \mu(R_Q)^{\an \cdot \frac{1}{r}} \prod_{i \in \tau} \left|b_i(x) - b_{i,R_Q}\right|^{k_i - t_i}\right. \\ \notag
  &\quad \quad \quad\quad \quad \quad \quad \quad \quad \left. \times \langle\left|f_i (b_i - b_{i,R_Q})^{t_i}\right|^r\rangle^{\frac{1}{r}}_{R_Q} \prod_{j \in \tau'} \langle\left|f_i\right|^r\rangle^{\frac{1}{r}}_{R_Q} \right) \chi_Q(x).
  \end{align}

Last, we define  
$ \mathcal S_{\mathfrak{k}}:=\{R_Q\in \mathscr D^\mathfrak{k}:\ Q\in\mathcal F\}, \ \ \mathfrak{k}\in\{1,2,\ldots,\mathcal K\}. $  
Given that $\mathcal F$ is ${1\over 2C_{A_0,\mu,\widetilde j_0}(2j_0+1)}$--sparse, it follows that each family $\mathcal S_{t}$ is ${1\over 2C_{A_0,\mu,\widetilde j_0}(2j_0+1)\overline c}$--sparse.  
We then set  
$$\delta:={1\over 2C_{A_0,\mu,\widetilde j_0}(2j_0+1)\overline c}, $$  
where $\overline c$ is a constant depending only on $\overline C$, $C_{\widetilde j_0}$, and the doubling constant $C_\mu$.

This proof is finished. \qedhere
\end{proof}

\section{\bf Proof for Quantitative weighted estimates}\label{weight.estimate}

\subsection{Proof of Theorem~\ref{Horm.pro}}
~~

Let us first consider the case where the proof can be simplified to $p_0 = 1$. The proof of \eqref{classical_A} is equivalent to: if $\sum\limits_{i} \frac{1}{p_i} - \frac{1}{q} = \an$ and $\vec \omega \in A^{\star}_{\vec p, q}(X)$, then
\begin{align*}
 \|\mathcal{A}_{\eta,\S,p_0,\gamma}(\vec{f})\|_{L^q(u)} \lesssim C \prod_{i=1}^m \|f_i\|_{L^{p_i}(\omega_i)},
\end{align*}           
where $u:=\omega$.

By a simple substitution, we conclude that if $\sum\limits_{i} \frac{1}{p_i/p_0} - \frac{1}{q/p_0} = \an \cdot p_0$, $\vec \omega \in A^{\star}_{\vec p/p_0,q/p_0}(X)$, then
\begin{align*}
\|\mathcal{A}_{p_0\eta,\S,1,\frac{\gamma}{p_0}}(\vec{f})\|_{L^q(u)} \lesssim C^{p_0} \prod_{i=1}^m \|f_i\|_{L^{p_i}(\omega_i)}.
\end{align*}

Therefore, if we denote by $\mathcal N(\vec p, \gamma, \omega, \sigma)$ the best constant for the case $p_0 = 1$, then for a general $p_0$, the best constant would be $\mathcal N(\vec p / p_0, \gamma / p_0, \omega, \sigma)^{1/p_0}$. Thus, it suffices to consider the case $p_0 = 1$, that is,
\begin{equation}\label{eq1-thm1.1}
\|\mathcal{A}_{\eta,\S,1,\gamma}(\vec{f}\sigma)\|_{L^q(u)} \lesssim [\vec{\omega}]_{A^{\star}_{\vec{p},q}(X)}^{\max\{\f{1}{\gamma},\tfrac{p_1'}{q},\dotsc,\tfrac{p_m'}{q}\}} \prod_{i=1}^m \|f_i\|_{L^{p_i}(\sigma_i)}.
  \end{equation}
Throughout the proof, let $\sigma_i=\omega_i^{1-p_i'}$, $\vec{f}\sigma =(f_1\sigma_1,\dotsc, f_m\sigma_m)$ and $f_i\geq 0$.  Combining $\vec{\omega} \in A^{\star}_{\vec{p},q}(X)$ and Lemma \ref{vweight4}, we obtain $\sigma_i, u \in A_\infty$.

Let $\theta=\min\{q,\gamma \}$. We get
$$
\begin{aligned}
\|\mathcal{A}_{\eta,\S,1,\gamma}(\vec{f}\sigma)\|_{L^q(u)}^{\theta}&= \left(\int_{X} [\sum_{Q\in \mathcal{S}} \Big(\mu(Q)^{\an}\prod_{i=1}^m \f{1}{\mu(Q)}\int_Q f_i\sigma_i\Big)^\gamma \chi_{Q}(x)]^{\frac{q}{\gamma}}u       \right)^{\frac{\theta}{q}}\\
 &\leq \left(\int_{X} [\sum_{Q\in \mathcal{S}} \Big(\mu(Q)^{\an}\prod_{i=1}^m \f{1}{\mu(Q)}\int_Q f_i\sigma_i\Big)^{\theta} \chi_{Q}(x)]^{\frac{q}{\theta}}u       \right)^{\frac{\theta}{q}},
\end{aligned}
$$
where we used the fact $\theta \leq \gamma$. Thus
\begin{equation}\label{eqdual-thm1.1}
\|\mathcal{A}_{\eta,\S,1,\gamma}(\vec{f}\sigma)\|_{L^q(u)}^{\theta} \leq \|[\mathcal{A}_{\eta,\S,1,\theta}(\vec{f}\sigma)]^{\theta}\|_{L^{q/\theta}(u)}.
\end{equation}
Denote $\beta=\max\{\f{1}{\theta},\tfrac{p_1'}{q},\dotsc,\tfrac{p_m'}{q}\}$.  Also assume that $g\in L^{(q/\theta)'}(u)$ and $g\geq 0$.

We have
$$
\int_{X}[\mathcal{A}_{\eta,\S,1,\theta}(\vec{f}\sigma)]^{\theta} gu=\sum_{Q\in \mathcal{S}}\int_Q gu \times \Big(\mu(Q)^{\eta}\prod_{i=1}^m\f{1}{\mu(Q)}\int_Q f_i\sigma_i\Big)^{\theta}.
$$
From this and the definition of $[{\vec \omega}]_{A^{\star}_{\vec{p},q}(X)}$, we obtain
\begin{align*}
  &\quad \sum_{Q\in \mathcal{S}}\int_Q gu \times \Big(\mu(Q)^{\an}\prod_{i=1}^m\f{1}{\mu(Q)}\int_Q f_i\sigma_i\Big)^{\theta}\\
&\leq [\vec \omega]_{A^{\star}_{\vec{p},q}(X)}^{\beta \theta}\sum_{Q\in \mathcal{S}}\f{\mu(Q)^{(m - \an)\theta(\beta q-1)}}{u(Q)^{\beta \theta-1}\prod\limits_{i=1}^m \sigma_i(Q)^{\theta(\beta q/p'_i-1)}}\times \Big(\f{1}{u(Q)}\int_Q gu\Big) \times \Big(\prod_{i=1}^m \f{1}{\sigma_i(Q)}\int_Q f_i\sigma_i\Big)^{\theta}\\
&\leq (1/\delta)^{(m - \an)\theta(\beta q-1)}[\vec \omega]_{A^{\star}_{\vec{p},q}(X)}^{\beta \theta}\sum_{Q\in \mathcal{S}}\f{\mu(E_Q)^{^{(m - \an)\theta(\beta q-1)}}}{u(E_Q)^{\beta \theta-1}\prod\limits_{i=1}^m \sigma_i(E_Q)^{\theta(\beta q/p'_i-1)}}\times \Big(\f{1}{u(Q)}\int_Q gu\Big)\\
& \ \ \ \times \Big(\prod_{i=1}^m \f{1}{\sigma_i(Q)}\int_Q f_i\sigma_i\Big)^{\theta},
\end{align*}
where in the last inequality we used the facts $ u(Q) \geq u(E_Q)$, $\sigma_i(Q)\geq \sigma_i(E_Q)$ (refer to Definition \ref{vweight3} and \ref{D:sparse}) and the positivity of the exponents.

On the other hand, by H\"older's inequality, we have
\begin{equation}\label{eq2-thm1.1}
\mu(E_Q)=\int_{E_Q}u^{\f{1}{(m-\eta)q}}\prod_{i=1}^m\sigma_i^{\f{1}{(m-\eta)p'_i}}\leq u(E_Q)^{\f{1}{(m-\eta)q}}\prod_{i=1}^m\sigma_i(E_Q)^{\f{1}{(m-\eta)p'_i}}.
\end{equation}
Insert this into the estimate above to conclude that
\begin{align*}
 & \quad \sum_{Q\in \mathcal{S}}\int_Q gu \times \Big(\mu(Q)^{\an}\prod_{i=1}^m\f{1}{\mu(Q)}\int_Q f_i\sigma_i\Big)^{\theta}\\
&\leq (1/\delta)^{^{(m - \eta)\theta(\beta q-1)}}[\vec{\omega}]_{A^{\star}_{\vec{p},q}(X)}^{\beta \theta} \\
&\quad\times \sum_{Q\in \mathcal{S}} \left[ \Big(\f{1}{u(Q)}\int_Q gu\Big)u(E_Q)^{\f{1}{(q/\theta)^{'}}} \right] \left[\prod_{i=1}^m \Big(\f{1}{\sigma_i(Q)}\int_Q f_i\sigma_i\Big) \sigma_i(E_Q)^{\f{1}{p_i}} \right]^{\theta},
\end{align*}
which together with H\"older's inequality and the disjointness of the family $\{E_Q\}_{Q\in \mathcal{S}}$ gives
\begin{align*}
&\quad \sum_{Q\in \mathcal{S}}\int_Q gu \times \Big(\mu(Q)^{\an}\prod_{i=1}^m\f{1}{\mu(Q)}\int_Q f_i\sigma_i\Big)^{\theta}\\
&\leq (1/\delta)^{^{(m - \eta)\theta(\beta q-1)}}[\vec{\omega}]_{A^{\star}_{\vec{p},q}(X)}^{\beta \theta} [\sum_{Q\in \mathcal{S}} \Big(\f{1}{u(Q)}\int_Q gu\Big)^{(q/\theta)'}u(E_Q)]^{\f{1}{(q/\theta)'}}\\
 & \ \ \ \times \prod_{i=1}^m [\sum_{Q\in \mathcal{S}}\Big(\f{1}{\sigma_i(Q)}\int_Q f_i\sigma_i\Big)^{p_i}\sigma_i(E_Q)]^{\theta/p_i}\\
&\leq (1/\delta)^{^{(m - \eta)\theta(\beta q-1)}}[\vec \omega]_{A^{\star}_{\vec{p},q}(X)}^{\beta \theta}\|M^{u}_{\d}(g)\|_{L^{(q/\theta)'}(u)} \times \prod_{i=1}^m \|M^{\sigma_i}_{\d}(f_i)\|_{L^{p_i}(\sigma_i)}^{\theta}\\
&\lesssim (1/\delta)^{^{(m - \eta)\theta(\beta q-1)}}[\vec \omega]_{A^{\star}_{\vec{p},q}(X)}^{\beta \theta} \|g\|_{L^{(q/\theta)'}(u)} \times \prod_{i=1}^m\|f_i\|_{L^{p_i}(\sigma_i)}^{\theta},
\end{align*}
{where to get the last inequality we applied \eqref{dyadicmaximal}}. Hence,
\begin{align*}
\|\mathcal{A}_{\eta,\S,1,\gamma}(\vec{f}\sigma)\|^{\theta}_{L^{q}(u)} 
&\overset{\text{(\ref{eqdual-thm1.1})}}{\leq}\ \|[\mathcal{A}_{\eta,\S,1,\theta}(\vec{f}\sigma)]^{\theta}\|_{L^{q/\theta}(u)} \\
&\leq \sup_{\| g\|_{L^{(q/\theta)'}(u)}=1 } \int_{X}[\mathcal{A}^{\theta}_{\eta,1,\gamma,\S}(\vec{f}\sigma)]^{\theta} gu\\
&\leq (1/\delta)^{^{(m - \eta)\theta(\beta q-1)}}[\vec{\omega}]_{A^{\star}_{\vec{p},q}(X)}^{\beta \theta} \times \prod_{i=1}^m\|f_i\|_{L^{p_i}(\sigma_i)}^{\theta}.
\end{align*}
This proves \eqref{eq1-thm1.1}. \qedhere


\subsection{Proof of Theorem \ref{Thm:1}}

Before proving Theorems \ref{Thm:1}, \ref{Horm.pro_2}, and \ref{Thm:3}, we need some preliminary results. The first proposition is from \cite{cov2004}, Proposition \ref{dyadicsum}.
\begin{proposition}[\cite{cov2004}, Proposition 2.2]\label{dyadicsum}
Let $1<s<\infty$, $\sigma$ be a positive Borel measure and
\[
\phi=\sum_{Q\in\mathcal D} \alpha_Q \mathbf \chi_Q,\qquad \phi_Q=\sum_{Q'\subset Q}\alpha_{Q'} \mathbf \chi_{Q'}.
\]
Then
\begin{equation}\label{dyadicsum_}
    \|\phi\|_{L^s(\sigma)}\approx \Big( \sum_{Q\in \mathcal D} \alpha_Q (\langle\phi_Q\rangle_Q^\sigma)^{s-1}\sigma(Q) \Big)^{1/s}.
\end{equation}
\end{proposition}

\begin{proposition}[\cite{Li2018}, Proposition 4.8]\label{kolmogorov}
Let $\mathcal S$ be a sparse family and $0\le s_1, s_2<1$ satisfying $a+b<1$. Then
\begin{equation}\label{eq:kolmogorov}
\sum_{\substack{Q\in \mathcal S\\Q\subseteq R}}\langle u \rangle_Q^{s_1} \langle v\rangle_Q^{s_2} \mu(Q) \lesssim \langle u \rangle_R^{s_1} \langle v\rangle_R^{s_2} \mu(R).
\end{equation}
\end{proposition}
A important observation can be stated as follows, as it was done in \cite{Li2017t}, \cite{HL}.
\begin{lemma}\label{lm:equivalent}
  Let $p_1, p_2 > \gamma$, and let $\mathcal N$ be the best constant for which the following inequality holds
\begin{equation}\label{eq:p01}
\|\mathcal A_{\eta,\mathcal S,1,\gamma }(f_1\sigma_1  ,f_2 \sigma_2 )\|_{L^q(X,u)}\le \mathcal N  \|f_1\|_{L^{p_1}(X,\sigma_1)} \|f_2\|_{L^{p_2}(X,\sigma_2)}.
\end{equation}
Then \eqref{eq:p01} is equivalent to the following inequality with $\mathcal N'\simeq \mathcal N^\gamma$
\begin{equation}\label{eq:gamma1}
 \Big\|\Big(\sum_{Q\in \mathcal S}\mu(Q)^{\an \cdot \gamma} \langle f_1   \rangle_{\sigma_1,Q} \langle   f_2  \rangle_{\sigma_2,Q} \langle \sigma_1\rangle_Q^\gamma \langle\sigma_2\rangle_Q^\gamma \mathbf \chi_Q\Big)^{\frac 1\gamma} \Big\|_{L^q(X,u)}^\gamma\le {\mathcal N'}  \|f_1\|_{L^{\frac{p_1}\gamma}(X,\sigma_1)} \|f_2\|_{L^{\frac{p_2}\gamma}(X,\sigma_2)}.
\end{equation}
\end{lemma}
\begin{proof}
  On the one hand, if \eqref{eq:gamma1} holds, we have
\begin{align*}
&\quad \|\mathcal A_{\eta,\mathcal S,1,\gamma }(f_1 \sigma_1 ,f_2 \sigma_2 )\|_{L^q(u)} \\
&\le \Big\|\Big(\sum_{Q\in \mathcal S} \mu(Q)^{\an \cdot \gamma} \langle M_{\mathcal D}^{\sigma_1}(f_1 )^\gamma \rangle_{\sigma_1,Q} \langle M_{\mathcal D}^{\sigma_2}(f_2)^\gamma \rangle_{\sigma_2,Q} \langle \sigma_1\rangle_Q^\gamma \langle\sigma_2\rangle_Q^\gamma \mathbf \chi_Q\Big)^{\frac 1\gamma} \Big\|_{L^q(u)}\\
&\lesssim\mathcal N  \| M_{\mathcal D}^{\sigma_1}(f_1 )^\gamma\|_{L^{p_1/\gamma}(\sigma_1)}^{1/\gamma}  \|M_{\mathcal D}^{\sigma_2}(f_2)^\gamma\|_{L^{p_2/\gamma}(\sigma_2)}^{1/\gamma}\\
&\le \mathcal N  \|f_1 \|_{L^{p_1}(\sigma_1)}   \|f_2\|_{L^{p_2}(\sigma_2)},
\end{align*}
where $M_{\mathcal{D}}^{\sigma}$ denotes the dyadic weighted maximal function defined in \eqref{dyadicmaximal}, and it is bounded on $L^p(X, \sigma)$ for all $p > 1$. 

On the other hand, if \eqref{eq:p01} holds, we have
\begin{align*}
&\quad \left\|\Big(\sum_{Q\in \mathcal S}\mu(Q)^{\an \cdot \gamma} \langle f_1   \rangle_{\sigma_1,Q} \langle   f_2  \rangle_{\sigma_2,Q} \langle \sigma_1\rangle_Q^\gamma \langle\sigma_2\rangle_Q^\gamma \mathbf \chi_Q\Big)^{\frac 1\gamma} \right\|_{L^q(u)}\\
&\le \Big\|\Big(\sum_{Q\in \mathcal S} \mu(Q)^{\an \cdot \gamma} (\langle M_{\gamma,\mathcal D}^{\sigma_1}(f_{1}^{1/\gamma})   \rangle_{\sigma_1,Q})^\gamma (\langle M_{\gamma,\mathcal D}^{\sigma_2}(f_{2}^{1/\gamma})   \rangle_{\sigma_2,Q})^\gamma \langle \sigma_1\rangle_Q^\gamma \langle\sigma_2\rangle_Q^\gamma \mathbf \chi_Q\Big)^{\frac 1\gamma} \Big\|_{L^q(u)}\\
&\le \mathcal N \|  M_{\gamma,\mathcal D}^{\sigma_1}(f_{1}^{1/\gamma})    \|_{L^{p_1}(\sigma_1)} \|  M_{\gamma,\mathcal D}^{\sigma_2}(f_{2}^{1/\gamma})    \|_{L^{p_2}(\sigma_2)}\\
&\lesssim  \mathcal N \| f_{1}^{1/\gamma}    \|_{L^{p_1}(\sigma_1)} \|  f_{2}^{1/\gamma}    \|_{L^{p_2}(\sigma_2)},
\end{align*}
where $M_{\gamma,\mathcal{D}}^{\sigma}(f) = \left( M_{\mathcal{D}}^{\sigma}(f^{\gamma}) \right)^{1/\gamma}$. In the final step, we used the fact that $p > \gamma$, which implies $p_1, p_2 > \gamma$ and ensures the boundedness of the maximal functions.
\end{proof}

Therefore, we simplify the problem to study \eqref{eq:gamma1}. We present the following lemma, which is crucial for proving Theorems \ref{Thm:1}, \ref{Horm.pro_2}, and \ref{Thm:3}.
\begin{lemma}\label{lm:testing}
Let $\gamma>0$. Let $1<p_1, p_2<\infty $ ,$\frac{1}{p}=\frac{1}{p_1}+ \frac{1}{p_2}$, and $\frac{1}{p}-\frac{1}{q}=\eta \in [0,2)$. Let $u$ and $\vec\sigma$ be weights. If $p_1,p_2>\gamma$, Then for any sparse collection $\mathfrak S$,
\begin{equation}\label{eq:testing}
\Big\|\Big(\sum_{Q\in \mathfrak S}\mu(Q)^{\an \cdot \gamma} \langle \sigma_1\rangle_Q^\gamma \langle \sigma_2\rangle_Q^\gamma \mathbf \chi_Q\Big)^{\frac 1\gamma}\Big\|_{L^q(u)}
\le \|u,\vec \sigma\|_{A_{\vec p,q}^{\star}(X)}^{1/q}
\left(\sum_{Q\in \mathfrak S} \langle\sigma_1\rangle^{\frac{q}{p_1}} \langle\sigma_2\rangle^{\frac{q}{p_2}} \mu(Q)^{1+\an \cdot {q}}\right)^{\frac{1}{q}}.
\end{equation} 
and if $q>\gamma$, then there holds
\begin{equation}\begin{split}\label{eq:dualtesting}
\Big\| \sum_{Q\in \mathfrak S} \mu(Q)^{\an \cdot \gamma} \langle \sigma_1\rangle_Q^\gamma \langle \sigma_2 &\rangle_Q^{\gamma-1}\langle u\rangle_Q \mathbf \chi_Q  \Big\|_{L^{(\frac{p_2}\gamma)'}(\sigma_2)} \\
&\le \|u,\vec \sigma\|_{A^{\star}_{\vec p,q}(X)}^{\frac \gamma q}\Big( \sum_{Q\in \mathfrak S} \langle \sigma_1\rangle_Q^{\frac {\gamma (\frac{p_2}{\gamma})'}{p_1}}\langle u\rangle_Q^{(\frac{p_2}\gamma)' (1-\frac \gamma q)} \mu(Q)^{1+\gamma \cdot\an \cdot (p_2/\gamma)'}\Big)^{\frac 1{(\frac{p_2}\gamma)'}}
\end{split}\end{equation}
and
\begin{equation}\begin{split}\label{eq:dualtesting1}
\Big\| \sum_{Q\in \mathfrak S}\mu(Q)^{\an \cdot \gamma} \langle \sigma_1\rangle_Q^{\gamma-1} \langle \sigma_2 &\rangle_Q^{\gamma}\langle u\rangle_Q \mathbf \chi_Q  \Big\|_{L^{(\frac{p_1}\gamma)'}(\sigma_1)}
\\
&\le \|u,\vec \sigma\|_{A^{\star}_{\vec p,q}(X)}^{\frac \gamma q}\Big( \sum_{Q\in \mathfrak S} \langle \sigma_2\rangle_Q^{\frac {\gamma (\frac{p_1}{\gamma})'}{p_2}}\langle w\rangle_Q^{(\frac{p_1}\gamma)' (1-\frac \gamma q)} \mu(Q)^{1+\gamma \cdot\an \cdot (p_1/\gamma)'}\Big)^{\frac 1{(\frac{p_1}\gamma)'}}.
\end{split}\end{equation}
\end{lemma}

\begin{proof}
  We start by proving \eqref{eq:testing}. The right-hand side of \eqref{eq:testing} is independent of $\gamma$, so we can focus on the case where $\gamma$ is small. Specifically, for a fixed $(\vec p, q)$, we consider $\gamma < \min\{q, 1\}$ and $(q/\gamma)' < \max\{p_1, p_2\}$.

Without loss of generality, assume that $(q/\gamma)' < p_1 = \max\{p_1, p_2\}$. Under this assumption, it is straightforward to verify that
\begin{equation}\label{eq:cond1}
0 \leq \gamma - \frac{\gamma p_1'}{p_2'} < 1 \quad \text{and} \quad 0 \leq 1 - \frac{\gamma p_1'}{q} < 1,
\end{equation}
as well as
\begin{equation}\label{eq:cond2}
\gamma - \frac{\gamma p_1'}{p_2'} + 1 - \frac{\gamma p_1'}{q} < 1.
\end{equation}

For clarity, let us define
\[
\lambda_Q = \mu(Q)^{\eta \cdot \gamma} \langle \sigma_1 \rangle_Q^\gamma \langle \sigma_2 \rangle_Q^\gamma u(Q)
\]
throughout this proof. By invoking Proposition \ref{dyadicsum}, we obtain

\begin{align*}
&\quad \left\|\left(\sum_{Q\in \mathfrak S}\mu(Q)^{\an \cdot \gamma}\langle \sigma_1 \rangle_{Q}^{\gamma} \langle \sigma_2 \rangle_{Q}^{\gamma} \chi_Q \right)^{\frac{1}{\gamma}}\right\|_{L^q(u)}\\
&= \left\|\left(\sum_{Q\in \mathfrak S}\mu(Q)^{\an \cdot \gamma}\langle \sigma_1 \rangle_{Q}^{\gamma} \langle \sigma_2 \rangle_{Q}^{\gamma} \chi_Q \right)\right\|^{\frac{1}{\gamma}}_{L^{q/\gamma}(u)}\\
& \overset{\text{(\ref{dyadicsum_})}}{\approx} \left(\sum_{Q\in \mathfrak S} \lambda_Q \left[\frac{1}{u(Q)} \sum_{Q'\subseteq Q} \mu(Q')^{\an \cdot \gamma}\langle \sigma_1 \rangle_{Q'}^{\gamma} \langle \sigma_2 \rangle_{Q'}^{\gamma} {u(Q')} \right]^{\frac{q}{\gamma}-1}   \right)^{\frac{1}{q}}\\
& \leq \|u,\vec \sigma\|_{A^{\star}_{\vec p,q}}^{\frac{p'_1(q-\gamma)}{q^2}} \left(\sum_{Q\in \mathfrak S} \lambda_Q \left[\frac{1}{u(Q)} \sum_{Q'\subseteq Q} \langle u \rangle_{Q'}^{1-\frac{\gamma p'_1}{q}} \langle \sigma_2 \rangle_{Q'}^{\gamma(1-\frac{p'_1}{p'_2})} \mu(Q')^{\an \cdot \gamma + 1} \right]^{\frac{q}{\gamma}-1}   \right)^{\frac{1}{q}}\\
& \overset{\text{(\ref{eq:kolmogorov})}}{\lesssim} \|u,\vec \sigma\|_{A^{\star}_{\vec p,q}}^{\frac{p'_1(q-\gamma)}{q^2}} \left(\sum_{Q\in \mathfrak S} \lambda_Q  \left[\frac{\mu(Q)^{\an \cdot \gamma + 1}}{u(Q)}  \langle u \rangle_{Q}^{1-\frac{\gamma p'_1}{q}} \langle \sigma_2 \rangle_{Q}^{\gamma(1-\frac{p'_1}{p'_2})} \right]^{\frac{q}{\gamma}-1}   \right)^{\frac{1}{q}}\\
&=\|u,\vec \sigma\|_{A^{\star}_{\vec p,q}}^{\frac{p'_1(q-\gamma)}{q^2}}\left(\sum_{Q\in \mathfrak S} \mu(Q)^{1+\an \cdot { q}} \langle\sigma_1\rangle_Q^{\gamma} \langle\sigma_2\rangle_Q^{\gamma + (1-\frac{p'_1}{p'_2})(q-\gamma)} \langle u \rangle_{Q}^{1-\frac{p'_1(q-\gamma)}{q}} \right)^{\frac{1}{q}}\\
& \leq \|u,\vec \sigma\|_{A^{\star}_{\vec p,q}}^{(q - \gamma)p'_1/q^2 + 1/q - (q-\gamma)p'_1/q^2}
\left(\sum_{Q\in \mathfrak S} \langle\sigma_1\rangle^{\frac{q}{p_1}} \langle\sigma_2\rangle^{\frac{q}{p_2}} \mu(Q)^{1+\an \cdot{ q}}\right)^{\frac{1}{q}}\\
& \leq \|u,\vec \sigma\|_{A^{\star}_{\vec p,q}}^{1/q}
\left(\sum_{Q\in \mathfrak S} \langle\sigma_1\rangle^{\frac{q}{p_1}} \langle\sigma_2\rangle^{\frac{q}{p_2}} \mu(Q)^{1+\an \cdot { q}}\right)^{\frac{1}{q}},
\end{align*}
Next, we establish \eqref{eq:dualtesting}, which is equivalent to proving \eqref{eq:dualtesting1} due to symmetry. We will examine the cases where $(q/\gamma)' \geq \max\{p_1, p_2\}$ and $(q/\gamma)' < \max\{p_1, p_2\}$ separately.

Focusing on the case $(q/\gamma)' < \max\{p_1, p_2\}$, we assume without loss of generality that $p_1 > p_2$. By combining \eqref{eq:cond1}, \eqref{eq:cond2}, and Proposition \ref{dyadicsum}, we obtain

 \begin{align*}
& \Big\| \sum_{Q\in \mathfrak S} \mu(Q)^{\eta \cdot \gamma}\langle \sigma_1\rangle_Q^\gamma \langle \sigma_2\rangle_Q^{\gamma-1}\langle u \rangle_Q \mathbf \chi_Q  \Big\|_{L^{(\frac{p_2}\gamma)'}(\sigma_2)}\\
&\approx \Big( \sum_{Q\in \mathfrak S} \lambda_Q   \Big( \frac 1{\sigma_2(Q)}\sum_{Q'\subseteq Q}\mu(Q)^{\eta \cdot \gamma}\langle \sigma_1\rangle_{Q'}^\gamma \langle \sigma_2\rangle_{Q'}^{\gamma } u(Q')   \Big)^{(\frac{p_2}\gamma)'-1}              \Big)^{\frac 1{(\frac{p_2}\gamma)'}}\\
&\le \|u,\vec \sigma\|_{A_{\vec p,q}^{\star}}^{\frac {p_1'\gamma^2}{qp_2}}\Big( \sum_{Q\in \mathfrak S} \lambda_Q   \Big( \frac 1{\sigma_2(Q)}\sum_{Q'\subseteq Q}  \langle \sigma_2\rangle_{Q'}^{\gamma(1-\frac{p_1'}{p_2'}) }  \langle \omega\rangle_{Q'}^{  1-\frac{\gamma p_1'}{p}}   \mu(Q')^{1+\eta \cdot \gamma} \Big)^{(\frac{p_2}\gamma)'-1}   \Big)^{\frac 1{(\frac{p_2}\gamma)'}}\\
&\overset{\eqref{eq:kolmogorov}}{\lesssim} \|u,\vec \sigma\|_{A_{\vec p,q}^{\star}}^{\frac {p_1'\gamma^2}{qp_2}}\Big( \sum_{Q\in \mathfrak S} \lambda_Q  \Big( \frac 1{\sigma_2(Q)}   \langle \sigma_2\rangle_{Q}^{\gamma(1-\frac{p_1'}{p_2'}) }  \langle u\rangle_{Q}^{  1-\frac{\gamma p_1'}{p}}  \mu(Q)^{1+ \eta \cdot \gamma}   \Big)^{(\frac{p_2}\gamma)'-1}              \Big)^{\frac 1{(\frac{p_2}\gamma)'}}\\
&=\hspace{-0.10cm} \|u,\vec \sigma\|_{A_{\vec p,q}^{\star}}^{\frac {p_1'\gamma^2}{qp_2}}\hspace{-0.10cm} \Big( \hspace{-0.10cm} \sum_{Q\in \mathfrak S}\langle \sigma_1\rangle_Q^\gamma \langle \sigma_2\rangle_Q^{\gamma(\frac{p_2}{\gamma})'-(\frac{\gamma p_1'}{p_2'}+1)((\frac{p_2}\gamma)'-1) } \hspace{-0.15cm}  \langle u\rangle_Q^{(\frac {p_2}\gamma)'-\frac{\gamma p_1'}p((\frac {p_2}\gamma)'-1)}\hspace{-0.05cm} \mu(Q)^{1+\gamma \cdot \eta \cdot(p_2/\gamma)'} \Big)^{\frac 1{(\frac{p_2}\gamma)'}}\\
&\le \|u,\vec \sigma\|_{A_{\vec p,q}^{\star}}^{\frac \gamma q}\Big(\sum_{Q\in \mathfrak S} \langle \sigma_1\rangle_Q^{\frac {\gamma (\frac{p_2}{\gamma})'}{p_1}}\langle u\rangle_Q^{(\frac{p_2}\gamma)' (1-\frac \gamma q)} \mu(Q)^{1+\gamma \cdot \eta \cdot(p_2/\gamma)'}\Big)^{\frac 1{(\frac{p_2}\gamma)'}},
 \end{align*}

It remains to consider when $(q/\gamma)'\ge \max\{p_1,p_2\}$. In the current setting, combining $q>\gamma$, we conclude that 
 \[
   0 \leq \gamma-\frac q{p_1'} < 1,\,\, 0 \leq \gamma-\frac q{p_2'} < 1.
 \]
 \[
 \gamma-\frac q{p_1'}+\gamma-\frac q{p_2'}=2\gamma -2q+1<1.
 \]
 Invoking Proposition \ref{dyadicsum} again, we obtain
  \begin{align*}
& \Big\| \sum_{Q\in \mathfrak S} \mu(Q)^{\eta \cdot \gamma}\langle \sigma_1\rangle_Q^\gamma \langle \sigma_2\rangle_Q^{\gamma-1}\langle u\rangle_Q \mathbf \chi_Q  \Big\|_{L^{(\frac{p_2}\gamma)'}(\sigma_2)}\\
&\approx \Big( \sum_{Q\in \mathfrak S} \lambda_Q  \Big( \frac 1{\sigma_2(Q)}\sum_{Q'\subseteq Q}\mu(Q)^{\eta \cdot \gamma}\langle \sigma_1\rangle_{Q'}^\gamma \langle \sigma_2\rangle_{Q'}^{\gamma } u(Q')   \Big)^{(\frac{p_2}\gamma)'-1}              \Big)^{\frac 1{(\frac{p_2}\gamma)'}}\\
&\le \|u,\vec \sigma\|_{A_{\vec p,q}^{\star}}^{\frac {\gamma}{p_2}}\Big( \sum_{Q\in \mathfrak S} \lambda_Q   \Big( \frac 1{\sigma_2(Q)}\sum_{Q'\subseteq Q} \langle \sigma_1\rangle_{Q'}^{\gamma-\frac p{p_1'}} \langle \sigma_2\rangle_{Q'}^{\gamma-\frac p{p_2'} }   \mu(Q')^{\eta \cdot \gamma +1} \Big)^{(\frac{p_2}\gamma)'-1}              \Big)^{\frac 1{(\frac{p_2}\gamma)'}}\\
&\overset{\eqref{eq:kolmogorov}}{\lesssim} \|u,\vec \sigma\|_{A_{\vec p,q}^{\star}}^{\frac { \gamma }{ p_2}}\Big( \sum_{Q\in \mathfrak S} \lambda_Q  \Big( \frac 1{\sigma_2(Q)}  \langle \sigma_1\rangle_{Q }^{\gamma-\frac p{p_1'}} \langle \sigma_2\rangle_{Q}^{\gamma-\frac p{p_2'} }    \mu(Q)^{1+\gamma \cdot \eta \cdot(p_2/\gamma)'}   \Big)^{(\frac{p_2}\gamma)'-1}              \Big)^{\frac 1{(\frac{p_2}\gamma)'}}\\
&= \|u,\vec \sigma\|_{A_{\vec p,q}^{\star}}^{\frac { \gamma }{ p_2}} \Big( \sum_{Q\in \mathfrak S}\langle \sigma_1\rangle_Q^{\frac \gamma{p_2-\gamma}(p_2-\frac p{p_1'}) } \langle \sigma_2\rangle_Q^{\frac 1{p_2-\gamma}(p_2(1-\gamma)-\frac {p\gamma}{p_2'}) }   \langle u\rangle_Q  \mu(Q)^{1+\gamma \cdot \eta \cdot(p_2/\gamma)'}  \Big)^{\frac 1{(\frac{p_2}\gamma)'}}\\
&\le \|u,\vec \sigma\|_{A_{\vec p,q}^{\star}}^{\frac \gamma q}\Big( \sum_{Q\in \mathfrak S} \langle \sigma_1\rangle_Q^{\frac {\gamma (\frac{p_2}{\gamma})'}{p_1}}\langle u\rangle_Q^{(\frac{p_2}\gamma)' (1-\frac \gamma p)} \mu(Q)^{1+\gamma \cdot \eta \cdot(p_2/\gamma)'}\Big)^{\frac 1{(\frac{p_2}\gamma)'}}.\qedhere
 \end{align*}
\end{proof}
Now we are ready to prove our main results.

\begin{proof}[Proof of Theorem~\ref{Thm:1}]
Inspired by \cite{Li2018}, we require the following definitions  in order to advance our estimation.
\begin{itemize}
  \item Let $\mathcal{F}_i$ be the stopping family starting at $Q_0$, defined by the stopping condition,
  \begin{equation*}
    \operatorname{ch}_{\mathcal{F}_i}\left(F_i\right):=\left\{F_i^{\prime} \in \mathcal{S}: F_i^{\prime} \subseteq F_i \text { maximal such that }\left\langle f_i\right\rangle_{F_i^{\prime}}^{\sigma_i^{\prime}}>2\left\langle f_i\right\rangle_{F_i}^{\sigma_i}\right\} .
  \end{equation*}
  Since
  \begin{equation*}
    \sum_{F_i'\in \rm{ch}_{\mathcal{F}_i}(\mathit{F_i})} \sigma_i(F_i')\leq \frac{1}{2} \frac{\sum\limits_{F_i'\in \rm{ch}_{\mathcal{F}_i}(\mathit{F_i})} \int_{F_i'} f d\sigma}{\int_{F_i} f d\sigma} \sigma_i(F_i) \leq \frac{1}{2}  \sigma_i(F_i),
  \end{equation*}
  each collection $\mathcal{F}_i$ is $\sigma_i$-sparse.
  \item The $\mathcal{F}_i$-stopping parent $\pi_{\mathcal{F}_i}(Q)$ of a cube $Q$ is defined by
  $$
  \pi_{\mathcal{F}_i}(Q):=\{F_i \in\mathcal{F}_i : \text{$F_i$ minimal such that $F_i\supseteq Q$}\}.
  $$
\end{itemize}
By the stopping condition, for every cube $Q$, we have $\langle f_i \rangle^{\sigma_i}_{Q} \leq 2 \langle f_i \rangle^{\sigma_i}_{\pi_{\mathcal{F}_i}(Q)}$. Furthermore, based on the construction above, we conclude that
\begin{align}\label{eq:2.2}
  \sum_{F_i \in \mathcal{F}_i}\left(\langle f\rangle_{F_i}^{\gamma_1}\right)^{p_1} \sigma_1(F_i) \lesssim\|f\|_{L^{p_1}\left(\sigma_1\right)}^{p_1} .
\end{align}
Let $\mathcal{H}$ denote the stopping family corresponding to $h$ and the weight $u$, satisfying the relevant properties.

First, we consider the case where $q > \gamma$, denoting $t = q / \gamma$. By Lemma \ref{lm:equivalent}, we obtain
\begin{align*}
& \Big\|\Big(\sum_{Q\in \mathcal S}\mu(Q)^{\an} \langle f_1   \rangle_{\sigma_1,Q} \langle   f_2  \rangle_{\sigma_2,Q} \langle \sigma_1\rangle_Q^\gamma \langle\sigma_2\rangle_Q^\gamma \mathbf \chi_Q\Big) \Big\|_{L^{q/\gamma}(u)} \\&=\sup_{\|h\|_{L^{t'}(u)}=1}  \sum_{Q\in \mathcal S} \mu(Q)^{\an} \langle f_1   \rangle_{\sigma_1,Q} \langle   f_2  \rangle_{\sigma_2,Q} \langle \sigma_1\rangle_Q^\gamma \langle\sigma_2\rangle_Q^\gamma  \int_Q h \textup{d}u\\
&= \sup_{\|h\|_{L^{t'}(u)}=1} \sum_{Q\in \mathcal S} \mu(Q)^{\an}  \langle f_1\rangle_{\sigma_1,Q}   \langle f_2\rangle_{\sigma_2,Q}  \langle h\rangle_{u,Q} \langle \sigma_1\rangle_Q^\gamma \langle\sigma_2\rangle_Q^\gamma u(Q).
\end{align*}

By integrating rearrangements with stopping parent properties and eliminating the supremum, we conclude that
\begin{align*}
&\sum_{Q\in \mathcal S} \mu(Q)^{\an} \langle f_1\rangle_{\sigma_1,Q} \langle f_2\rangle_{\sigma_2,Q}  \langle h\rangle_{u,Q} \langle \sigma_1\rangle_Q^\gamma \langle\sigma_2\rangle_Q^\gamma u(Q)\\
&= \Bigg( \sum_{F_1\in{\mathcal F}_1} \sum_{\substack{F_2\in {\mathcal F}_2\\ F_2\subseteq F_1}} \sum_{\substack{H\in \mathcal H\\ H\subseteq F_2}}\sum_{\substack{Q \in\mathcal S\\ \pi(Q)=(F_1,F_2, H)}} + \sum_{F_2\in {\mathcal F}_2} \sum_{\substack{F_1\in {\mathcal F}_1\\ F_1\subseteq F_2}} \sum_{\substack{H\in \mathcal H\\ H\subseteq F_1}}\sum_{\substack{Q \in\mathcal S\\ \pi(Q)=(F_1,F_2, H)}}   \\
&+\sum_{F_1\in {\mathcal F}_1} \sum_{\substack{H\in \mathcal H\\ H\subseteq F_1}} \sum_{\substack{F_2\in {\mathcal F}_2\\ F_2\subseteq H}}\sum_{\substack{Q \in\mathcal S\\ \pi(Q)=(F_1,F_2, H)}} +\sum_{F_2\in{\mathcal F}_2} \sum_{\substack{H\in \mathcal H\\ H\subseteq F_2}} \sum_{\substack{F_1\in {\mathcal F}_1\\ F_1\subseteq H}}\sum_{\substack{Q \in\mathcal S\\ \pi(Q)=(F_1,F_2, H)}}\\
&+\sum_{H\in\mathcal H} \sum_{\substack{F_1\in {\mathcal F}_1\\ F_1\subseteq H}} \sum_{\substack{F_2\in {\mathcal F}_2\\ F_2\subseteq F_1}}\sum_{\substack{Q \in\mathcal S\\ \pi(Q)=(F_1,F_2, H)}} + \sum_{H\in\mathcal H} \sum_{\substack{F_2\in {\mathcal F}_2\\ F_2\subseteq H}} \sum_{\substack{F_1\in {\mathcal F}_1\\ F_1\subseteq F_2}}\sum_{\substack{Q \in\mathcal S\\ \pi(Q)=(F_1,F_2, H)}}\Bigg) \\ &\times \langle f_1\rangle_{\sigma_1,Q} \langle f_2\rangle_{\sigma_2,Q} \langle h\rangle_{u,Q}\lambda_Q\\
&:= I+I'+II+II'+III+III',
\end{align*}
where $\pi(Q) = F_i$ means that $\pi_{\mathcal{F}_i}(Q)=F_i$, for all $i=1,2$ and $\pi_{\mathcal{H}}(Q)=H$ and
\[
\lambda_Q:= \mu(Q)^{\an \cdot \gamma} \langle \sigma_1\rangle_Q^\gamma \langle\sigma_2\rangle_Q^\gamma u(Q).
\]

Due to symmetry, it suffices to estimate $I$, $II$, and $III$. We start with estimating $I$. Using the inequality 
$
\langle f_i \rangle^{\sigma_i}_{Q} \leq 2 \langle f_i \rangle^{\sigma_i}_{\pi_{\mathcal{F}_i}(Q)},
$ 
we obtain
\begin{align*}
I&\le \sum_{F_1\in{\mathcal F}_1} \sum_{\substack{F_2\in {\mathcal F}_2\\ F_2\subseteq F_1}} \sum_{\substack{H\in \mathcal H\\ H\subseteq F_2}}\sum_{\substack{Q \in\mathcal S\\ \pi(Q)=(F_1,F_2, H)}} \langle f_1\rangle_{\sigma_1,Q} \langle f_2\rangle_{\sigma_2,Q}  \langle h\rangle_{u,Q} \lambda_Q\\
&\le 8\sum_{F_1\in{\mathcal F}_1}  \langle f_1\rangle_{\sigma_1,F_1}  \sum_{\substack{F_2\in {\mathcal F}_2\\ F_2\subseteq F_1}}  \langle f_2\rangle_{\sigma_2,F_2} \sum_{\substack{H\in \mathcal H\\ H\subseteq F_2}} \langle h\rangle_{u,H}\sum_{\substack{Q \in\mathcal S\\ \pi(Q)=(F_1,F_2, H)}} \lambda_Q\\
&\lesssim  \sum_{F_1\in{\mathcal F}_1}  \langle f_1\rangle_{\sigma_1,F_1}  \sum_{\substack{F_2\in {\mathcal F}_2\\ F_2\subseteq F_1}}  \langle f_2\rangle_{\sigma_2,F_2}   \int \Big(\sup_{\substack{H'\in \mathcal H\\ \pi_{{\mathcal F}_2}(H')=F_2} } \langle h\rangle_{u,H'} \mathbf \chi_{H'} \Big) \\ &\times\sum_{\substack{H\in \mathcal H\\ H\subseteq F_2}} \sum_{\substack{Q \in\mathcal S\\ \pi(Q)=(F_1,F_2, H)}} \frac{\lambda_Q}{u(Q)} \mathbf \chi_Q d u\\
&\le \sum_{F_1\in{\mathcal F}_1}  \langle f_1\rangle_{\sigma_1,F_1}  \sum_{\substack{{F_2}\in {\mathcal F}_2\\ \pi_{{\mathcal F}_1}(F_2)= F_1}}  \langle f_2\rangle_{\sigma_2,F_2}   \Big\|\sum_{\substack{H\in \mathcal H\\ \pi_{{\mathcal F}_2}(H)= F_2}} \sum_{\substack{Q \in\mathcal S\\ \pi(Q)=(F_1,F_2, H)}}  \frac{\lambda_Q}{u(Q)}\mathbf \chi_Q \Big\|_{L^t(u)}\\
&\qquad\times \Big\|\sup_{\substack{H'\in \mathcal H\\ \pi_{{\mathcal F}_2}(H')=F_2} } \langle h\rangle_{u,H'} \mathbf \chi_{H'}  \Big\|_{L^{t'}(u)}\\
&\leq \Big(\hspace{-0.2cm}\sum_{F_1\in{\mathcal F}_1}  \sum_{\substack{F_2\in {\mathcal F}_2\\ \pi_{{\mathcal F}_2}(F_2)= F_1}}  \hspace{-0.5cm}(\langle f_1\rangle_{\sigma_1,F_1} \langle f_2\rangle_{\sigma_2,F_2} )^t  \Big\|\hspace{-0.3cm}\sum_{\substack{H\in \mathcal H\\ \pi_{{\mathcal F}_2}(H)= F_2}} \sum_{\substack{Q \in\mathcal S\\ \pi(Q)=(F_1,F_2, H)}} \hspace{-0.7cm} \frac{\lambda_Q}{u(Q)} \mathbf \chi_Q \Big\|_{L^t(u)}^t\Big)^{1/t}\\
&\qquad \times \Big(\sum_{F_1\in{\mathcal F}_1}  \sum_{\substack{F_2\in {\mathcal F}_2\\\pi_{{\mathcal F}_1}(F_2)= F_1}} \sum_{\substack{H'\in \mathcal H\\ \pi_{{\mathcal F}_2}(H')=F_2} } (\langle h\rangle_{u,H'})^{t'} u(H') \Big )^{1/{t'}}\\
&\lesssim \hspace{-0.1cm}\Big(\sum_{F_1\in{\mathcal F}_1} \hspace{-0.2cm} \sum_{\substack{F_2\in {\mathcal F}_2\\ \pi_{{\mathcal F}_1}(F_2)= F_1}}  \hspace{-0.5cm}(\langle f_1\rangle_{\sigma_1,F_1} \langle f_2\rangle_{\sigma_2,F_2} )^t  \Big\|\hspace{-0.2cm}\sum_{\substack{H\in \mathcal H\\ \pi_{\mathcal F_2}(H)= F_2}} \sum_{\substack{Q \in\mathcal S\\ \pi(Q)=(F_1,F_2, H)}} \hspace{-0.7cm}\frac{\lambda_Q}{u(Q)} \mathbf \chi_Q \Big\|_{L^t(u)}^t\Big)^{1/t}.
\end{align*}
By \eqref{eq:testing}, we have
\begin{align*}
  \Big\| \sum_{\substack{Q\in \mathcal S\\ \pi_{\mathcal F_2}(Q)= F_2}}\frac{\lambda_Q}{u(Q)}\mathbf \chi_Q\Big\|_{L^{q/\gamma}(u)}
 &\le \|u,\vec \sigma\|_{A_{\vec p,q}^{\star}}^{\frac \gamma q}\Big( \sum_{\substack{Q\in \mathcal S\\ \pi_{\mathcal F_2}(Q)= F_2}} \langle\sigma_1\rangle_Q^{\frac q{p_1}} \langle \sigma_2\rangle_Q^{\frac q{p_2}} \mu(Q)^{1+\an \cdot q} \Big)^{\gamma/q}.
\end{align*}
Therefore,
\begin{align*}
I&\le \|u,\vec \sigma\|_{A_{\vec{p},q}^{\star}}^{\frac \gamma q}  \Big(   \sum_{F_1\in{\mathcal F}_1}  \sum_{\substack{F_2\in {\mathcal F}_2\\ \pi_{{\mathcal F}_1}(F_2)= F_1}}  (\langle f_1\rangle_{\sigma_1,F_1} \langle f_2\rangle_{\sigma_2,F_2} )^t     \sum_{\substack{Q\in \mathcal S\\ \pi_{\mathcal F_2}(Q)= F_2}} \langle\sigma_1\rangle_Q^{\frac q{p_1}} \langle \sigma_2\rangle_Q^{\frac q{p_2}} \mu(Q)^{1+\an \cdot q}    \Big)^{1/t}\\
&\le \|u,\vec \sigma\|_{A_{\vec{p},q}^{\star}}^{\frac \gamma q}  \Big(   \sum_{F_1\in{\mathcal F}_1}  \sum_{\substack{F_2\in {\mathcal F}_2\\ \pi_{{\mathcal F}_1}(F_2)= F_1}}  (\langle f_1\rangle_{\sigma_1,F_1} \langle f_2\rangle_{\sigma_2,F_2} )^t     \\
&\times\Big(\sum_{\substack{Q\in \mathcal S\\ \pi_{\mathcal F_2}(Q)= F_2}} \langle\sigma_1\rangle_Q \mu(Q)\Big)^{\frac q{p_1}}  \Big(\sum_{\substack{Q\in \mathcal S\\ \pi_{\mathcal F_2}(Q)= F_2}} \langle\sigma_2\rangle_Q \mu(Q)\Big)^{\frac q{p_2}}   \Big)^{1/t}\\
&\le \|u,\vec \sigma\|_{A_{\vec{p},q}^{\star}}^{\frac \gamma q}  \Big( \sum_{F_1\in{\mathcal F}_1} (\langle f_1\rangle_{\sigma_1,F_1})^{p_1/\gamma}\sum_{\substack{F_2\in {\mathcal F}_2\\ \pi_{{\mathcal F}_1}(F_2)= F_1}} \sum_{\substack{Q\in \mathcal S\\ \pi_{\mathcal F_2}(Q)= F_2}} \langle\sigma_1\rangle_Q \mu(Q)  \Big)^{\frac \gamma{p_1}}\\
&\times \Big( \sum_{F_2\in{\mathcal F}_2} (\langle f_2\rangle_{\sigma_2,F_2})^{p_2/\gamma} \sum_{\substack{Q\in \mathcal S\\ \pi_{\mathcal F_2}(Q)= F_2}} \langle\sigma_2\rangle_Q \mu(Q)  \Big)^{\frac \gamma{p_2}}\\
&\le \|u,\vec \sigma\|_{A_{\vec{p},q}^{\star}}^{\frac \gamma q}(\prod_{i=1}^2 [\sigma_i]_{A_\infty}^{\frac \gamma{p_i}}) \|    f_1 \|_{L^{p_1/\gamma}(\sigma_1)}\|f_2\|_{L^{p_2/\gamma}(\sigma_2)}.
\end{align*}

Next we set $
\lambda'_Q:= \mu(Q)^{\an \cdot \gamma} \langle \sigma_1\rangle_Q^\gamma \langle\sigma_2\rangle_Q^{\gamma - 1} \langle u\rangle_Q\sigma_2(Q).
$ To estimate $II$, as shown above, it is straightforward to demonstrate
\begin{align*}
  II &\lesssim \sum_{F_1\in{\mathcal F}_1}  \langle f_1\rangle_{\sigma_1,F_1}  \sum_{\substack{H\in \mathcal H\\ H\subseteq F_1}} \langle h\rangle_{u,H}\sum_{\substack{F_2\in {\mathcal F}_2\\ F_2\subseteq H}}  \langle f_2\rangle_{\sigma_2,F_2}\sum_{\substack{Q \in\mathcal S\\ \pi(Q)=(F_1,F_2, H)}} \lambda'_Q\\
  &\le \sum_{F_1\in{\mathcal F}_1}  \langle f_1\rangle_{\sigma_1,F_1}  \sum_{\substack{H\in \mathcal H\\ \pi_{{\mathcal F}_1}(H)= F_1}}  \langle h\rangle_{u,H}   \Big\|\sum_{\substack{\substack{F_2\in {\mathcal F}_2 \\ \pi_{{\mathcal H}}(F_2)=H}}} \sum_{\substack{Q \in\mathcal S\\ \pi(Q)=(F_1,F_2, H)}}  \frac{\lambda'_Q}{\sigma_2(Q)}\mathbf \chi_Q \Big\|_{L^{(p_2/\gamma)'}(\sigma_2)}\\
&\qquad\times \Big\|\sup_{\substack{F'_2\in {\mathcal F}_2 \\ \pi_{{\mathcal H}}(F'_2)=H} } \langle f_2\rangle_{\sigma_2,F'_2} \mathbf \chi_{F'_2}  \Big\|_{L^{p_2/\gamma}(\sigma_2)}\\
&\leq \Big(\hspace{-0.2cm}\sum_{F_1\in{\mathcal F}_1}  \sum_{{\substack{H\in \mathcal H\\ \pi_{{\mathcal F}_1}(H)= F_1}}}  \hspace{-0.5cm}(\langle f_1\rangle_{\sigma_1,F_1} \langle h\rangle_{u,H} )^{(p_2/\gamma)'}  \Big\|\hspace{-0.3cm}\sum_{\substack{F_2\in {\mathcal F}_2 \\ \pi_{{\mathcal H}}(F_2)=H}} \sum_{\substack{Q \in\mathcal S\\ \pi(Q)=(F_1,F_2, H)}} \hspace{-0.7cm} \frac{\lambda'_Q}{\sigma_2(Q)} \mathbf \chi_Q \Big\|_{L^{(p_2/\gamma)'}(\sigma_2)}^{(p_2/\gamma)'}\Big)^{1/(p_2/\gamma)'}\\
&\qquad \times \Big(\sum_{F_1\in{\mathcal F}_1}  \sum_{{\substack{H\in \mathcal H\\ \pi_{{\mathcal F}_1}(H)= F_1}}} \sum_{\substack{F'_2\in {\mathcal F}_2 \\ \pi_{{\mathcal H}}(F'_2)=H} } (\langle f_2\rangle_{\sigma_2,F'_2})^{{(p_2/\gamma)}}\sigma_2(F'_2) \Big )^{1/{(p_2/\gamma)}}\\
&\leq \Big(\hspace{-0.2cm}\sum_{F_1\in{\mathcal F}_1}  \sum_{{\substack{H\in \mathcal H\\ \pi_{{\mathcal F}_1}(H)= F_1}}}  \hspace{-0.5cm}(\langle f_1\rangle_{\sigma_1,F_1} \langle h\rangle_{u,H} )^{(p_2/\gamma)'}  \Big\|\hspace{-0.3cm}\sum_{\substack{Q \in\mathcal S \\ \pi_{{\mathcal H}}(Q)=H}}  \hspace{-0.3cm} \frac{\lambda'_Q}{\sigma_2(Q)} \mathbf \chi_Q \Big\|_{L^{(p_2/\gamma)'}(\sigma_2)}^{(p_2/\gamma)'}\Big)^{1/(p_2/\gamma)'} \hspace{-0.3cm}\cdot  \|f_2\|_{L^{p_2/\gamma}(\sigma_2)}.\\
\end{align*}
Thus, it follows from \eqref{eq:testing} and \eqref{eq:2.2} that
\begin{align*}
  II&\le \|u,\vec \sigma\|_{A_{\vec{p},q}^{\star}}^{\frac \gamma q}  \Big(   \sum_{F_1\in{\mathcal F}_1}  \sum_{{\substack{H\in \mathcal H\\ \pi_{{\mathcal F}_1}(H)= F_1}}}  (\langle f_1\rangle_{\sigma_1,F_1} \langle h\rangle_{u,H} )^{(p_2/\gamma)'} \\  
  & \quad\times \sum_{\substack{Q \in\mathcal S \\ \pi_{{\mathcal H}}(Q)=H}} \langle \sigma_1\rangle_Q^{\frac {\gamma (\frac{p_2}{\gamma})'}{p_1}}\langle u\rangle_Q^{(\frac{p_2}\gamma)' (1-\frac \gamma q)} \mu(Q)^{1+\gamma \cdot\an \cdot (p_2/\gamma)'}    \Big)^{1/(p_2/\gamma)'} \cdot \|f_2\|_{L^{p_2/\gamma}(\sigma_2)}\\
  &\le \|u,\vec \sigma\|_{A_{\vec{p},q}^{\star}}^{\frac \gamma q}  \Big(   \sum_{F_1\in{\mathcal F}_1}  \sum_{{\substack{H\in \mathcal H\\ \pi_{{\mathcal F}_1}(H)= F_1}}}  (\langle f_1\rangle_{\sigma_1,F_1} \langle h\rangle_{u,H} )^{(p_2/\gamma)'}     \\
  &\times\Big(\sum_{\substack{Q \in\mathcal S \\ \pi_{{\mathcal H}}(Q)=H}} \langle\sigma_1\rangle_Q \mu(Q)\Big)^{\frac {\gamma (\frac{p_2}{\gamma})'}{p_1}}  \Big(\sum_{\substack{Q \in\mathcal S \\ \pi_{{\mathcal H}}(Q)=H}} \langle u \rangle_Q \mu(Q)\Big)^{(\frac{p_2}\gamma)' (1-\frac \gamma q)}   \Big)^{1/(p_2/\gamma)'} \cdot \|f_2\|_{L^{p_2/\gamma}(\sigma_2)}\\
  &\le \|u,\vec \sigma\|_{A_{\vec{p},q}^{\star}}^{\frac \gamma q}  \Big( \sum_{F_1\in{\mathcal F}_1} (\langle f_1\rangle_{\sigma_1,F_1})^{p_1/\gamma}\sum_{{\substack{H\in \mathcal H\\ \pi_{{\mathcal F}_1}(H)= F_1}}} \sum_{\substack{Q \in\mathcal S \\ \pi_{{\mathcal H}}(Q)=H}} \langle\sigma_1\rangle_Q \mu(Q)  \Big)^{\frac \gamma{p_1}}\\
  &\times \Big( \sum_{H\in{\mathcal H}} (\langle h\rangle_{u,H})^{1-\frac{\gamma}{q}} \sum_{\substack{Q \in\mathcal S \\ \pi_{{\mathcal H}}(Q)=H}}  \langle u \rangle_Q  \mu(Q)  \Big)^{1/(1-\frac{\gamma}{q})}\cdot\|f_2\|_{L^{p_2/\gamma}(\sigma_2)}\\
  &\le \|u,\vec \sigma\|_{A_{\vec{p},q}^{\star}}^{\frac \gamma q}[u]_{A_\infty}^{1-\frac \gamma {q}}[\sigma_1]_{A_\infty}^{\frac \gamma {p_1}}  \|f_1\|_{L^{p_1/\gamma}(\sigma_2)}\|f_2\|_{L^{p_2/\gamma}(\sigma_2)}.
  \end{align*}
Note that to prove $III$, we can repeat the steps from the proof of $II$. By combining \eqref{eq:dualtesting1}, we obtain that
 \begin{align*}
 III \le \|u,\vec \sigma\|_{A_{\vec{p},q}^{\star}}^{\frac \gamma q}[u]_{A_\infty}^{1-\frac \gamma {q}}[\sigma_2]_{A_\infty}^{\frac \gamma {p_1}}  \|f_1\|_{L^{p_1/\gamma}(\sigma_2)}\|f_2\|_{L^{p_2/\gamma}(\sigma_2)}.
 \end{align*}

Beyound that, it remains to consider the case $q\le \gamma$. By Lemma \ref{lm:equivalent}, we have
\begin{align*}
 &\quad \Big\|\Big(\sum_{Q\in \mathcal S} \mu(Q)^{\an \cdot \gamma} \langle f_1\rangle_{\sigma_1,Q}   \langle f_2\rangle_{\sigma_2,Q}   \langle\sigma_1\rangle_Q^\gamma \langle\sigma_2\rangle_Q^\gamma \mathbf \chi_Q\Big)^{\frac 1\gamma} \Big\|_{L^q(u)}^\gamma\\
&\lesssim  \Big\|\Big(\sum_{F_1\in \mathcal F_1}\langle f_1\rangle_{\sigma_1,F_1}\sum_{F_2\in \mathcal F_2}\langle f_2\rangle_{\sigma_2,F_2}\sum_{\substack{Q\in \mathcal S\\ \pi(Q)=(F_1, F_2)}}        \mu(Q)^{\an \cdot \gamma}\langle\sigma_1\rangle_Q^\gamma \langle\sigma_2\rangle_Q^\gamma \mathbf \chi_Q\Big)^{\frac 1\gamma}\Big\|_{L^q(u)}^\gamma\\
&\le  \Big(\sum_{F_1\in \mathcal F_1}(\langle f_1\rangle_{\sigma_1,F_1})^t\sum_{F_2\in \mathcal F_2}(\langle f_2\rangle_{\sigma_2,F_2})^t \Big\|\sum_{\substack{Q\in \mathcal S\\ \pi(Q)=(F_1, F_2)}} \mu(Q)^{\an \cdot \gamma}       \langle\sigma_1\rangle_Q^\gamma \langle\sigma_2\rangle_Q^\gamma \mathbf \chi_Q\Big\|_{L^t(u)}^t\Big)^{\frac 1t}\\
&\lesssim \Big( \sum_{F_1\in \mathcal F_1}(\langle f_1\rangle_{\sigma_1,F_1})^t\sum_{\substack{F_2\in \mathcal F_2\\ F_2\subseteq F_1}}(\langle f_2\rangle_{\sigma_2,F_2})^t \Big\|\sum_{\substack{Q\in \mathcal S\\ \pi(Q)=(F_1, F_2)}}  \mu(Q)^{\an \cdot \gamma}      \langle\sigma_1\rangle_Q^\gamma \langle\sigma_2\rangle_Q^\gamma \mathbf \chi_Q\Big\|_{L^t(u)}^t\Big)^{\frac 1t}\\
&\quad +\Big(\sum_{F_2\in \mathcal F_2}(\langle f_2\rangle_{\sigma_2,F_2})^t \sum_{\substack{F_1\in \mathcal F_1\\F_1\subseteq F_2}}(\langle f_1\rangle_{\sigma_1,F_1})^t\Big\|\sum_{\substack{Q\in \mathcal S\\ \pi(Q)=(F_1, F_2)}}  \mu(Q)^{\an \cdot \gamma}    \langle\sigma_1\rangle_Q^\gamma \langle\sigma_2\rangle_Q^\gamma \mathbf \chi_Q\Big\|_{L^t(u)}^t\Big)^{\frac 1t}.
\end{align*}
The desired estimate follows from the previous arguments. This concludes the proof.
\end{proof}

\subsection{Proof of Theorem \ref{Horm.pro_2}}
~~

The proof of Theorem \ref{Horm.pro_2} follows a similar approach to the previous theorem. Thus, we present only the key steps and omit others. We will check the estimate for $I$, as the other terms are analogous. By \eqref{eq:testing}, we can obtain
 \begin{align*}
   &\quad \Big\| \sum_{\substack{Q\in \mathcal S\\Q\subseteq F_2}} \mu(Q)^{\an \cdot \gamma} \langle \sigma_1\rangle_Q^\gamma\langle\sigma_2\rangle_Q^\gamma \mathbf \chi_Q\Big\|_{L^q(u)}\\
 &\le \|u,\vec \sigma\|_{A_{\vec p,q}^{\star}}^{\frac \gamma q}\Big( \sum_{\substack{Q\in \mathcal S\\Q\subseteq F_2}} \langle\sigma_1\rangle_Q^{\frac q{p_1}} \langle \sigma_2\rangle_Q^{\frac q{p_2}} \mu(Q)^{1+ \an \cdot q} \Big)^{\gamma/q}\\
 &\lesssim \|u,\vec \sigma\|_{A_{\vec p,q}^{\star}}^{\frac \gamma q} \Big( \int_{F_2} \prod_{i=1}^2 M(\mathbf \chi_{F_2}\sigma_i)^{q/{p_i}}d \mu               \Big)^{\gamma/q}\\
 &\le \|u,\vec \sigma\|_{A_{\vec p,q}^{\star}}^{\frac \gamma q}[\vec \sigma]_{W_{\vec p,q}^\infty}^{\gamma/q} \Big( \int_{F_2} \prod_{i=1}^2 \sigma_i^{q/{p_i}}d \mu   \Big)^{\gamma/q}.
 \end{align*}
Moreover, we have
  \begin{align*}
 I&\le \|u,\vec \sigma\|_{A^{\star}_{\vec{p},q}}^{\frac \gamma q} [\vec \sigma]_{W_{\vec p,q}^\infty}^{\gamma/q}\Big(   \sum_{F_1\in{\mathcal F}_1}  \sum_{\substack{F_2\in {\mathcal F}_2\\ \pi_{{\mathcal F}_1}(F_2)= F_1}}  (\langle f_1\rangle_{\sigma_1,F_1} \langle f_2\rangle_{\sigma_2,F_2} )^t      \int_{F_2} \prod_{i=1}^2 \sigma_i^{q/{p_i}}d \mu    \Big)^{1/t}\\
 &\le  \hspace{-0.05cm}\|u,\vec \sigma\|_{A^{\star}_{\vec{p},q}}^{\frac \gamma q} [\vec \sigma]_{W_{\vec p,q}^\infty}^{\gamma/q}\Big( \hspace{-0.10cm}         \int  \sum_{F_1\in{\mathcal F}_1} \hspace{-0.20cm}   (\langle f_1\rangle_{\sigma_1,F_1} )^t \mathbf \chi_{F_1} \hspace{-0.20cm}   \sum_{\substack{F_2\in {\mathcal F}_2\\ \pi_{{\mathcal F}_1}(F_2)= F_1}} \hspace{-0.20cm}   (\langle f_2\rangle_{\sigma_2,F_2} )^t  \mathbf \chi_{F_2}  \prod_{i=1}^2 \sigma_i^{q/{p_i}}d \mu               \Big)^{1/t}\\
 &\le \|u,\vec \sigma\|_{A^{\star}_{\vec{p},q}}^{\frac \gamma q}[\vec \sigma]_{W_{\vec p,q}^\infty}^{\gamma/q}\Big(          \int  M_{\mathcal D}^{\sigma_1} (f_1)^t  M_{\mathcal D}^{\sigma_2}(f_2)^t   \prod_{i=1}^2 \sigma_i^{q/{p_i}}d \mu               \Big)^{1/t}\\
 &\le \|u,\vec \sigma\|_{A^{\star}_{\vec{p},q}}^{\frac \gamma q} [\vec \sigma]_{W_{\vec p,q}^\infty}^{\gamma/q}  \| M_{\mathcal D}^{\sigma_1} (f_1)\|_{L^{p_1/\gamma}(\sigma_1)}\cdot \| M_{\mathcal D}^{\sigma_2} (f_2)\|_{L^{p_2/\gamma}(\sigma_2)}\\
 &\lesssim \|u,\vec \sigma\|_{A^{\star}_{\vec{p},q}}^{\frac \gamma q} [\vec \sigma]_{W_{\vec p,q}^\infty}^{\gamma /p} \| f_1\|_{L^{p_1/\gamma}(\sigma_1)}\cdot \|f_2\|_{L^{p_2/\gamma}(\sigma_2)}. \qedhere
 \end{align*}

\subsection{Proof of Theorem \ref{Thm:3}}
~~

Again, the proof of Theorem \ref{Thm:3} follows a similar approach to that of the previous theorem. Therefore, we present only the key steps and omit the rest.
\begin{proof}[Proof of Theorem \ref{Thm:3}]
  Similarly, we focus solely on estimating $I$.  
  Using \eqref{eq:testing} again, we have
  \begin{align*}
  & \Big\| \mu(Q)^{\eta \cdot \gamma}\sum_{\substack{Q\in \mathcal S\\ \pi(Q)=F_2}} \langle \sigma_1\rangle_Q^\gamma\langle\sigma_2\rangle_Q^\gamma \mathbf \chi_Q\Big\|_{L^t(u)}\\
   &\le \|u,\vec \sigma\|_{A^{\star}_{\vec p,q}}^{\frac \gamma q}\Big( \sum_{\substack{Q\in \mathcal S\\ \pi(Q)=F_2}} \langle\sigma_1\rangle_Q^{\frac q{p_1}} \langle \sigma_2\rangle_Q^{\frac q{p_2}} \mu(Q)^{1+\eta \cdot \gamma} \Big)^{\gamma/q}\\
  &\le \|u,\vec \sigma\|_{A^{\star}_{\vec p,q}}^{\frac \gamma q}[\vec \sigma]_{H_{\vec p,q}^\infty}^{\frac \gamma q}\Big( \sum_{\substack{Q\in \mathcal S\\ \pi(Q)=F_2}} \prod_{i=1}^2 \exp\Big( \dashint_Q \log \sigma_i  \Big)^{\frac p{p_i}} \mu(Q)^{1+\eta \cdot \gamma} \Big)^{\gamma /q}\\
  &\le \|u,\vec \sigma\|_{A^{\star}_{\vec p,q}}^{\frac \gamma q}[\vec \sigma]_{H_{\vec p,q}^\infty}^{\frac \gamma q}\prod_{i=1}^2 \Big( \sum_{\substack{Q\in \mathcal S\\ \pi(Q)=F_2}} \exp\Big( \dashint_Q \log \sigma_i  \Big)  \mu(Q)^{1+\eta \cdot \gamma} \Big)^{\gamma/{p_i}}\\
  &\lesssim \|u,\vec \sigma\|_{A^{\star}_{\vec p,q}}^{\frac \gamma q}[\vec \sigma]_{H_{\vec p,q}^\infty}^{\frac \gamma q}  \Big( \sum_{\substack{Q\in \mathcal S\\ \pi(Q)=F_2}} \exp\Big( \dashint_Q \log \sigma_1  \Big)  \mu(Q)^{1+\eta \cdot \gamma} \Big)^{\gamma /{p_1}}\|M_0 (\mathbf \chi_{F_2}\sigma_2)\|_{L^1}^{\frac \gamma {p_2}}\\
  &\le \|u,\vec \sigma\|_{A^{\star}_{\vec p,q}}^{\frac \gamma q}[\vec \sigma]_{H_{\vec p,q}^\infty}^{\frac \gamma q} \sigma_2(F)^{\frac \gamma {p_2}} \Big( \sum_{\substack{Q\in \mathcal S\\ \pi(Q)=F_2}} \exp\Big( \dashint_Q \log \sigma_1  \Big)  \mu(Q)^{1+\eta \cdot \gamma} \Big)^{\gamma /{p_1}},
  \end{align*}
  where
  \begin{equation}
  M_0 (f) := \sup_Q \exp{\left( \dashint_Q \log{|f|}\right)}\mathbf \chi_Q,
  \end{equation}
  is the logarithmic maximal function. 

  According to \cite[Lemma 2.1]{HP}, this maximal function is bounded on $L^p$ for $p \in (0, \infty)$.  
  Furthermore, we obtain
\begin{align*}
I&\le \|u,\vec \sigma\|_{A^{\star}_{\vec{p},q}}^{\frac \gamma q} [\vec \sigma]_{H_{\vec p,q}^\infty}^{\gamma /q}\Big(   \sum_{F_1\in{\mathcal F}_1} (\langle f_1\rangle_{F_1}^{\sigma_1})^t \sum_{\substack{F_2\in {\mathcal F}_2\\ \pi_{{\mathcal F}_1}(F_2)= F_1}}  ( \langle f_2\rangle_{F_2}^{\sigma_2} )^t \sigma_2(F)^{\frac q{p_2}}   \\
 &\quad\times \Big( \sum_{\substack{Q\in \mathcal S\\ \pi(Q)=F_2}} \exp\Big( \dashint_Q \log \sigma_1  \Big)  \mu(Q) \Big)^{q/{p_1}}\Big)^{\frac \gamma q} \\
 &\le \|u,\vec \sigma\|_{A^{\star}_{\vec{p},q}}^{\frac \gamma q} [\vec \sigma]_{H_{\vec p,q}^\infty}^{\gamma/q}\Big(   \sum_{F_1\in{\mathcal F}_1} (\langle f_1\rangle_{F_1}^{\sigma_1})^t \Big(\sum_{\substack{F_2\in {\mathcal F}_2\\ \pi_{{\mathcal F}_1}(F_2)= F_1}}  ( \langle f_2\rangle_{F_2}^{\sigma_2} )^{p_2/\gamma} \sigma_2(F)\Big)^{\frac q{p_2}}\\
 &\quad\times \Big( \sum_{\substack{F_2\in {\mathcal F}_2\\ \pi_{{\mathcal F}_1}(F_2)= F_1}}\sum_{\substack{Q\in \mathcal S\\ \pi(Q)=F_2}} \exp\Big( \dashint_Q \log \sigma_1  \Big)  \mu(Q) \Big)^{q/{p_1}}\Big)^{\frac \gamma q}\\
 &\le \hspace{-0.05cm}   \|u,\vec \sigma\|_{A^{\star}_{\vec{p},q}}^{\frac \gamma q} [\vec \sigma]_{H_{\vec p,q}^\infty}^{\gamma/q}\Big(   \sum_{F_1\in{\mathcal F}_1} (\langle f_1\rangle_{F_1}^{\sigma_1})^{p_1/\gamma} \Big(\hspace{-0.3cm}   \sum_{\substack{F_2\in {\mathcal F}_2\\ \pi_{{\mathcal F}_1}(F_2)= F_1}}\sum_{\substack{Q\in \mathcal S\\\pi(Q)=F_2}} \exp\Big( \dashint_Q \log \sigma_1  \Big)  \mu(Q) \Big) \Big)^{\frac \gamma{p_1}}\\
 &\quad\times \Big(\sum_{F_1\in{\mathcal F}_1}\sum_{\substack{F_2\in {\mathcal F}_2\\ \pi_{{\mathcal F}_1}(F_2)= F_1}}  ( \langle f_2\rangle_{F_2}^{\sigma_2} )^{p_2/\gamma} \sigma_2(F)\Big)^{\frac \gamma{p_2}}\\
 &\le \|u,\vec \sigma\|_{A^{\star}_{\vec{p},q}}^{\frac \gamma q} [\vec \sigma]_{H_{\vec p,q}^\infty}^{\gamma/q} \|f_1\|_{L^{p_1/\gamma}(\sigma_1)}\|f_2\|_{L^{p_2/\gamma}(\sigma_2)}.
\end{align*}
\end{proof}

\subsection{Proof of Theorem \ref{Caopro_1}}
~~

We first need some inequalities from \cite{Cao2018}, which can be extended to homogeneous type spaces.
\begin{lemma}[\cite{Cao2018},Lemma 2.9]
Let $r>0$ and $b \in \mathrm{B M O}(X)$. Then there is a constant $C>0$ independent of $b$ such that the following inequalities hold
\begin{align} \label{BMO.basic_1}
C^{-1} \langle|f|\rangle_Q & \leq\|f\|_{L(\log L)(X), Q} \leq C\langle|f|^{r+1}\rangle_Q^{\frac{1}{r+1}}; \\ \label{BMO.basic_2}
\sup _Q\big(\dashint_Q\left|b-b_Q\right|^r\big)^{\frac{1}{r}} & \leq C\|b\|_{\mathrm{BMO}(X)}; \\ \label{BMO.basic_3}
\left\|\left|b-b_Q\right|^r\right\|_{\exp L^{\frac{1}{r}}(X), Q} & \leq C\|b\|_{\mathrm{BMO}(X)}^r ;\\ \label{BMO.basic_4}
\langle\left|f_1 \ldots f_k g\right|\rangle_Q& \leq C\left\|f_1\right\|_{\exp L^{s_1}(X), Q} \cdots\left\|f_k\right\|_{\exp L^{s_k}(X),Q}\|g\|_{L(\log L)^{\frac{1}{s}}(X), Q},
\end{align}
where $\frac{1}{s}=\frac{1}{s_1}+\cdots+\frac{1}{s_k}$ with $s_1, \ldots, s_k \geq 1$.
\end{lemma}
\begin{lemma}[\cite{Cao2018},Lemma 3.1]
  Let $s>1, t>0$, and $\omega \in A_{\infty}(X)$. Then there holds
  \begin{align}\label{lemma:Cao3.1_1}
  & \left\|\omega^{1 / s}\right\|_{L^s(\log L)^{s t}(X), Q} \lesssim[\omega]_{A_{\infty}(X)}^t\langle\omega\rangle_Q^{1 / s} \\ \label{lemma:Cao3.1_2}
  & \|f \omega\|_{L(\log L)^t(X), Q} \lesssim[\omega]_{A_{\infty}(X)}^t \inf _{x \in Q} M_\omega\left(|f|^s\right)(x)^{1 / s}\langle\omega\rangle_Q
  \end{align}
  \end{lemma}
\begin{proof}[Proof of Theorem~\ref{Caopro_1}]
Following the approach in \cite{Cao2018}, proving our desired result is equivalent to
\begin{align*}
&\quad \left\|{\mathcal A}_{\eta ,\S,\tau,r}^\mathbf{b}(\vec f \sigma^{\frac{1}{r}})\right\|_{L^q\left(u\right)}\\
&\lesssim\left[u\right]_{A_{\infty}}^{|\tau|} \prod_{j \in \tau^{c}}\left[\sigma_j\right]_{A_{\infty}}[\vec{\omega}]_{A^{\star}_{\vec{p} / r,q/r}}^{\max _{1 \leq i \leq m}\left\{1, \frac{1}{q}\left(\frac{p_i}{r}\right)^{\prime}\right\}} \prod_{i=1}^m\left\|b_i\right\|_{\mathrm{BMO}} \left\|f_i\right\|_{L^{p_i}\left(\sigma_i\right)},
\end{align*}
where $\big(\vec{f} \sigma^{1 / r}\big):=\left(f_1 \sigma_1^{1 / r}, \ldots, f_m \sigma_m^{1 / r}\right)$.

Consider $q > 1$. Let $g \in L^{q'}(u)$ be a nonnegative function with $\|g\|_{L^{q'}( u)} = 1$, $\mathcal{D}$ a dyadic system in $(X, d, \mu)$, and $\mathcal{S}$ a sparse family from $\mathcal{D}$.
It follows from \eqref{BMO.basic_3} and \eqref{BMO.basic_4} that
\begin{align*}
& \quad\quad\int_{X} {\mathcal A}_{\eta ,\S,\tau,r}^\mathbf{b}(\vec{f} \sigma^{1 / r})(x) g(x) u d \mu \\
& \quad \lesssim \sum_{Q \in \mathcal{S}}\mu(Q)^{\an \cdot \frac{1}{r} +1}\Big\langle \prod_{i \in \tau}\langle f_i \sigma_i^{1 / r}\rangle_{Q, r}\left|b_i(x)-b_{i, Q}\right| g(x) u \Big\rangle_Q \\
& \quad \quad\times \prod_{j \in \tau^{c}}\left\|\left(b_j-b_{j, Q}\right)^r\right\|_{\exp L^{\frac{1}{r}}, Q}^{1 / r}\left\|f_j^r \sigma_j\right\|_{L(\log L)^r, Q}^{1 / r} \\
& \quad \lesssim \sum_{Q \in \mathcal{S}}\mu(Q)^{\an \cdot \frac{1}{r} + 1}\left\|g u\right\|_{L(\log L)^{|\tau|}, Q} \prod_{i \in \tau}\langle f_i \sigma^{1 / r}\rangle_{Q, r}\left\|b_i-b_{i, Q}\right\|_{\exp L, Q} \\
& \quad \quad\times \prod_{j \in \tau^c}\left\|\left(b_j-b_{j, Q}\right)^r\right\|_{\exp L^{\frac{1}{r}}, Q}^{1 / r}\left\|f_j^r \sigma_j\right\|_{L(\log L)^r, Q}^{1 / r} \\
& \quad \lesssim \sum_{Q \in \mathcal{S}}\mu(Q)^{\an \cdot \frac{1}{r}} \left(\prod_{i \in \tau}\langle f_i \sigma_i^{1 / r}\rangle_{Q, r} \prod_{j \in \tau^{c}}\left\|f_j^r \sigma_j\right\|_{L(\log L)^r, Q}^{1 / r}\right) \\
& \quad \quad\times\mu(Q)\left\|g u\right\|_{L(\log L)^{|\tau|}, Q} \prod_{i=1}^m\left\|b_i\right\|_{\mathrm{B M O}} .
\end{align*}
To proceed, setting $1<r s_j<p_j$ with $j \in \tau^{c}$, $1<s<p^{\prime}$, collecting the following estimates, 
\begin{align*}
  \langle f_i \sigma_i^{1 / r}\rangle_{Q, r}^r \leq \inf _{x \in Q} M_{\sigma_i}\left(f_i^r\right)(x)\left\langle\sigma_i\right\rangle_Q \leq\left\langle M_{\sigma_i}\left(f_i^r\right) \cdot \sigma_i\right\rangle_Q,
\end{align*}
\begin{align*}
  \left\|f_j^r \sigma_j\right\|_{L(\log L)^r, Q}^{1 / r} \overset{\text{(\ref{lemma:Cao3.1_1})}}{\lesssim}\left[\sigma_j\right]_{A_{\infty}}\left(\frac{1}{\mu(Q)} \int_Q M_{\sigma_j}\left(\left|f_j\right|^{r s_j}\right)(x)^{1 / s_j} \sigma_j d \mu\right)^{1 / r},
  \end{align*}
and 
$$
\mu(Q)\left\|g u\right\|_{L(\log L)^{|\tau|}, Q} \overset{\text{(\ref{lemma:Cao3.1_1})}}{\lesssim}\left[u\right]_{A_{\infty}}^{|\tau|} \inf _{x \in Q} M_{u}\left(|g|^s\right)(x)^{1 / s} u(Q),
$$ 
we obtain
\begin{align*}
  & \quad\int_{X} {\mathcal A}_{\eta ,\S,\tau,r}^\mathbf{b}(\vec{f} \sigma^{1 / r})(x) g(x) u d \mu \\
  & \lesssim[u]_{A_{\infty}}^{|\tau|} \prod_{j \in \tau^c}[\sigma_j]_{A_{\infty}} \prod_{i=1}^m\|b_i\|_{\mathrm{BMO}} \int_{X} \sum_{Q \in \mathcal{S}} \mu(Q)^{\an \cdot \frac{1}{r}} \prod_{i \in \tau}\big\langle M_{\sigma_i}\big(f_i^r\big) \cdot \sigma_i\big\rangle_Q^{1 / r} \\
  & \quad \times \prod_{j \in \tau^c}\big\langle M_{\sigma_j}\big(\big|f_j\big|^{r s_j}\big)^{1 / s_j} \cdot \sigma_j\big\rangle_Q ^{1 / r} \chi_Q(x) \cdot M_{u}\big(|g|^s\big)(x)^{1 / s} u d \mu \\
  & \lesssim[u]_{A_{\infty}}^{|\tau|} \prod_{j \in \tau^c}[\sigma_j]_{A_{\infty}} \prod_{i=1}^m\|b_i\|_{\mathrm{BMO}}\|M_{u}\big(|g|^s\big)^{1 / s}\|_{L^{q^{\prime}}(u)}\\
  &\quad \times\bigg\|\bigg(\sum_{Q \in \mathcal{S}} \mu(Q)^{\an \cdot \frac{1}{r}} \bigg(\prod_{i \in \tau}\big\langle M_{\sigma_i}\big(f_i^r\big) \cdot \sigma_i\big\rangle_Q \prod_{j \in \tau^{c}}\big\langle M_{\sigma_j}\big(\big|f_j\big|^{r s_j}\big)^{1 / s_j} \cdot \sigma_j\big\rangle_Q\bigg)^{1 / r} \chi_Q\bigg)^r\bigg\|_{L^{\frac{q}{r}}(u)}^{1 / r}\\
  & \overset{\text{$r \geq 1$}}{\lesssim}[u]_{A_{\infty}}^{|\tau|} \prod_{j \in \tau^c}[\sigma_j]_{A_{\infty}} \prod_{i=1}^m\|b_i\|_{\mathrm{BMO}} \left\||g|^s\right\|_{L^{p^{\prime} / s}(u)} \\
  &\quad \times\bigg\|\sum_{Q \in \mathcal{S}} \mu(Q)^{\an}\bigg(\prod_{i \in \tau}\big\langle M_{\sigma_i}\big(f_i^r\big) \cdot \sigma_i\big\rangle_Q \bigg)\bigg(\prod_{j \in \tau^{c}}\big\langle M_{\sigma_j}\big(\big|f_j\big|^{r s_j}\big)^{1 / s_j} \cdot \sigma_j\big\rangle_Q\bigg)\chi_Q\bigg\|_{L^{\frac{q}{r}}(u)}^{1 / r}
  \end{align*}

Since $\|g\|_{L^{q^{\prime}}\left(u\right)}=1$ and setting
\begin{align*}
  \mathcal{C}_0 = [u]_{A_{\infty}}^{|\tau|} \prod_{j \in \tau^c}[\sigma_j]_{A_{\infty}} \prod_{i=1}^m\|b_i\|_{\mathrm{BMO}}.
\end{align*}
Then, we demonstrate \eqref{cao.e0}.

It is derived from \eqref{eq1-thm1.1} with $\vec \omega \in A_{\vec p,q}^{\star}(X)$ that
\begin{align*}
& \quad  \int_{X} {\mathcal A}_{\eta ,\S,\tau,r}^\mathbf{b}(\vec{f} \sigma^{1 / r})(x) g(x) u d \mu \\
& \quad \lesssim[u]_{A_{\infty}}^{|\tau|} \prod_{j \in \tau^{c}}\left[\sigma_j\right]_{A_{\infty}}\left\||g|^s\right\|_{L^{p^{\prime} / s}\left(u\right)}^{1 / s}[\vec{\omega}]_{A^{\star}_{\vec{p} / r,q/r}}^{\max\limits_i \left\{1, \frac{r}{q}\left(\frac{p_i}{r}\right)^{\prime}\right\} \times \frac{1}{r}}\\
& \quad\quad \times \prod_{i=1}^m\left\|b_i\right\|_{\mathrm{BMO}} \prod_{i \in \tau}\left\|M_{\sigma_i}\left(f_i^r\right)\right\|_{L^{\frac{p_i}{r}}\left(\sigma_i\right)}^{1 / r} \prod_{j \in \tau^{c}}\left\|M_{\sigma_j}\left(\left|f_j\right|^{r s_j}\right)^{1 / s_j}\right\|_{L^{\frac{p_j}{r}}\left(\sigma_j\right)}^{1 / r} \\
& \quad \lesssim\left[u\right]_{A_{\infty}}^{|\tau|} \prod_{j \in \tau^c}\left[\sigma_j\right]_{A_{\infty}}[\vec{\omega}]_{A^{\star}_{\vec{p} / r,q/r}}^{\max\limits_i\left\{\frac{1}{r}, \frac{1}{q}\left(\frac{p_i}{r}\right)^{\prime}\right\}} \prod_{i=1}^m\left\|b_i\right\|_{\mathrm{BMO}}\left\|f_i\right\|_{L^{p_i}\left(\sigma_i\right)}.
\end{align*}

It remains to justify the case $0< q  \leq 1$. 
In light of \eqref{BMO.basic_3}, \eqref{BMO.basic_4}, and \eqref{lemma:Cao3.1_1}, we have
\begin{align*}
&\quad \left\|{\mathcal A}_{\eta ,\S,\tau,r}^\mathbf{b}(\vec{f} \sigma^{1 / r})\right\|_{L^q\left(u\right)}^q\\
&\leq \sum_{Q \in \mathcal{S}} \mu(Q)^{\an \cdot q/r} \left(\int_Q \prod_{i \in \tau}\langle f_i \sigma_i^{1 / r}\rangle_{Q, r}^q\left|b_i(x)-b_{i, Q}\right|^q u d \mu\right) \times \prod_{j \in \tau^{c}}\langle\left(b_j-b_{j, Q}\right) f_j \sigma_j^{1 / r}\rangle_{Q, r}^q\\
&\lesssim \sum_{Q \in \mathcal{S}}\mu(Q)^{\an \cdot q/r + 1}\left\|u\right\|_{L(\log L)^{q|\tau|}(X), Q} \prod_{i \in \tau}\left\langle f_i^r \sigma_i\right\rangle_Q^{q / r}\left\|\left(b_i-b_{i, Q}\right)^q\right\|_{\exp L^{\frac{1}{q}}(X), Q}\\
&\quad \times \prod_{j \in \tau^{c}}\left\|\left(b_j-b_{j, Q}\right)^r\right\|_{\exp L^{\frac{1}{r}}(X), Q}^{q/ r}\left\|f_j^r \sigma_j\right\|_{L(\log L)^r(X), Q}^{q/ r}\\
&\lesssim \sum_{Q \in \mathcal{S}}\left( \mu(Q)^{\an \cdot q/r} \prod _ { i \in \tau } \langle f_i^r \sigma_i\rangle_Q ^{q / r} \prod_{j \in \tau^{c}}\langle M_{\sigma_j}\left(\left|f_j\right|^{r s_j}\right)^{1 / s_j} \cdot \sigma_j\rangle_Q ^{q / r}\right) u(Q)\\
&\quad \times\left[u\right]_{A_{\infty}}^{q|\tau|} \prod_{j \in \tau^{c}}\left[\sigma_j\right]_{A_{\infty}}^q \prod_{i=1}^m\left\|b_i\right\|_{\mathrm{BMO}(X)}^q .
\end{align*}
utilizing the techniques in \cite[Theorem 1.2]{LMS}, one can get that
\begin{align*}
  &\big\|{\mathcal A}_{\eta ,\S,\tau,r}^\mathbf{b} (\vec{f} \sigma^{1 / r})\big\|_{L^q(u)}^q \\
  \lesssim & {[u]_{A_{\infty}}^{q|\tau|} \prod_{j \in \tau^c}[\sigma_j]_{A_{\infty}}^q[\vec{\omega}]_{A_{\vec{p} / r,q/r}}^{\max\limits_i  \big\{\big(\frac{p_i}{r}\big)^{\prime}\big\}} \prod_{i=1}^m\|b_i\|_{\mathrm{BMO}}^q } \\
  \times & \prod_{i \in \tau}\|f_i^r\|_{L^{\frac{p_i}{r}}(\sigma_i)}^{p_i / r} \prod_{j \in \tau^c}\big\|M_{\sigma_j}\big(\big|f_j\big|^{r s_j}\big)^{1 / s_j}\big\|_{L^{\frac{p_j}{r}}(\sigma_j)}^{p_j / r} \\
  \lesssim & [u]_{A_{\infty}}^{q|\tau|} \prod_{j \in \tau^c}[\sigma_j]_{A_{\infty}}^q[\vec{\omega}]_{A^{\star}_{\vec{p} / r,q/r}}^{\max\limits_i\big\{\big(\frac{p_i}{r}\big)^{\prime}\big\}} \prod_{i=1}^m\|b_i\|_{\mathrm{BMO}}^q \prod_{i=1}^m\|f_i\|_{L^{p_i}(\sigma_i)}^q.\qedhere
  \end{align*}  
\end{proof}
\subsection{Proof of Theorem \ref{est.new.}}
~~

  For simplicity, we set $\sigma_i = \omega_i^{1 - p'_i}$ and modify $f_j$ to $f_j \omega_j^{1 / p_j} \sigma_j^{-1 / p_j}$. If $0 < q \leq 1$, a similar technique applies as in the previous proof. Thus, we will focus on the case where $q > 1$. 

Let $g \in L^{q'}(\nu)$ be a nonnegative function with $\|g\|_{L^{q'}(\nu)} = 1$. Using equations \eqref{BMO.basic_3} and \eqref{BMO.basic_4}, it follows that
  \begin{align*}
      & \quad \quad \int_{X} {\mathcal A}_{\eta ,\S,\tau,r}^\mathbf{b}(\vec{f} \sigma)(x) g(x) \nu d x \\
      & \quad \lesssim \sum_{Q \in \mathcal{S}}\mu(Q)^{\an  +1}\left\langle \prod_{i \in \tau}\langle f_i \sigma_i\rangle_{Q}\left|b_i(x)-b_{i, Q}\right| g(x) \nu \right\rangle_Q  \times \prod_{j \in \tau^{c}}\left\langle\left(b_j-b_{j, Q}\right) f_j\right\rangle_{Q} \\
      & \quad \lesssim \sum_{Q \in \mathcal{S}}\mu(Q)^{\an  + 1}\left\|g \nu\right\|_{L(\log L)^{|\tau|}, Q} \prod_{i \in \tau}\langle f_i \sigma_i\rangle_{Q}\left\|b_i-b_{i, Q}\right\|_{\exp L, Q}  \times \prod_{j \in \tau^c}\left\langle\left(b_j-b_{j, Q}\right) f_j\right\rangle_{Q} \\
      & \quad \lesssim \sum_{Q \in \mathcal{S}}\mu(Q)^{\an } \prod_{i \in \tau}\langle f_i \sigma_i\rangle_{Q} \times\mu(Q)\left\|g \nu\right\|_{L(\log L)^{|\tau|}, Q} \prod_{i \in \tau}\left\|b_i\right\|_{\mathrm{BMO}} \times \prod_{j \in \tau^c}\left\langle\left(b_j-b_{j, Q}\right) f_j\right\rangle_{Q}
      \end{align*}
      
  Now we analyze the inner terms. Note that
  \begin{align*}
    \langle f_i \sigma_i\rangle_{Q} \leq \inf _{x \in Q} M_{\sigma_i}\left(f_i\right)(x)\left\langle\sigma_i\right\rangle_Q \leq\left\langle M_{\sigma_i}\left(f_i\right) \cdot \sigma_i\right\rangle_Q .
  \end{align*}
It follows from  \eqref{lemma:Cao3.1_1} that
  $$
  \mu(Q)\left\|g \nu\right\|_{L(\log L)^{|\tau|}, Q} \lesssim\left[\nu\right]_{A_{\infty}}^{|\tau|} \inf _{x \in Q} M_{\nu}\left(|g|^s\right)(x)^{1 / s} \nu(Q).
  $$
  where $1<s<p^{\prime}$. 
  
  Additionally, we claim that
  \begin{align}\label{9.16}
      &\prod_{j \in \tau^c}\langle (b_j - b_{j,Q})|f_j\sigma_j| \rangle_{Q} 
     \leq C_{n,r} \prod_{j \in \tau^c} \|b_j\|_{BMO_{\nu_j}}  \left\langle A_{\tilde{\mathcal{S}}_j}\left(|f_j\sigma_j|\right) {\nu_j}\right\rangle_Q.
  \end{align}
  Collecting these estimates, we obtain
  \begin{align*}
    & \quad \int_{X} {\mathcal A}_{\eta ,\S,\tau,r}^\mathbf{b}(\vec{f} \sigma)(x) g(x) \nu d x \\
    & \lesssim[\nu]_{A_{\infty}}^{|\tau|}  \prod_{i \in \tau}\|b_i\|_{\mathrm{BMO}} \prod_{j \in \tau^c}\|b_j\|_{\mathrm{BMO}_{\nu_j}}\int_{X} \sum_{Q \in \mathcal{S}} \mu(Q)^{\an } \prod_{i \in \tau}\big\langle M_{\sigma_i}\big(f_i\big) \cdot \sigma_i\big\rangle_Q \\
    & \times \prod_{j \in \tau^c}\big\langle A_{\tilde{\mathcal{S}}_j}\big(|f_j\sigma_j|\big) {\nu_j}\big\rangle_Q \cdot M_{\nu}\big(|g|^s\big)(x)^{1 / s} \nu d x \\
    & \lesssim[\nu]_{A_{\infty}}^{|\tau|}  \prod_{i \in \tau}\|b_i\|_{\mathrm{BMO}} \prod_{j \in \tau^c}\|b_j\|_{\mathrm{BMO}_{\nu_j}}\|M_{\nu}\big(|g|^s\big)^{1 / s}\|_{L^{q^{\prime}}(\nu)}\\
    & \times \left\|\sum_{Q \in \mathcal{S}} \mu(Q)^{\an } \big(\prod_{i \in \tau}\big\langle M_{\sigma_i}\big(f_i\big) \cdot \sigma_i\big\rangle_Q \prod_{j \in \tau^c}\big\langle A_{\tilde{\mathcal{S}}_j}\big(|f_j\sigma_j|\big) {\nu_j}\big\rangle_Q\big) \chi_Q\right\|_{L^{{q}}(\nu)}\\
    & \lesssim[\nu]_{A_{\infty}}^{|\tau|}  \prod_{i \in \tau}\|b_i\|_{\mathrm{BMO}} \prod_{j \in \tau^c}\|b_j\|_{\mathrm{BMO}_{\nu_j}}\left\||g|^s\right\|_{L^{p^{\prime} / s}\left(\nu\right)}^{1 / s}\\
    & \times \left\|\sum_{Q \in \mathcal{S}} \mu(Q)^{\an } \big(\prod_{i \in \tau}\big\langle M_{\sigma_i}\big(f_i\big) \cdot \sigma_i\big\rangle_Q \prod_{j \in \tau^c}\big\langle A_{\tilde{\mathcal{S}}_j}\big(|f_j\sigma_j|\big) {\nu_j}\big\rangle_Q\big) \chi_Q\right\|_{L^{{q}}(\nu)}\\
    \intertext{Since $\|g\|_{L^{q^{\prime}}\left(\nu\right)}=1$, }
    & \lesssim[\nu]_{A_{\infty}}^{|\tau|}  \prod_{i \in \tau}\|b_i\|_{\mathrm{BMO}} \prod_{j \in \tau^c}\|b_j\|_{\mathrm{BMO}_{\nu_j}} \times \|\mathcal{A}_{\eta,\tilde{\mathcal{S}}}(\vec {\bf A})\|_{L^{q}(\nu)}.\\
    \end{align*}
where $\vec {\bf A} = \big(\big(\big\langle M_{\sigma_i}\big(f_i\big) \cdot \sigma_i\big\rangle_Q\big)_{i \in \tau} ,\big(\big\langle A_{\tilde{\mathcal{S}}_j}\big(|f_j\sigma_j|\big) {\nu_j}\big\rangle_Q\big)_{j \in \tau^c}\big)$.


We now prove the key part of Remark~\ref{mix.remark}. Using Theorems~\ref{Horm.pro}--\ref{Thm:3}, we derive quantitative estimates for the boundedness of \(\|\mathcal{A}_{\eta,\tilde{\mathcal{S}}}(\vec{\bf A})\|_{L^{q}(\nu)}\), leading to
\begin{align*}
\|\mathcal{A}_{\eta,\tilde{\mathcal{S}}}(\vec{\bf A})\|_{L^{q}(\nu)}
&\lesssim \left\|\mathcal{A}_{\eta,\tilde{\mathcal{S}}}\right\|_{\prod_{i=1}^{m} L^{p_i}(X,\varpi_i) \to L^{q}(X,\nu)}
    \prod_{i \in \tau} \left\|M_{\sigma_i}(f_i)\right\|_{L^{p_i}(\sigma_i)} \prod_{j \in \tau^{c}} \left\|A_{\tilde{\mathcal{S}}_j}(|f_j \sigma_j|)\right\|_{L^{p_j}(\lambda_j \nu_j^{p_j})} \\
&\intertext{Since \(\nu_j = \lambda_j^{-\frac{1}{p_j}} \omega_j^{\frac{1}{p_j}}\) for \(j \in \tau^{c}\),}
&\lesssim \left\|\mathcal{A}_{\eta,\tilde{\mathcal{S}}}\right\|_{\prod_{i=1}^{m} L^{p_i}(X,\varpi_i) \to L^{q}(X,\nu)}
    \prod_{i \in \tau} \left\|M_{\sigma_i}(f_i)\right\|_{L^{p_i}(\sigma_i)} \prod_{j \in \tau^{c}} \left\|A_{\tilde{\mathcal{S}}_j}(|f_j \sigma_j|)\right\|_{L^{p_j}(\omega_j)} \\
&\lesssim \left\|\mathcal{A}_{\eta,\tilde{\mathcal{S}}}\right\|_{\prod_{i=1}^{m} L^{p_i}(X,\varpi_i) \to L^{q}(X,\nu)} \prod_{j \in \tau^{c}} \left[\omega_j\right]_{A_{p_j}}^{\max\left\{1, \frac{p_j'}{p_j}\right\}} \prod_{i=1}^{m} \|f_i\|_{L^{p_i}(\sigma_i)}.
\end{align*}

Setting
\begin{align*}
\mathcal{N}_{0} = [\nu]_{A_{\infty}}^{|\tau|} \prod_{i \in \tau} \|b_i\|_{\mathrm{BMO}} \prod_{j \in \tau^{c}} \|b_j\|_{\mathrm{BMO}_{\nu_j}} \prod_{j \in \tau^{c}} \left[\omega_j\right]_{A_{p_j}}^{\max\left\{1, \frac{p_j'}{p_j}\right\}},
\end{align*}
we obtain equation~\eqref{mix2,0}.
Combining Theorems~\ref{Horm.pro}--\ref{Thm:3}, we obtain the control constants \(\mathcal{N}_s\) for \(s = 1, 2, 3, 4\) in Remark~\ref{mix.remark}.

  Finally, we prove \eqref{9.16} as follows. 
  
Lemma \ref{zhang:6.1} implies that there exists a sparse family $\tilde{\mathcal{S}}_k \supseteq \mathcal{S}_k$ such that for all $Q \in \tilde{\mathcal{S}}_k$, we obtain
  $$
  \begin{aligned}
  \int_Q|b_j\left(x_j\right)-b_{j,Q} \| f_j\left(x_j\right)| & \lesssim  \sum_{R \in \mathcal{\mathcal { S }}_j, R \subseteq Q}\left|b_j-b_{j,R}\right|_R \int_R\left|f_j\left(x_j\right)\right| \\
  &=  \sum_{R \in \mathcal{\mathcal { S }}_j, R \subseteq Q} \frac{1}{{\nu_j}(R)} \int_{R} \left|b_j-b_{j,R}\right| dy \cdot \left|f_j\right|_R {\nu_j}(R)
  \\
  & \lesssim \left\|b_j\right\|_{\mathrm{BMO}_{\nu_j}} \sum_{R \in \tilde{\mathcal{S}}_j, R \subseteq Q}\left|f_j\right|_R {\nu_j}(R) \\
  & \lesssim \left\|b_j\right\|_{\mathrm{BMO}_{\nu_j}} \int_Q A_{\tilde{\mathcal{S}}_j}\left(\left|f_j\right|\right)\left(x_j\right) {\nu_j}\left(x_j\right).
  \end{aligned}
  $$

\subsection{Proof of Theorem \ref{zhang2018:th1.7} and  \ref{path2}}
~~


By duality, we begin by observing that
\begin{align*}
    &\quad \|{\mathcal A}_{\eta ,\S,\tau}^\mathbf{b,k,t} (\vec f)\|_{L^{q}(\lambda)}\\
&\leq\sup_{\|g\|_{L^{q'}(\lambda)}=1}\sum_{Q\in\mathcal{S}}\mu(Q)^{\an}\left(\int_{Q}|g\lambda|\prod_{i \in \tau}|b_i-b_{i,Q}|^{k_i-t_i}\right)\prod_{i \in \tau}\langle\left|f_i (b_i - b_{i,Q})^{t_i}\right|\rangle_{Q} \prod_{j \in \tau^c}\langle\left|f_j\right|\rangle_{Q}.\label{eq:DualAmk}
\end{align*}

By Lemma \ref{zhang:6.1}, there exists a sparse family $\tilde{\mathcal{S}}\subseteq {\mathcal D}$
containing $\mathcal{S}$ and such that if $Q\in\tilde{\mathcal{S}}$,
then for a.e. $x\in Q$,
\[
|b_i(x)-b_{i,Q}|\leq C\sum_{P\in\tilde{\mathcal{S}},\ P\subseteq Q} \langle|b_i-b_{j,P}|\rangle_{P} \chi_{P}(x).
\]
From this, assuming that $b_i \in {\rm BMO}_{\eta_i}$, where $\eta_i$ is a weight to be chosen later, we obtain
\[
|b_i(x)-b_{i,Q}|\leq C\|b_i\|_{{\rm BMO}_{\eta_i}(X)}\sum_{P\in\tilde{\mathcal{S}},\ P\subseteq Q}\eta_{i,P}\chi_{P}(x).
\]

In the following, we employ two methods to estimate the bloom type of ${\mathcal A}_{\eta ,\S,\tau}^\mathbf{b,k,t}$.

Before that, we claim that
\begin{align}\label{xxing}
  \int_Q|h|\left(\sum_{P \in \tilde{\mathcal{S}}, P \subseteq Q} \eta_{P} \chi_P\right)^l  \lesssim \int_Q A_{\tilde{\mathcal{S}},\eta}^l (|h|).
  \end{align}
  where $A^l_{\mathcal{\tilde{S}},{\eta}}$ denotes the operator $A_{\mathcal{\tilde{S}},{\eta}}$ iterated $l$ times.

  In fact, we observe that since the cubes in $\tilde{\mathcal S}$ are dyadic, for every $l \in \mathbb{N}$,
  \begin{eqnarray*}
  \left(\sum_{P\in\tilde{\mathcal{S}},\ P\subseteq Q}\eta_{j,P}\chi_{P}\right)^{l}&=&\sum_{P_1,P_2,\dots,P_l\subseteq Q,\, P_i\in {\tilde{\mathcal S}}}\eta_{j,P_1}\eta_{j,P_2}\dots \eta_{j,P_l}\chi_{P_1\cap P_2\cap\dots\cap P_l}\\
  &\le& l!\sum_{P_l\subseteq P_{l-1}\subseteq\dots\subseteq P_1\subseteq Q,\, P_i\in {\tilde{\mathcal S}}}\eta_{j,P_1}\eta_{j,P_2}\dots \eta_{j,P_l}\chi_{P_l}.
  \end{eqnarray*}
  Therefore,
  \[
  \int_{Q}|h|\Big(\sum_{P\in\tilde{\mathcal{S}},\ P\subseteq Q}\eta_{j,P}\chi_{P}\Big)^{l}\leq l!\sum_{P_{l}\subseteq P_{l-1}\subseteq\dots\subseteq P_{1}\subseteq Q,\,P_{i}\in\tilde{\mathcal{S}}}\eta_{j,P_{1}}\eta_{j,P_{2}}\dots\eta_{j,P_{l}}\langle|h|\rangle_{P_{l}}\mu(P_{l}).
  \]
  Further, for the first term of \eqref{eq:Int}, we obtain
  \[
  \begin{split} & \sum_{P_{l}\subseteq P_{l-1}\subseteq\dots\subseteq P_{1}\subseteq Q,\,P_{i}\in\tilde{\mathcal{S}}}\eta_{j,P_{1}}\eta_{j,P_{2}}\dots\eta_{j,P_{l}}\langle|h|\rangle_{P_{l}}\mu(P_{l})\\
   & =\sum_{P_{l-1}\subseteq\dots\subseteq P_{1}\subseteq Q,\,P_{i}\in\tilde{\mathcal{S}}}\eta_{j,P_{1}}\eta_{j,P_{2}}\dots\eta_{j,P_{l-1}}\sum_{P_{l}\subseteq P_{l-1},P_{l}\in\tilde{\mathcal{S}}}\langle|h|\rangle_{P_{l}}\int_{P_{l}}\eta_j.\\
   & \leq\sum_{P_{l-1}\subseteq\dots\subseteq P_{1}\subseteq Q,\,P_{i}\in\tilde{\mathcal{S}}}\eta_{j,P_{1}}\eta_{L_{2}}\dots\eta_{j,P_{l-1}}\int_{P_{l-1}}A_{\tilde{\mathcal{S}}}(|h|)\eta_j.\\
   & =\sum_{P_{l-1}\subseteq\dots\subseteq P_{1}\subseteq Q,\,P_{i}\in\tilde{\mathcal{S}}}\eta_{j,P_{1}}\eta_{L_{2}}\dots\eta_{j,P_{l-1}}\left\langle A_{\tilde{\mathcal{S}}}(|h|)\eta_j \right\rangle_{P_{l-1}}\mu(P_{l-1}),
  \end{split}
  \]
  where $A_{\mathcal{\tilde{S}},{\eta_j}}h=A_{\tilde{\mathcal{S}}}(h)\eta_j$ and $A_{\tilde{\mathcal{S}}}(h)=\sum\limits_{Q\in \tilde{\mathcal{S}}}h_Q\chi_Q$.
  
  Iterating this argument, we get \eqref{xxing}.

\subsection*{Maximal weight method}
~~

Without loss of generality, we consider the case for $m=3$, which means $3$-linear, and set $\tau_{m} = \{1, 2, 3\}$ and $\tau = \{1, 2\}$. We assume that $\eta_0:=\max\limits_{i=1,2}\{\eta_1,\eta_2\}$. Therefore,
  \begin{align}\notag
    & \quad \sum_{Q\in\mathcal{S}}\mu(Q)^{\an}\left(\int_{Q}|g\lambda|\prod_{i=1}^{2}|b_i-b_{i,Q}|^{k_i-t_i}\right)\prod_{i=1}^{2}\langle\left|f_i (b_i - b_{i,Q})^{t_i}\right|\rangle_{Q} \langle\left|f_3\right|\rangle_{Q}\\ \notag
    & \overset{\eqref{xxing}}{\lesssim} \prod_{i=1}^{2} \|b_i\|_{{\rm BMO}_{\eta_0}(X)}^{k_i-t_i} \|b_i\|_{{\rm BMO}_{\eta_i}(X)}^{t_i} \sum_{Q\in\mathcal{\tilde{S}}}\mu(Q)^{\an}\Big(\int_{Q}|g\lambda|\Big(\sum_{P\in\tilde{\mathcal{S}},\ P\subseteq Q}\eta_{0,P}\chi_{P}\Big)^{k_1+k_2-t_1-t_2}\Big)\\ \label{eq:Int}
    &\quad  \times\prod_{i=1}^{2}\Big(\frac{1}{\mu(Q)}\int_{Q}\Big(\sum_{P\in\tilde{\mathcal{S}},\ P\subseteq Q}\eta_{i,P}\chi_{P}\Big)^{t_i}|f_i|\Big) \langle\left|f_3\right|\rangle_{Q}.
    \end{align}
    
    From this we obtain that the right-hand side of \eqref{eq:Int} is controlled by
    \begin{align}\notag
    &\quad \prod_{i=1}^{2} \|b_i\|_{{\rm BMO}_{\eta_0}(X)}^{k_i-t_i} \|b_i\|_{{\rm BMO}_{\eta_i}(X)}^{t_i}\sum_{Q\in \mathcal{\tilde{S}}} \mu(Q)^{\an} \left(\frac{1}{\mu(Q)}\int_{Q}
    A^{k_1+k_2-t_1-t_2}_{\tilde{\mathcal{S}},\eta_0}(|g|\lambda)\right) \\ \notag
    &\quad \quad \times \left(\frac{1}{\mu(Q)}\int_{Q}A_{\tilde{\mathcal{S}},\eta_1}^{t_1}(|f_j|)\right) \left(\frac{1}{\mu(Q)}\int_{Q}A_{\tilde{\mathcal{S}},\eta_2}^{t_2}(|f_j|)\right) \mu(Q)\cdot \langle\left|f_3\right|\rangle_{Q}\\ \notag
    &=\prod_{i=1}^{2} \|b_i\|_{{\rm BMO}_{\eta_0}(X)}^{k_i-t_i} \|b_i\|_{{\rm BMO}_{\eta_i}(X)}^{t_i}\int_{X}\mathcal{A}_{\eta,\tilde{\mathcal{S}}}\left(A_{\tilde{\mathcal{S}},\eta_1}^{t_1}(|f_1|),A_{\tilde{\mathcal{S}},\eta_2}^{t_2}(|f_2|),f_3 \right)A^{k_1+k_2-t_1-t_2}_{\tilde{\mathcal{S}},\eta_0}(|g|\lambda)\\ \label{eq:DualAmk_}
    &=:\prod_{i=1}^{2} \|b_i\|_{{\rm BMO}_{\eta_0}(X)}^{k_i-t_i} \|b_i\|_{{\rm BMO}_{\eta_i}(X)}^{t_i}\int_{X}\mathcal{A}_{\eta,\tilde{\mathcal{S}}}(\vec {\bf A})A^{k_1+k_2-t_1-t_2}_{\tilde{\mathcal{S}},\eta_0}(|g|\lambda),
    \end{align}
    where $\vec {\bf A} = \left(A_{\tilde{\mathcal{S}},\eta_1}^{t_1}(|f_1|),A_{\tilde{\mathcal{S}},\eta_2}^{t_2}(|f_2|),f_3 \right)$.
    
    Using that the operator $A_{\tilde{\mathcal{S}}}$ is self-adjoint, we proceed as follows
    \begin{align*} 
        &\quad \int_{X}\mathcal{A}_{\eta,\tilde{\mathcal{S}}}(\vec {\bf A})A^{k_1+k_2-t_1-t_2}_{\tilde{\mathcal{S}},\eta_0}(|g|\lambda) =\int_{X}\mathcal{A}_{\eta,\tilde{\mathcal{S}}}(\vec {\bf A} )A_{\tilde{\mathcal{S}}}\left(A^{k_1+k_2-t_1-t_2-1}_{\tilde{\mathcal{S}},\eta_0}(|g|\lambda)\right)\\
        &=\int_{X} A_{\tilde{\mathcal{S}}}\left(A_{\tilde{\mathcal{S}},\eta_0}\left(\mathcal{A}_{\eta,\tilde{\mathcal{S}},\eta_0}(\vec {\bf A})\right)\right) \eta_0 \cdot A^{k_1+k_2-t_1-t_2-2}_{\tilde{\mathcal{S}},\eta_0}(|g|\lambda)\\
        &=\int_{X} A_{\tilde{\mathcal{S}}}\left(A^{k_1+k_2-t_1-t_2-1}_{\tilde{\mathcal{S}},\eta_0}\left(\mathcal{A}_{\eta,\tilde{\mathcal{S}},\eta_0}(\vec {\bf A})\right)\right) (|g|\lambda).
    \end{align*}
    Combining the obtained estimates with \eqref{eq:DualAmk_} yields
    \begin{equation}\label{AAm}
    \big\|{\mathcal A}_{\eta ,\S,\tau}^\mathbf{b,k,t}(\vec f) \big\|_{L^{q}(\lambda)}\lesssim\prod_{i=1}^{2} \|b_i\|_{\mathrm{BMO}_{\eta_0}}^{k_i-t_i} \|b_i\|_{\mathrm{BMO}_{\eta_i}}^{t_i}\left\|A_{\tilde{\mathcal{S}}}\left(A^{k_1+k_2-t_1-t_2-1}_{\tilde{\mathcal{S}},\eta_0}\left(\mathcal{A}_{\eta,\tilde{\mathcal{S}},\eta_0}(\vec {\bf A})\right)\right)\right\|_{L^{q}(\lambda)}.
    \end{equation}
    Setting $t = k_1+k_2-t_1-t_2-1$. Applying Theorem \ref{Horm.pro}, we obtain
\begin{align*}
    & \quad \left\|A_{\tilde{\mathcal{S}}}\left(A^{t}_{\tilde{\mathcal{S}},\eta_0}\left(\mathcal{A}_{\eta,\tilde{\mathcal{S}},\eta_0}(\vec {\bf A})\right)\right)\right\|_{L^{q}(\lambda)}\\
    & \lesssim [\lambda]_{A_{q}}^{\max\left\{ 1,\frac{1}{q-1}\right\}} \left\|A^{t}_{\tilde{\mathcal{S}},\eta_0}\left(\mathcal{A}_{\eta,\tilde{\mathcal{S}},\eta_0}(\vec {\bf A})\right)\right\|_{L^{q}(\lambda)}\\
    &= [\lambda]_{A_{q}}^{\max\left\{ 1,\frac{1}{q-1}\right\}} \left\|A^{t}_{\tilde{\mathcal{S}}}\left(\mathcal{A}_{\eta,\tilde{\mathcal{S}},\eta_0}(\vec {\bf A})\right)\right\|_{L^{q}(\lambda\eta_0^q)}\\
    & \lesssim \big([\lambda]_{A_{q}}[\lambda \eta_0^q]_{A_{q}}[\lambda \eta_0^{2q}]_{A_{q}}\dots [\lambda \eta_0^{tq}]_{A_{q}}\big)^{\max\left\{ 1,\frac{1}{q-1}\right\}}\left\|\mathcal{A}_{\eta,\tilde{\mathcal{S}},\eta_0}(\vec {\bf A})\right\|_{L^q(\lambda\eta_0^{tq})}\\
   & = \big([\lambda]_{A_{q}}[\lambda \eta_0^q]_{A_{q}}[\lambda \eta_0^{2q}]_{A_{q}}\dots [\lambda \eta_0^{tq}]_{A_{q}}\big)^{\max\left\{ 1,\frac{1}{q-1}\right\}}\left\|\mathcal{A}_{\eta,\tilde{\mathcal{S}}}(\vec {\bf A})\right\|_{L^q(\lambda\eta_0^{q(t+1)})}
    \end{align*}
    Hence, setting $\eta_0=(\mu_0/\lambda)^{1/q(t+1)}$ and applying (\ref{AAm}), we obtain
    \begin{align}\notag
    &\quad\big\|{\mathcal A}_{\eta ,\S,\tau}^\mathbf{b,k,t}(\vec f) \big\|_{L^{q}(\lambda)}\\ \label{zsabove}
    &\lesssim\prod_{i=1}^{2} \|b_i\|_{{\rm BMO}_{\eta_0}}^{k_i-t_i} \|b_i\|_{{\rm BMO}_{\eta_i}}^{t_i}
    \left([\lambda]_{A_{q}}[\mu_0]_{A_{q}}\prod_{i=1}^{t-1}[\lambda^{1-\frac{i}{t}}\mu_0^{\frac{i}{t}}]_{A_{q}}\right)^{\max\left\{ 1,\frac{1}{q-1}\right\}}\|\mathcal{A}_{\eta,\tilde{\mathcal{S}}}(\vec {\bf A})\|_{L^q(\mu_0)}.
    \end{align}
    
    By H\"older's inequality,
    $$
    \prod_{i=1}^{t-1}[\lambda^{1-\frac{i}{t}}\mu_0^{\frac{i}{t}}]_{A_{q}}\le \prod_{i=1}^{t-1}[\lambda]_{A_q}^{1-\frac{i}{t}}[\mu_0]_{A_q}^{\frac{i}{t}}=([\lambda]_{A_q}[\mu_0]_{A_q})^{\frac{t-1}{2}}.
    $$
This, along with the previous estimate, yields
    \begin{align*}
        \big\|{\mathcal A}_{\eta ,\S,\tau}^\mathbf{b,k,t}(\vec f)\big\|_{L^{q}(\lambda)}&\lesssim\prod_{i=1}^{2} \|b_i\|_{\mathrm{BMO}_{\eta_0}}^{k_i-t_i} \|b_i\|_{\mathrm{BMO}_{\eta_i}}^{t_i}
    \left([\lambda]_{A_{q}}[\mu_0]_{A_{q}}\right)^{\frac{t+1}{2}\max\left\{ 1,\frac{1}{q-1}\right\} }\|\mathcal{A}_{\eta,\tilde{\mathcal{S}}}(\vec {\bf A})\|_{L^{q}(\mu_0)}.
    \end{align*}

It follows that
\begin{align*}
\big\|{\mathcal A}_{\eta ,\S,\tau}^\mathbf{b,k,t}(\vec f)\big\|_{L^{q}(\lambda)} &\lesssim \prod_{i=1}^{2} \|b_i\|_{\mathrm{BMO}_{\eta_0}}^{k_i - t_i} \|b_i\|_{\mathrm{BMO}_{\eta_i}}^{t_i}
\left([\lambda]_{A_{q}} [\mu_0]_{A_{q}}\right)^{\frac{t + 1}{2} \max\left\{ 1, \frac{1}{q - 1} \right\} } \|\mathcal{A}_{\eta,\tilde{\mathcal{S}}}({\vec{\bf A}})\|_{L^{q}(\mu_0)} \\
&\lesssim \prod_{i=1}^{2} \|b_i\|_{\mathrm{BMO}_{\eta_0}}^{k_i - t_i} \|b_i\|_{\mathrm{BMO}_{\eta_i}}^{t_i}
\left([\lambda]_{A_{q}} [\mu_0]_{A_{q}}\right)^{\frac{t + 1}{2} \max\left\{ 1, \frac{1}{q - 1} \right\} } \\
&\times \left\|\mathcal{A}_{\eta,\tilde{\mathcal{S}}}\right\|_{\prod\limits_{i=1}^{m} L^{p_i}(X, \varpi_i) \rightarrow L^{q}(X, \mu_0)}
\prod_{i=1}^{2} \left\| A_{\tilde{\mathcal{S}}, \eta_i}^{t_i}(|f_i|) \right\|_{L^{p_i}(v_i)} \|f_3\|_{L^{p_3}(\mu_3)}.
\end{align*}

We now focus on obtaining quantitative estimates for the right side of the above inequality.

Repeat the same steps to estimate the last term in \eqref{zsabove} for $i = 1, 2$, yielding
\begin{align*}
    \left\|A_{\tilde{\mathcal{S}},\eta_i}^{t_i}(|f_i|)\right\|_{L^{p_i}(v_i)} & =  \left\|A_{\tilde{\mathcal{S}}}\left(A_{\tilde{\mathcal{S}},\eta_i}^{t_i - 1}(|f_i|)\right)\right\|_{L^{p_i}(v_i\eta_i^{p_i})}\\
    & \lesssim \left(\left[v_i \eta_i^{p_i}\right]_{A_{p_i}}\left[v_i \eta_i^{2 p_i}\right]_{A_{p_i}} \ldots\left[v_i \eta_i^{t_i p_i}\right]_{A_{p_i}}\right)^{\max \left\{1, \frac{1}{p_i-1}\right\}}\|f_i\|_{L^{p_i}(v_i\eta_i^{t_ip_i})}\\
    \intertext{Setting $\eta_i=(\mu_i/v_i)^{1/t_ip_i}$ and applying (\ref{AAm}),}
    &\lesssim \left([\mu_i]_{A_{p_i}} \prod_{k=1}^{t_i-1}\left[v_i^{1-\frac{k}{t_i}} \mu_i^{\frac{k}{t_i}}\right]_{A_{p_i}}\right)^{\max \left\{1, \frac{1}{p_i-1}\right\}}\|f_i\|_{L^{p_i}(\mu_i)}\\
    &\overset{\text{H\"older}}{\lesssim} \left([\mu_i]^{\frac{t_i+1}{2}}_{A_{p_i}} \left[v_i\right]^{\frac{t_i-1}{2}}_{A_{p_i}}\right)^{\max \left\{1, \frac{1}{p_i-1}\right\}}\|f_i\|_{L^{p_i}(\mu_i)},
\end{align*}
where $v_i,\mu_i \in A_{p_i}$, $i \in 1,2$.


Finally, it follows that \eqref{max.weight_}, by setting
\begin{align*}
  \mathcal{W}_{0}
  &= \prod_{i \in \tau} \left([\mu_i]^{\frac{t_i+1}{2}}_{A_{p_i}} \, [v_i]^{\frac{t_i-1}{2}}_{A_{p_i}}\right)^{\max \left\{1, \frac{1}{p_i-1}\right\}} \left([\lambda]_{A_q} [\mu_0]_{A_q}\right)^{\frac{t+1}{2} \max\left\{1, \frac{1}{q-1}\right\}} \\
  &\quad \times \prod_{i \in \tau} \|b_i\|_{\mathrm{BMO}_{\eta_0}}^{k_i - t_i} \|b_i\|_{\mathrm{BMO}_{\eta_i}}^{t_i}.
\end{align*}

Applying Theorems~\ref{Horm.pro}--\ref{Thm:3}, we can derive the boundedness constants $\mathcal{W}_s$ for $s = 1, 2, 3, 4$ as noted in Remark \ref{remarkb1}.

 \subsection*{Iterated weight method}
~~

Without loss generality, we consider the case for $m=4$, which means $4$-linear, and set $\tau = \{\tau(1),\tau(2),\tau(3)\} \subseteq \tau_4$ and $\tau^c = \{\tau(4)\}$.
It follows from \eqref{xxing} that
\begin{align*}
    &\quad \sum_{Q\in\mathcal{S}}\mu(Q)^{\an}\left(\int_{Q}|g\lambda|\prod_{i \in \tau}|b_i-b_{i,Q}|^{k_i-t_i}\right)\prod_{i \in \tau}\langle\left|f_i (b_i - b_{i,Q})^{t_i}\right|\rangle_{Q} \prod_{j \in \tau^c}\langle\left|f_j\right|\rangle_{Q}\\
    &\lesssim \prod_{i \in \tau}\|b_i\|^{k_i-t_i}_{\mathrm{BMO}_{\eta_i}} \sum_{Q\in\mathcal{S}}\mu(Q)^{\an}\left(\int_{Q}|g\lambda|\prod_{i \in \tau}\left(\sum_{P\in\tilde{\mathcal{S}},\ P\subseteq Q}\eta_{i,P}\chi_{P}(x)\right)^{k_i-t_i}\right)\\
    & \quad\quad\quad\quad\quad\quad\quad\quad\quad\times\prod_{i \in \tau}\Big(\frac{1}{\mu(Q)}\int_{Q}\Big(\sum_{P\in\tilde{\mathcal{S}},\ P\subseteq Q}\eta_{i,P}\chi_{P}\Big)^{t_i}|f_i|\Big) \prod_{j \in \tau^c}\langle\left|f_j\right|\rangle_{Q}
    \\
    &\lesssim \prod_{i \in \tau}\|b_i\|^{k_i}_{\mathrm{BMO}_{\eta_i}} \sum_{Q\in\mathcal{S}}\mu(Q)^{\an}\left(\int_{Q}|g\lambda|\prod_{i \in \tau}\left(\sum_{P\in\tilde{\mathcal{S}},\ P\subseteq Q}\eta_{i,P}\chi_{P}(x)\right)^{k_i-t_i}\right)\\
    & \quad\quad\quad\quad\quad\quad\quad\quad\quad\times \left(\prod_{i \in \tau}\left\langle  A_{\tilde{\mathcal{S}},\eta_i}^{t_i}(|f_i|)\right\rangle_{Q}\right) \prod_{j \in \tau^c}\langle\left|f_j\right|\rangle_{Q}.
    \\
\end{align*}

We claim that 
\begin{align}\label{path2_1}
  \int_{Q} |g \lambda| \prod_{i \in \tau} \Big( \sum_{P \in \tilde{\mathcal{S}},\ P \subseteq Q} \eta_{i,P} \chi_{P}(x) \Big)^{k_i - t_i} \lesssim \int_{Q} A_{\tilde{\mathcal{S}}, \eta_{\tau(1)}}^{k_{\tau(1)} - t_{\tau(1)}} \big( A_{\tilde{\mathcal{S}}, \eta_{\tau(2)}}^{k_{\tau(2)} - t_{\tau(2)}} \big( A_{\tilde{\mathcal{S}}, \eta_{\tau(3)}}^{k_{\tau(3)} - t_{\tau(3)}} |g \lambda| \big) \big ).
\end{align}

Clearly, this follows directly from the definitions that
\begin{align*}
  &\quad  \int_{Q} |g \lambda| \prod_{i \in \tau} \Big( \sum_{P \in \tilde{\mathcal{S}},\ P \subseteq Q} \eta_{i,P} \chi_{P}(x) \Big)^{k_i - t_i} \\
  &  \lesssim \int_{Q} A_{\tilde{\mathcal{S}}, \eta_{\tau(1)}}^{k_{\tau(1)} - t_{\tau(1)}} \Big( |g \lambda| \big( \sum_{P \in \tilde{\mathcal{S}},\ P \subseteq Q} \eta_{\tau(2),P} \chi_{P}(x) \big)^{k_{\tau(2)} - t_{\tau(2)}} \big( \sum_{P \in \tilde{\mathcal{S}},\ P \subseteq Q} \eta_{\tau(3),P} \chi_{P}(x) \big)^{k_{\tau(3)} - t_{\tau(3)}} \Big) \\
  &  = \int_{Q} A_{\tilde{\mathcal{S}}, \eta_{\tau(1)}}^{k_{\tau(1)} - t_{\tau(1)} - 1} \Big( A_{\tilde{\mathcal{S}}} \big( |g \lambda| \big( \sum_{P \in \tilde{\mathcal{S}},\ P \subseteq Q} \eta_{\tau(2),P} \chi_{P} \big)^{k_{\tau(2)} - t_{\tau(2)}} \big( \sum_{P \in \tilde{\mathcal{S}},\ P \subseteq Q} \eta_{\tau(3),P} \chi_{P} \big)^{k_{\tau(3)} - t_{\tau(3)}} \big) \eta_{\tau(1)} \Big) \\
  &  = \int_{Q} A_{\tilde{\mathcal{S}}, \eta_{\tau(1)}}^{k_{\tau(1)} - t_{\tau(1)} - 1} \Big( \eta_{\tau(1)} \sum_{Q' \in \tilde{\mathcal{S}}} \Big\langle |g \lambda| \big( \sum_{P \in \tilde{\mathcal{S}},\ P \subseteq Q} \eta_{\tau(2),P} \chi_{P} \big)^{k_{\tau(2)} - t_{\tau(2)}} \big( \sum_{P \in \tilde{\mathcal{S}},\ P \subseteq Q} \eta_{\tau(3),P} \chi_{P} \big)^{k_{\tau(3)} - t_{\tau(3)}} \Big\rangle_{Q'} \chi_{Q'} \Big) \\
  &\overset{\eqref{xxing}}{\lesssim} \int_{Q} A_{\tilde{\mathcal{S}}, \eta_{\tau(1)}}^{k_{\tau(1)} - t_{\tau(1)} - 1} \Big( \big( \sum_{Q' \in \tilde{\mathcal{S}}} \left\langle A_{\tilde{\mathcal{S}}, \eta_{\tau(2)}}^{k_{\tau(2)} - t_{\tau(2)}} \big( A_{\tilde{\mathcal{S}}, \eta_{\tau(3)}}^{k_{\tau(3)} - t_{\tau(3)}} |g \lambda| \big) \right\rangle_{Q'} \chi_{Q'} \big) \eta_{\tau(1)} \Big) \\
  &  = \int_{Q} A_{\tilde{\mathcal{S}}, \eta_{\tau(1)}}^{k_{\tau(1)} - t_{\tau(1)}} \big( A_{\tilde{\mathcal{S}}, \eta_{\tau(2)}}^{k_{\tau(2)} - t_{\tau(2)}} \big( A_{\tilde{\mathcal{S}}, \eta_{\tau(3)}}^{k_{\tau(3)} - t_{\tau(3)}} |g \lambda| \big) \big ).
\end{align*}

Combining above estimate \eqref{path2_1}, we obtain
\begin{align*}
  &\quad \sum_{Q\in\mathcal{S}}\mu(Q)^{\an}\left(\int_{Q}|g\lambda|\prod_{i \in \tau}|b_i-b_{i,Q}|^{k_i-t_i}\right)\prod_{i \in \tau}\langle\left|f_i (b_i - b_{i,Q})^{t_i}\right|\rangle_{Q} \prod_{j \in \tau}\langle\left|f_j\right|\rangle_{Q}\\
  &\lesssim  \prod_{i \in \tau}\|b_i\|^{k_i}_{\mathrm{BMO}_{\eta_i}}\int_{Q} A_{\tilde{\mathcal{S}}, \eta_{\tau(1)}}^{k_{\tau(1)} - t_{\tau(1)}} \big( A_{\tilde{\mathcal{S}}, \eta_{\tau(2)}}^{k_{\tau(2)} - t_{\tau(2)}} \big( A_{\tilde{\mathcal{S}}, \eta_{\tau(3)}}^{k_{\tau(3)} - t_{\tau(3)}} |g \lambda| \big) \big )\mathcal{A}_{\eta,\tilde{\mathcal{S}}}(\vec {\bf A})
    \intertext{Since $A_{\tilde{\mathcal{S}}}$ is self-adjoint, we conclude}
    &=\prod_{i \in \tau}\|b_i\|^{k_i}_{\mathrm{BMO}_{\eta_i}}\int_{X} A_{\tilde{\mathcal{S}}}\left(A_{\tilde{\mathcal{S}},\eta_{\tau(1)}}^{k_{\tau(1)} - t_{\tau(1)} - 1} \left(\mathcal{A}_{\eta,\tilde{\mathcal{S}},\eta_{\tau(1)}}(\vec {\bf A})\right)\right) A_{\tilde{\mathcal{S}}, \eta_{\tau(2)}}^{k_{\tau(2)} - t_{\tau(2)}} \big( A_{\tilde{\mathcal{S}}, \eta_{\tau(3)}}^{k_{\tau(3)} - t_{\tau(3)}} |g \lambda| \big)\\
    &=\prod_{i \in \tau}\|b_i\|^{k_i}_{\mathrm{BMO}_{\eta_i}}\int_{X} A_{\tilde{\mathcal{S}}} \left( A_{\tilde{\mathcal{S}}, \eta_{\tau(3)}}^{k_{\tau(3)} - t_{\tau(3)}} \left(A_{\tilde{\mathcal{S}},\eta_{\tau(2)}}^{k_{\tau(2)} - t_{\tau(2)}}\left(A_{\tilde{\mathcal{S}},\eta_{\tau(1)}}^{k_{\tau(1)} - t_{\tau(1)} - 1} \left(\mathcal{A}_{\eta,\tilde{\mathcal{S}},\eta_{\tau(1)}}(\vec {\bf A})\right)\right)\right)\right) |g \lambda|,\\
\end{align*}
where $\mathcal{A}_{\eta,\tilde{\mathcal{S}},\eta_{\tau(1)}}(\vec {\bf A})$ means that $\mathcal{A}_{\eta,\tilde{\mathcal{S}}}(\vec {\bf A}) \eta_{\tau(1)}$ and 
\begin{align*}
  \vec {\bf A} = \left(A_{\tilde{\mathcal{S}},\eta_{\tau(1)}}^{t_{\tau(1)}}(|f_{\tau(1)}|),A_{\tilde{\mathcal{S}},\eta_{\tau(2)}}^{t_{\tau(2)}}(|f_{\tau(2)}|),A_{\tilde{\mathcal{S}},\eta_{\tau(3)}}^{t_{\tau(3)}}(|f_{\tau(3)}|),f_{\tau(4)} \right).
\end{align*}

Firstly, applying Theorem \ref{Horm.pro}, we obtain 
\begin{align*}
  & \quad \left\| A_{\tilde{\mathcal{S}}} \left( A_{\tilde{\mathcal{S}}, \eta_{\tau(3)}}^{k_{\tau(3)} - t_{\tau(3)}} \left(A_{\tilde{\mathcal{S}},\eta_{\tau(2)}}^{k_{\tau(2)} - t_{\tau(2)}}\left(A_{\tilde{\mathcal{S}},\eta_{\tau(1)}}^{k_{\tau(1)} - t_{\tau(1)} - 1} \left(\mathcal{A}_{\eta,\tilde{\mathcal{S}},\eta_{\tau(1)}}(\vec {\bf A})\right)\right)\right)\right) \right\|_{L^q(\lambda)}\\
  &\lesssim [\lambda]_{A_q}^{\max\left\{ 1,\frac{1}{q-1}\right\}}\left\| A_{\tilde{\mathcal{S}}, \eta_{\tau(3)}}^{k_{\tau(3)} - t_{\tau(3)}} \left(A_{\tilde{\mathcal{S}},\eta_{\tau(2)}}^{k_{\tau(2)} - t_{\tau(2)}}\left(A_{\tilde{\mathcal{S}},\eta_{\tau(1)}}^{k_{\tau(1)} - t_{\tau(1)} - 1} \left(\mathcal{A}_{\eta,\tilde{\mathcal{S}},\eta_{\tau(1)}}(\vec {\bf A})\right)\right)\right) \right\|_{L^q(\lambda\eta_{\tau(3)}^q)}\\
  & \lesssim \big([\lambda]_{A_{q}}[\lambda \eta_{\tau(3)}^q]_{A_{q}}[\lambda \eta_{\tau(3)}^{2q}]_{A_{q}}\dots [\lambda \eta_{\tau(3)}^{(k_{\tau(3)}-t_{\tau(3)})q}]_{A_{q}}\big)^{\max\left\{ 1,\frac{1}{q-1}\right\}}\\
  &\quad \quad \quad \quad  \times \left\|A_{\tilde{\mathcal{S}},\eta_{\tau(2)}}^{k_{\tau(2)} - t_{\tau(2)}}\left(A_{\tilde{\mathcal{S}},\eta_{\tau(1)}}^{k_{\tau(1)} - t_{\tau(1)} - 1} \left(\mathcal{A}_{\eta,\tilde{\mathcal{S}},\eta_{\tau(1)}}(\vec {\bf A})\right)\right)\right\|_{L^q(\lambda\eta_{\tau(3)}^{(k_{\tau(3)}-t_{\tau(3)})q})}\\
    \intertext{Setting  $\eta_{\tau(3)}=(\xi_{\tau(3)}/\lambda)^{1/(k_{\tau(3)} - t_{\tau(3)})q}$ and following the previous progress, it}
    &\lesssim \left([\lambda]_{A_q}[\xi_{\tau(3)}]_{A_q}\right)^{\frac{k_{\tau(3)} - t_{\tau(3)} +1}{2} \max \left\{1, \frac{1}{q-1}\right\}}\\
    &\quad \quad \quad \quad \times \left\|A_{\tilde{\mathcal{S}},\eta_{\tau(2)}}^{k_{\tau(2)} - t_{\tau(2)}}\left(A_{\tilde{\mathcal{S}},\eta_{\tau(1)}}^{k_{\tau(1)} - t_{\tau(1)} - 1} \left(\mathcal{A}_{\eta,\tilde{\mathcal{S}},\eta_{\tau(1)}}(\vec {\bf A})\right)\right)\right\|_{L^q(\lambda\eta_{\tau(3)}^{(k_{\tau(3)}-t_{\tau(3)})q})},
    \intertext{We repeat the above process, setting  $\eta_{\tau(2)}=(\xi_{\tau(2)}/\xi_{\tau(3)})^{1/(k_{\tau(3)} - t_{\tau(3)})q}$, as the previous process, the above}
    &\lesssim \left([\lambda]_{A_q}[\xi_{\tau(3)}]_{A_q}\right)^{\frac{k_{\tau(3)} - t_{\tau(3)} +1}{2} \max \left\{1, \frac{1}{q-1}\right\}} \big([\xi_{\tau(3)}]^{\frac{k_{\tau(2)} - t_{\tau(2)} + 1}{2}}_{A_{q}} \left[\xi_{\tau(2)}\right]^{\frac{k_{\tau(2)} - t_{\tau(2)} - 1}{2}}_{A_{q}}\big) ^{\max \left\{1, \frac{1}{q-1}\right\}}\\
    &\quad \quad \quad \quad \times \left\|A_{\tilde{\mathcal{S}},\eta_{\tau(2)}}^{k_{\tau(2)} - t_{\tau(2)}}\left(A_{\tilde{\mathcal{S}},\eta_{\tau(1)}}^{k_{\tau(1)} - t_{\tau(1)} - 1} \left(\mathcal{A}_{\eta,\tilde{\mathcal{S}},\eta_{\tau(1)}}(\vec {\bf A})\right)\right)\right\|_{L^q(\lambda\eta_{\tau(3)}^{(k_{\tau(3)}-t_{\tau(3)})q})},
\end{align*}

Secondly, in the same way and setting $\eta_{\tau(1)}=(\xi_{\tau(1)}/\xi_{\tau(2)})^{1/(k_{\tau(1)}-t_{\tau(1)}-1)q}$, we have 
\begin{align*}
    &\quad \|A_{\tilde{\mathcal{S}}}^{k_{\tau(1)} - t_{\tau(1)} - 1} \left(\mathcal{A}_{\eta,\tilde{\mathcal{S}},\eta_{\tau(1)}}(\vec {\bf A})\right)\|_{L^q(\xi_{\tau(2)}\eta_{\tau(1)}^q)}\\
    &\lesssim \left([\xi_{\tau(2)}]^{\frac{k_{\tau(1)} - t_{\tau(1)} - 2}{2}}_{A_{q}} \left[\xi_{\tau(1)}\right]^{\frac{k_{\tau(1)} - t_{\tau(1)}}{2}}_{A_{q}}\right)^{\max \left\{1, \frac{1}{q-1}\right\}}\big\|\mathcal{A}_{\eta,\tilde{\mathcal{S}},\eta_{\tau(1)}}(\vec {\bf A})\big\|_{L^{q}(\xi_{\tau(1)})},
    \intertext{Setting $\eta_{\tau(1)}=(\zeta/\xi_{\tau(1)})^{1/q}$, the above}
    &= \left([\xi_{\tau(2)}]^{\frac{k_{\tau(1)} - t_{\tau(1)} - 2}{2}}_{A_{q}} \left[\xi_{\tau(1)}\right]^{\frac{k_{\tau(1)} - t_{\tau(1)}}{2}}_{A_{q}}\right)^{\max \left\{1, \frac{1}{q-1}\right\}}\big\|\mathcal{A}_{\eta,\tilde{\mathcal{S}}}(\vec {\bf A})\big\|_{L^{q}(\zeta)},
\end{align*}
where
  $\xi_{\tau(1)}, \xi_{\tau(2)} \in A_q$.


\begin{align*}
    &\quad\|\mathcal{A}_{\eta,\tilde{\mathcal{S}}}(\vec {\bf A})\|_{L^{q}(\zeta)} \le \left\|\mathcal{A}_{\eta,\tilde{\mathcal{S}}}\right\|_{  \prod\limits_{i=1}^{m} L^{p_i}(X,\varpi_i)\rightarrow L^{q}(X,\zeta)}
    \prod_{i \in \tau} \left\|A_{\tilde{\mathcal{S}},\eta_i}^{t_i}(|f_i|)\right\|_{L^{p_i}(\vartheta_i)} \prod_{j \in \tau^c}\|f_j\|_{L^{p_j}(\zeta_j)}.
    \end{align*}
    
    Just as in the previous proof, by setting
\begin{align*}
    \mathcal{I}_0 = \mathcal{I}_{I} \times \mathcal{I}_{II} \times \mathcal{I}_{III} \times  \prod\limits_{i \in \tau} \left([\zeta_i]^{\frac{t_i + 1}{2}}_{A_{p_i}} \left[\vartheta_i\right]^{\frac{t_i-1}{2}}_{A_{p_i}}\right)^{\max \left\{1, \frac{1}{p_i-1}\right\}},
\end{align*}
with
\begin{align*}
	\mathcal{I}_{I}& :=[\lambda]_{A_q}[\eta_{\tau({|\tau|})}]_{A_q}^{\frac{k_{\tau(|\tau|)} - t_{\tau(|\tau|)} + 1}{2}\max\{1,\frac{1}{q-1}\}},\\
	\mathcal{I}_{II}&:=\prod_{i = 2}^{|\tau|-1}\big([\xi_{\tau(i+1)}]^{\frac{k_{\tau(i)} - t_{\tau(i)} + 1}{2}}_{A_{q}} \left[\xi_{\tau(i)}\right]^{\frac{k_{\tau(i)} - t_{\tau(i)} - 1}{2}}_{A_{q}}\big) ^{\max \left\{1, \frac{1}{q-1}\right\}},\\
\mathcal{I}_{III}&:=\big([\xi_{\tau(2)}]^{\frac{k_{\tau(1)} - t_{\tau(1)} - 2}{2}}_{A_{q}} \left[\xi_{\tau(1)}\right]^{\frac{k_{\tau(1)} - t_{\tau(1)}}{2}}_{A_{q}}\big) ^{\max \left\{1, \frac{1}{q-1}\right\}},
\end{align*}
we can conclue \eqref{iter.weight_}.

Combing with Theorems \ref{Horm.pro}--\ref{Thm:3}, we get the control constants $\mathcal{I}_s$, $s=1,2,3,4$ in Remark \ref{remarkb2}.

\subsection{Proof of Theorem \ref{endingBound}}
~~

Given $\delta >0$, the maximal operator ${M_\delta }$ is defined by
\[{M_\delta }f(x) := {\left( {M({{\left| f \right|}^\delta })(x)} \right)^{\frac{1}{\delta }}}\]

\begin{definition}
Let \( f \in L^1_{\text{loc}}(X) \). The sharp maximal operator is defined by
\[
M^\sharp f(x) = \sup_{B \subseteq X} \frac{1}{\mu(B)} \int_B |f(y) - f_B| \, d\mu(y) \cdot \chi_B(x),
\]
where \( f_B = \frac{1}{\mu(B)} \int_B f(y) \, d\mu(y) \).

Building upon the structure of spaces of homogeneous type, there exist \( \mathcal{K} \) dyadic lattices \( \mathscr{D}^\mathfrak{k} \) for \( \mathfrak{k} = 1, 2, \ldots, \mathcal{K} \), such that
\[
M^\sharp f(x) \approx \sum_{\mathfrak{k}=1}^{\mathcal{K}} M_{\mathscr{D}^\mathfrak{k}}^\sharp f(x),
\]
where the dyadic sharp maximal operator \( M_{\mathscr{D}^\mathfrak{k}}^\sharp f \) is defined by
\[
M_{\mathscr{D}^\mathfrak{k}}^\sharp f(x) = \sup_{Q \in \mathscr{D}^\mathfrak{k}} \frac{1}{\mu(Q)} \int_Q |f(y) - f_Q| \, d\mu(y) \cdot \chi_Q(x).
\]
\end{definition}

\begin{lemma}[\cite{Feffer.1972}]\label{Caolemma2.7}
Let $0 < p, \delta < \infty$ and $\omega \in A_{\infty}$.

\begin{enumerate}[(i)]
    \item There exists a constant $C$, depending only on $[\omega]_{A_{\infty}}$, such that
    \[
    \| M_\delta(f) \|_{L^p(\omega)} \leq C \| M_\delta^{\sharp}(f) \|_{L^p(\omega)},
    \]
    for any function $f$ for which the left-hand side is finite.
    
    \item If $\varphi: (0, \infty) \to (0, \infty)$ is a doubling function, then there exists a constant $C$, depending on $[\omega]_{A_{\infty}}$ and the doubling constant of $\varphi$, such that
    \[
    \sup_{\lambda > 0} \varphi(\lambda) \, \omega\left( \left\{ x \in X \;\big|\; M_\delta f(x) > \lambda \right\} \right) \leq C \sup_{\lambda > 0} \varphi(\lambda) \, \omega\left( \left\{ x \in X \;\big|\; M_\delta^{\sharp} f(x) > \lambda \right\} \right),
    \]
    for every function $f$ where the left-hand side is finite.
\end{enumerate}
\end{lemma}

\begin{lemma}[\cite{Hu2004}]\label{Cao2018:2.10}
Let $s \geq 0$ and $1 \leq q < \infty$. Suppose the operator $T$ satisfies
\[
\mu\left(\left\{ x \in X \;\big|\; |T f(x)| > \lambda \right\}\right) \lesssim \int_{X} \Phi_s\left(\frac{|f(x)|^q}{\lambda^q}\right) \, d\mu,
\]
where $\Phi_s(t) = t \left(1 + \log^{+} t \right)^s$. Then, there exists a constant $c_{p,q} > 0$ such that for any $0 < p < q$, any measurable set $E \subseteq X$ with finite measure, and any function $f$ supported on $E$, the following inequality holds:
\[
\left( \dashint_E |T f(x)|^p \, d\mu \right)^{1/p} \leq c_{p,q} \left\| f^q \right\|_{L(\log L)^s, E}^{1/q}.
\]
\end{lemma}

Given a sequence of Young functions $\left\{ \Phi_i \right\}_{i=1}^m$ and a dyadic lattice $\d$, we define the multilinear fractional dyadic Orlicz maximal operator as follows
\[
\mathcal{M}_{\eta,\vec{\Phi}}(\vec{f})(x) := \sup_{Q \in \d} \mu(Q)^{\eta} \prod_{i=1}^m \|f_i\|_{\Phi_i, Q}\chi_Q(x).
\]
In particular, when $\Phi_i(t) = t$ for each $i = 1, \ldots, m$, we denote $\mathcal{M}_{\eta,\vec{\Phi}}$ by $\mathcal{M}_{\eta}$. 

\begin{lemma}\label{Caolemma2.11}
Let $\left\{ \Phi_i \right\}_{i=1}^m$ be a sequence of submultiplicative Young functions. If $\vec{\omega} \in A^{\star}_{\vec{1}, q_0}(X)$, where $q_0 := \frac{1}{m - \eta}$ with $\eta \in [0,m)$, then for any $\lambda > 0$, the following estimate holds:
\[
\omega\left( \left\{ x \in X \;\big|\; \mathcal{M}_{\eta,\vec{\Phi}}(\vec{f})(x) > \lambda^m \right\} \right) \leq [\vec{\omega}]_{A^{\star}_{\vec{1}, q_0}} \prod_{i=1}^m \left( \int_{X} \Phi_1 \circ \cdots \circ \Phi_m\left( \frac{|f_i(x)|}{\lambda} \right) \omega_i(x) \, d\mu \right)^{q_0}.
\]
\end{lemma}

The proof of Lemma \ref{Caolemma2.11} is based on the following.

\begin{lemma}[\cite{Gra2011}, Lemma 6.2]\label{Cao18:2.12}
Let $m \in \mathbb{N}$ and let $E \subseteq X$ be an arbitrary set. Suppose that each $\Phi_i$, for $i = 1, \ldots, m$, is a submultiplicative Young function. Then there exists a constant $c > 0$ such that whenever
\[
\prod_{i=1}^m \left\| f_i \right\|_{\Phi_i, E} > 1,
\]
it follows that
\[
\prod_{i=1}^m \left\| f_i \right\|_{\Phi_i, E} \leq c \prod_{i=1}^m \dashint_E \Phi_1 \circ \cdots \circ \Phi_m \left( |f_i(x)| \right) \, d\mu(x).
\]
\end{lemma}

\begin{proof}[Proof of Lemma \ref{Caolemma2.11}]
~~

Give the set
  $$
  E_{\eta,\lambda, k}:={E_{\eta,\lambda }} \cap B(0,k), \text { where } E_{\eta,\lambda}=\left\{x \in \mathbb{R}^n ; \mathcal{M}_{\eta,\vec{\Phi}}(\vec{f})(x)>\lambda^m\right\}.
  $$
By the monotone convergence theorem, it suffices to bound $E_{\eta,\lambda, k}$.

For any $x \in E_{\eta,\lambda, k}$, there is a ball $B_x \ni x$ such that
  $$
  \lambda^m<\left(\prod_{j=1}^m\left\|f_j\right\|_{\Phi_j, B_x}\right) \mu(B_x)^{\eta}.
  $$
 Therefore, we have
$E_{\eta,\lambda, k} \subseteq \bigcup\limits_{x \in E_{\eta,\lambda, k}} B_x.$
Applying the Vitali-Wiener type covering lemma as presented in \cite[Theorem 3.1]{CW1977}, there exists a sequence \( \{B_j\}_{j \in \mathbb{N}} \subseteq \{B_x\}_{x \in E_{\eta,\lambda, k}} \) such that

  $$
  E_{\eta,\lambda, k} \subseteq \bigcup_{j \in \mathbb{N}} 4 B_j \quad \text { and } \quad \lambda^m<\left(\prod_{j=1}^m\left\|f_j\right\|_{\Phi_j, B_j}\right)\mu(B_j)^{\eta}.
  $$
  It is follows from Lemma \ref{Cao18:2.12} that
\begin{align}\label{laoda}
  1 \leq \prod_{j=1}^m\left\|\frac{f_j}{\lambda}\mu(B_j)^{\frac{\eta}{m}}\right\|_{\Phi_j, B_j} \lesssim \prod_{j=1}^m \dashint_{B_j} \Phi_1 \circ \cdots \circ \Phi_m\left(\frac{\left|f_j(x)\right|}{\lambda}\mu(B_j)^{\frac{\eta}{m}}\right) d \mu.
\end{align}
Then, we obtain
  \begin{align*}
    \omega\left(E_{\eta,\lambda, k}\right) &\lesssim \sum_{i} \omega(B_i)  \leq [\vec \omega]_{A^{\star}_{\vec{1},q_0}} \sum_i\mu(B_i)^{q_0(m-\an )} \prod_{j=1}^m (\inf _{x \in B_i} \omega_j(x))^{q_0}\\
    & \overset{\text{(\ref{laoda})}}{\leq} [\vec \omega]_{A^{\star}_{\vec{1},q_0}} \sum_i \mu(B_i)^{q_0(m-\an )} \prod_{j=1}^m\left(\dashint_{B_i} \Phi_1 \circ \cdots \circ \Phi_m\left(\frac{\left|f_j(x)\right|}{\lambda}\mu(B_i)^{\frac{\eta}{m}}\right) \omega_j(x) d \mu\right)^{q_0}\\
    & = [\vec \omega]_{A^{\star}_{\vec{1},q_0}} \sum_i \mu(B_i)^{- \an  \cdot q_0} \prod_{j=1}^m\left(\int_{B_i} \Phi_1 \circ \cdots \circ \Phi_m\left(\frac{\left|f_j(x)\right|}{\lambda}\mu(B_i)^{\frac{\eta}{m}}\right) \omega_j(x) d \mu\right)^{q_0}\\
    & = [\vec \omega]_{A^{\star}_{\vec{1},q_0}} \sum_i \prod_{j=1}^m\left(\int_{B_i} \Phi_1 \circ \cdots \circ \Phi_m\left(\frac{\left|f_j(x)\right|}{\lambda}\right) \omega_j(x) d \mu\right)^{q_0}\\
    & \leq [\vec \omega]_{A^{\star}_{\vec{1},q_0}} \prod_{j=1}^m \left(\int_X \Phi_1 \circ \cdots \circ \Phi_m\left(\frac{\left|f_j(x)\right|}{\lambda}\right) \omega_j(x) d \mu\right)^{q_0}.  \end{align*}
This completes the proof.
\end{proof}

A simple calculation gives that for $0<\delta \leq 1$,
\begin{align}\label{Cao:3.4}
  M_{\mathcal{D}, \delta}^{\sharp}\left(\mathcal{A}_{\eta,\mathcal{S}, L(\log L)^r}^{\tau^{c}}(\vec{f})\right)(x) \lesssim \mathcal{M}_{\eta,L(\log L)^r}^{\tau^{c}}(\vec{f}^r)(x)^{1 / r}.
\end{align}
Gathering Lemma \ref{Caolemma2.11}, this implies
\begin{align}\label{Cao2018:3.5}
  \omega\left(\left\{x \in X \;\big|\; \mathcal{A}_{\eta,\mathcal{S}, L(\log L)^r}^{\tau^{c}}(\vec{f})(x)>\lambda^m\right\}\right) \leq [\vec \omega]_{A^{\star}_{\vec{1},q_0}} \prod_{i=1}^m\left(\int_X \Phi_r^{\left(\left|\tau^{c}\right|\right)}\left(\frac{\left|f_i(x)\right|^r}{\lambda^r}\right) \omega_i d \mu\right)^{q_0},
\end{align}
where $\Phi_r(t)=t\left(1+\log ^{+} t\right)^r$ and $\Phi_r^{(j)}=\overbrace{\Phi_r \circ \cdots \circ \Phi_r}^j$.

Hence, the inequality \eqref{eq:LlogL_1} follows from \eqref{Sparse.to.g}.
\begin{proof}[\bf Proof of \eqref{eq:LlogL_2}]
~~

We begin with stating that
\begin{align}\label{Cao18:3.6}
  \omega\left(\left\{x \in X \;\big|\; {\mathcal A}_{\eta ,\S,\tau,r}^\mathbf{b}(\vec{f})(x)  > \lambda^m \right\}\right) \lesssim_{\bf b} \sum_{\tau \subseteq \tau_{\ell}} \prod_{i=1}^m \left( \int_{X} \Phi_r^{(\ell - |\tau|)} \left( \frac{|f_i(x)|^r}{\lambda^r} \right) \omega_i \, d\mu \right)^{q_0},
\end{align}
where $\Phi_r(t) = t \left( 1 + \log^{+} t \right)^r$ and $\Phi_r^{(j)}$ denotes the $j$-fold composition of $\Phi_r$.


To verify this, consider that
$$
\Phi_r \circ \Phi_r(t) = t \left( 1 + \log^{+} t \right)^r \left( 1 + \log^{+} \left( t \left( 1 + \log^{+} t \right)^r \right) \right)^r.
$$
Therefore, we have
$$
\begin{aligned}
  t \left( 1 + \log^{+} t \right)^{2r} &\leq \Phi_r \circ \Phi_r(t) \leq t \left( 1 + \log^{+} t \right)^r \left( 1 + \log^{+} t^{r + 1} \right)^r \\
  &\leq (r + 1)^r t \left( 1 + \log^{+} t \right)^r \left( 1 + \log^{+} t \right)^r \\
  &= (r + 1)^r t \left( 1 + \log^{+} t \right)^{2r}.
\end{aligned}
$$
Inductively, we have
$$
t \left( 1 + \log^{+} t \right)^{j r} \leq \Phi_r^{(j)}(t) \leq C_{j, r} t \left( 1 + \log^{+} t \right)^{j r},
$$
which implies $\Phi_r^{(j)} \equiv \Phi_{j r}$. Moreover,
$$
\begin{aligned}
  & t^r \left( 1 + \log^{+} t \right)^{j r} \leq \Phi_{j r}(t^r) \\
  & = t^r \left( 1 + \log^{+} t^r \right)^{j r} = t^r \left( 1 + r \log^{+} t \right)^{j r} \leq r^{j r} t^r \left( 1 + \log^{+} t \right)^{j r}.
\end{aligned}
$$
Thus, for $\tau \subseteq \tau_{\ell}$, we obtain
$$
\Phi_r^{(\ell - |\tau|)}(t^r) \leq \Phi_{r \ell}(t^r) \leq C_{r, \ell} t^r \left( 1 + \log^{+} t \right)^{r \ell} = C_{r, \ell} \Phi_{r, \ell}(t).
$$

Thus, we need to demonstrate \eqref{Cao18:3.6}. We will only provide the proof for the case where $\ell=m$. 
Note that ${\mathcal A}_{\eta ,\S,\tau,r}^\mathbf{b}(\vec{f})(x) \lesssim \Delta_{\eta,\mathcal{S}, \tau}^{b, \tau}(\vec{f})(x)$, where
\begin{align*}
  \varDelta_{\eta,\mathcal{S}, \tau}^{b, \sigma}(\vec{f})(x): & =\sum_{Q \in \mathcal{S}} \mu(Q)^{\an \cdot \frac{1}{r}}\prod_{i \in \sigma}\left|b_i(x)-b_{i, Q}\right| \\ & \times \prod_{i \in \tau}\left\langle f_i\right\rangle_{Q, r} \times \prod_{j \in \tau^{c}}\left\|b_j\right\|_{\mathrm{BMO}}\left\|f_j^r\right\|_{L(\log L)^r, Q}^{1 / r}.
\end{align*}

Accordingly, we derive \eqref{Cao18:3.6} from Theorem \ref{Cao2018:th3.3}, \eqref{Cao:3.4}, and Lemmas \ref{Caolemma2.7} and \ref{Caolemma2.11}, following the proof scheme in \cite[Theorem 3.16]{Preze2009}.

\end{proof}

\begin{theorem}\label{Cao2018:th3.3}
Let $1 \leq r<\infty$, $0<\delta<\min \left\{\epsilon, \frac{1}{m}\right\}$, and $\eta \in[0,m)$. Then there holds that
  \begin{align*} 
    M_{\mathcal{D}, \delta}^{\sharp}\big(\varDelta_{\eta,\mathcal{S}, \tau}^{b, \tau}(\vec{f})\big)(x) \lesssim & \prod_{i=1}^m\|b_i\|_{\mathrm{BMO}}\big(\mathcal{M}_{\eta,L(\log L)^r}^{\tau^{c}}(\vec{f}^r)(x)^{1 / r}+M_\epsilon\big(\mathcal{A}_{\eta,\mathcal{S}, L(\log L)^r}^{\tau^{c}}(\vec{f})\big)(x)\big) \\ & +\sum_{\sigma \subsetneq \tau} \prod_{i \in \sigma}\|b_i\|_{\mathrm{BMO}} M_\epsilon\big(\varDelta_{\eta,\mathcal{S}, \tau}^{b, \sigma^{\prime}}(\vec{f})\big)(x).
    \end{align*}    
\end{theorem}
\begin{proof}
  To simplify notation and illustrate the iterative process, we focus on the case where $m=3$ and $\tau=\{1,2\}$. Without loss of generality, we assume that $\|b_j\|_{\mathrm{BMO}} = 1$ for all $j \in \tau^c$.

  For a fixed cube $Q^{\prime} \in \mathcal{D}$ containing $x$, we denote
  \begin{align*}
    c=\sum_{\substack{Q \in \mathcal{S} \\ Q \supseteq Q^{\prime}}} \mu(Q)^{\an \cdot \frac{1}{r}} \prod_{i \in \tau}\left|b_{i, Q^{\prime}}-b_{i, Q}\right|\left\langle f_i\right\rangle_{Q, r} \times \prod_{j \in \tau^{c}}\left\|b_j\right\|_{\mathrm{BMO}}\left\|f_j^r\right\|_{L(\log L)^r, Q}^{1 / r} .
  \end{align*}
Then by inserting $b_{i, Q^{\prime}}$, we obtain
\begin{align*}
  \left(\dashint_{Q^{\prime}}\left|| {\mathcal A}_{\eta ,\S,\tau}^\mathbf{b}(\vec{f})(x)|^\delta-c^\delta \right| d \mu\right)^{1 / \delta} 
  & \leq\left(\dashint_{Q^{\prime}}\left|\varDelta_{\eta,\mathcal{S}, \tau}^{b, \tau}(\vec{f})(x)-c\right|^\delta d \mu\right)^{1 / \delta} \\ & \leq \mathcal{N}_1+\mathcal{N}_2+\mathcal{N}_3+\mathcal{N}_4.\\
\end{align*}
where
\begin{align*}
  \mathcal{N}_1&=\Big(\dashint_ { Q ^ { \prime } } \big(\prod_{i=1}^2 \big| b_i(x)\big.-b_{i, Q^{\prime}} \big| \big.\\
  &\quad\quad\quad\quad\big. \times \big.\sum_{Q \in \mathcal{S}}\mu(Q)^{\an \cdot \frac{1}{r}}\left\langle f_1\right\rangle_{Q, r}\left\langle f_2\right\rangle_{Q, r}\big\|f_3^r\big\|_{L(\log L)^r, Q}^{1 / r} \chi_Q(x)\big)^\delta d \mu\Big)^{{1}/{\delta}},\\
  \mathcal{N}_2 &= \Big( \dashint_{Q'} \big( \big|b_1(x) - b_{1, Q'}\big| \big. \big.\\
  &\quad\quad\quad\quad\big.\times \big.\sum_{Q \in \mathcal{S}} \mu(Q)^{\an \cdot \frac{1}{r}} \big|b_{2, Q'} 
- b_{2,Q}\big| \langle f_1 \rangle_{Q,r} \langle f_2 \rangle_{Q,r} \big\| f_3 ^r\big\|_{L(\log L)^r, Q}^{1/r} 
\chi_{Q}(x) \big)^\delta \, d \mu \Big)^{1/\delta},\\
\mathcal{N}_3 & = \Big( \dashint_{Q'} \big( \big|b_2(x) - b_{2, Q'}\big| \big. \big.\\
&\quad\quad\quad\quad\big.\big.
\times \sum_{Q \in \mathcal{S}} \mu(Q)^{\an \cdot \frac{1}{r}} \big|b_{1, Q'} 
- b_{1,Q}\big| \langle f_1 \rangle_{Q,r} \langle f_2 \rangle_{Q,r} \big\| f_3^r \big\|_{L(\log L)^r, Q}^{1/r} 
\chi_{Q}(x) \big)^\delta \, d \mu \Big)^{1/\delta},\\
\mathcal{N}_4 & = \Big( \dashint_{Q'} \big| \sum_{Q \in \mathcal{S}} \mu(Q)^{\an \cdot \frac{1}{r}} \prod_{i=1}^{2} \big|b_{i,Q'} 
- b_{i,Q} \big|\langle f_i \rangle_{Q,r} \big\| f_3^r \big\|_{L(\log L)^r, Q}^{1/r} 
\chi_{Q}(x) - c \big|^\delta \, d \mu \Big)^{1/\delta}.\\
\end{align*}

We will examine them individually. Let \(s_1, s_2, s_3 \in (1, \infty)\) be chosen such that \(\delta s_3 < \epsilon\), \(\delta s_1^{\prime} < \epsilon\), and \(\frac{1}{s_1} + \frac{1}{s_2} + \frac{1}{s_3} = 1\). By applying Hölder's inequality and 
\eqref{BMO.basic_2}, we obtain
\begin{align*}
  \mathcal{N}_2 \leq & \Big(\dashint_{Q^{\prime}}\big|b_1(x)-b_{1, Q^{\prime}}\big|^{\delta s_1}\Big)^{\frac{1}{\delta s_1}} \\ 
  & \times\Big(\dashint_{Q^{\prime}}\big(\sum_{Q \in \mathcal{S}}\mu(Q)^{\an \cdot \frac{1}{r}}\big|b_{2, Q^{\prime}}-b_{2, Q}\big|\left\langle f_1\right\rangle_{Q, r}\left\langle f_2\right\rangle_{Q, r}\big\|f_3^r\big\|_{L(\log L)^r, Q^{\prime}}^{1 / r} \chi_Q(x)\big)^{\delta s_1^{\prime}} d \mu\Big)^{\frac{1}{\delta s_1^{\prime}}} \\ 
  \lesssim & \big\|b_1\big\|_{\mathrm{BMO}}\Big(\dashint_{Q^{\prime}}\big(\sum_{Q \in \mathcal{S}}\mu(Q)^{\an \cdot \frac{1}{r}}\big|b_2(x)-b_{2, Q}\big|\left\langle f_1\right\rangle_{Q, r}\left\langle f_2\right\rangle_{Q, r}\big\|f_3^r\big\|_{L(\log L)^r, Q^{\prime}}^{1 / r} \chi_Q(x)\big)^{\delta s_1^{\prime}} d \mu\Big)^{\frac{1}{\delta s_1^{\prime}}} \\ 
  & +\big\|b_1\big\|_{\mathrm{BMO}}\Big(\dashint_{Q^{\prime}}\big(\big|b_2(x)-b_{2, Q^{\prime}}\big| \sum_{Q \in \mathcal{S}}\mu(Q)^{\an \cdot \frac{1}{r}}\left\langle f_1\right\rangle_{Q, r}\left\langle f_2\right\rangle_{Q, r}\big\|f_3^r\big\|_{L(\log L)^r, Q}^{1 / r} \chi_Q(x)\big)^{\delta s_1^{\prime}} d \mu\Big)^{\frac{1}{\delta s_1^{\prime}}}\\
  & \lesssim\big\|b_1\big\|_{\mathrm{BMO}} M_{\delta s_1^{\prime}}\big(\varDelta_{\eta,\mathcal{S}, \tau}^{b,\{2\}}(\vec{f})\big)(x)+\big\|b_1\big\|_{\mathrm{BMO}}\Big(\dashint_{Q^{\prime}}\big|b_2(x)-b_{2, Q^{\prime}}\big|^{\delta s_2} d \mu\Big)^{\frac{1}{\delta s_2}} \\ 
  & \times\Big(\dashint_{Q^{\prime}}\big(\sum_{Q \in \mathcal{S}}\mu(Q)^{\an \cdot \frac{1}{r}}\left\langle f_1\right\rangle_{Q, r}\left\langle f_2\right\rangle_{Q, r}\big\|f_3^r\big\|_{L(\log L)^r, Q}^{1 / r} \chi_Q(x)\big)^{\delta s_3} d \mu\Big)^{\frac{1}{\delta s_3}} \\ & \lesssim\big\|b_1\big\|_{\mathrm{BMO}} M_{\delta s_1^{\prime}}\big(\varDelta_{\eta,\mathcal{S}, \tau}^{b,\{2\}}(\vec{f})\big)(x)+\prod_{i=1}^2\|b_i\|_{\mathrm{BMO}} M_{\delta s_3}\big(\mathcal{A}_{\eta,\mathcal{S}, L(\log L)^r}^{\{3\}}(\vec{f})\big)(x) \\ 
  & \leq\big\|b_1\big\|_{\mathrm{BMO}} M_\epsilon\big(\varDelta_{\eta,\mathcal{S}, \tau}^{b,\{2\}}(\vec{f})\big)(x)+\prod_{i=1}^2\|b_i\|_{\mathrm{BMO}} M_\epsilon\big(\mathcal{A}_{\eta,\mathcal{S}, L(\log L)^r}^{\{3\}}(\vec{f})\big)(x).
  \end{align*}  
Analogously, one can get that
\begin{align*}
  \mathcal{N}_3 \lesssim\left\|b_2\right\|_{\mathrm{BMO}} M_\epsilon\left(\varDelta_{\eta,\mathcal{S}, \tau}^{b,\{1\}}(\vec{f})\right)(x)+\prod_{i=1}^2\left\|b_i\right\|_{\mathrm{BMO}} M_\epsilon\left(\mathcal{A}_{\eta,\mathcal{S}, L(\log L)^r}^{\{3\}} r(\vec{f})\right)(x).
\end{align*}

In addition, it follows directly from Hölder's inequality and \eqref{BMO.basic_2} that
$$
\mathcal{N}_1 \lesssim \prod_{i=1}^2\left\|b_i\right\|_{\mathrm{BMO}} M_\epsilon\left(\mathcal{A}_{\eta,\mathcal{S}, L(\log L)^r}^{\tau^{c}}(\vec{f})\right)(x).
$$

It remains to analyze the last term. We dominate
\begin{align*}
\mathcal{N}_4 & =\left(\dashint_{Q^{\prime}}\left(\sum_{Q \subseteq Q^{\prime}} \mu(Q)^{\eta \cdot \frac{1}{r}} \prod_{i=1}^2\left|b_{i, Q^{\prime}}-b_{i, Q}\right|\left\langle f_i\right\rangle_{Q, r}\left\|f_3^r\right\|_{L(\log L)^r, Q}^{1 / r} \chi_Q(x)\right)^\delta d \mu\right)^{\frac{1}{\delta}} \\
& \leq \mathcal{N}_{41}+\mathcal{N}_{42}+\mathcal{N}_{43}+\mathcal{N}_{44}.
\end{align*}
where
\begin{align*}
  \mathcal{N}_{41}&=\Big(\dashint_ { Q ^ { \prime } } \big(\prod_{i=1}^2 \big| b_i(x)\big.-b_{i, Q^{\prime}} \big| \big.\\
  &\quad\quad\quad\quad\big. \times \big.\sum_{Q \subseteq Q'}\mu(Q)^{\an \cdot \frac{1}{r}}\left\langle f_1\right\rangle_{Q, r}\left\langle f_2\right\rangle_{Q, r}\big\|f_3^r\big\|_{L(\log L)^r, Q}^{1 / r} \chi_Q(x)\big)^\delta d \mu\Big)^{{1}/{\delta}},\\
  \mathcal{N}_{42} &= \Big( \dashint_{Q'} \big( \big|b_1(x) - b_{1, Q'}\big| \big. \big.\\
  &\quad\quad\quad\quad\big.\times \big.\sum_{Q \subseteq Q'} \mu(Q)^{\an \cdot \frac{1}{r}} \big|b_{2, Q'} 
- b_{2,Q}\big| \langle f_1 \rangle_{Q,r} \langle f_2 \rangle_{Q,r} \big\| f_3 ^r\big\|_{L(\log L)^r, Q}^{1/r} 
\chi_{Q}(x) \big)^\delta \, d \mu \Big)^{1/\delta},\\
\mathcal{N}_{43} & = \Big( \dashint_{Q'} \big( \big|b_2(x) - b_{2, Q'}\big| \big. \big.\\
&\quad\quad\quad\quad \times
\sum_{Q \subseteq Q'} \mu(Q)^{\an \cdot \frac{1}{r}} \big|b_{1, Q'} 
- b_{1,Q}\big| \langle f_1 \rangle_{Q,r} \langle f_2 \rangle_{Q,r} \big\| f_3^r \big\|_{L(\log L)^r, Q}^{1/r} 
\chi_{Q}(x) \big)^\delta \, d \mu \Big)^{1/\delta},\\
\mathcal{N}_{44} & = \Big( \dashint_{Q'} \big( \sum_{Q \subseteq Q'} \mu(Q)^{\an \cdot \frac{1}{r}} \prod_{i=1}^{2} \big|b_{i,Q'} 
- b_{i,Q} \big|\langle f_i \rangle_{Q,r} \big\| f_3^r \big\|_{L(\log L)^r, Q}^{1/r} 
\chi_{Q}(x) \big)^\delta \, d \mu \Big)^{1/\delta}.\\
\end{align*}

We start by controlling $\mathcal{N}_{41}$. Let $\delta_i \in (0, \infty)$ satisfy $\frac{1}{\delta} = \frac{1}{\delta_1} + \frac{1}{\delta_2} + \frac{1}{\delta_3}$. By Hölder's inequality, we obtain
\begin{align}\notag
	\mathcal{N}_{41} \leq & \Big(\dashint_{Q^{\prime}}\big|b_1(x)-b_{1, Q^{\prime}}\big|^{\delta_1}\big)^{\frac{1}{\delta_1}} \big(\dashint_{Q^{\prime}}\big|b_2(x)-b_{2, Q^{\prime}}\big|^{\delta_2}\Big)^{\frac{1}{\delta_2}} \\ \label{N41_1}
 & \times \Big(\dashint_{Q^{\prime}}\big(\sum_{Q \subseteq Q^{\prime}} \mu(Q)^{\an \cdot \frac{1}{r}}  \left\langle f_1\right\rangle_{Q, r}\left\langle f_2\right\rangle_{Q, r}\left\|f_3^r\right\|_{L(\log )^r, Q^{1 / r}} \chi_Q(x)\Big)^{\delta_3} d \mu\Big)^{\frac{1}{\delta_3}}.
\end{align}
To proceed, we claim that 
\begin{align}\notag
  &\quad\Big(\dashint_{Q^{\prime}}\big(\sum_{Q \subseteq Q^{\prime}} \mu(Q)^{\an \cdot \frac{1}{r}}  \left\langle f_1\right\rangle_{Q, r}\left\langle f_2\right\rangle_{Q, r}\left\|f_3^r\right\|_{L(\log )^r, Q^{1 / r}} \chi_Q(x)\big)^{\delta_3} d \mu\Big)^{\frac{1}{\delta_3}}\\ \label{claim:endingBound_1}
  &\lesssim \mu(Q')^{\an \cdot \frac{1}{r}} \prod_{i=1}^{2} \big\langle f_i\big\rangle_{Q^{\prime}, r} \big\|f_3^r\big\|_{L(\log L)^r, Q^{\prime}}^{1 / r}.
\end{align}

Gathering above estimates, we obtain
\begin{align*}
\mathcal{N}_{41} 
& \lesssim \prod_{i=1}^2\left\|b_i\right\|_{\mathrm{BMO}} \times \mu(Q')^{\an \cdot \frac{1}{r}} \prod_{i=1}^{2} \left\langle f_i\right\rangle_{Q^{\prime}, r} \left\|f_3^r\right\|_{L(\log L)^r, Q^{\prime}}^{1 / r} \\
& \leq \prod_{i=1}^2\left\|b_i\right\|_{\mathrm{BMO}} \mathcal{M}_{\eta,L(\log L)^r}^{\tau^{c}}(\vec{f})(x).
\end{align*}

Similarly, combining Hölder's inequality, \eqref{BMO.basic_2} and Lemma \ref{Cao2018:2.10}, we deduce that for any $t_i \in\left(0, \delta^{-1}\right)$ with $1=\frac{1}{t_1}+\frac{1}{t_2}+\frac{1}{t_3}$ and \eqref{N41_3}
\begin{align*}
  \mathcal{N}_{42} \lesssim & \big\|b_1\big\|_{\mathrm{BMO}}\Big(\dashint_{Q^{\prime}}\big(\sum_{Q \subseteq Q^{\prime}} \big\langle f_1\big\rangle_{Q, r} \chi_Q(x)\big)^{\delta t_1} d \mu\Big)^{\frac{1}{\delta t_1}} \times \mu(Q')^{\an \cdot \frac{1}{r}} \\
  & \times\Big(\dashint_{Q^{\prime}}\big(\sum_{Q \subseteq Q^{\prime}} \big|b_2(x)-b_{2, Q}\big|\big\langle f_2\big\rangle_{Q, r} \chi_Q(x)\big)^{\delta t_2} d \mu\Big)^{\frac{1}{\delta t_2}} \\
  & \times\Big(\dashint_{Q^{\prime}}\big(\sum_{Q \subseteq Q^{\prime}}\big\|f_3^r\big\|_{L(\log L)^r, Q} \chi_Q(x)\big)^{\delta t_3} d \mu\Big)^{\frac{1}{\delta t_3}} \\
  \lesssim & \prod_{i=1}^2\|b_i\|_{\mathrm{BMO}}\big\langle f_i\big\rangle_{Q^{\prime}, r} \times\big\|f_3^r\big\|_{L(\log L)^r, Q^{\prime}}^{1 / r} \times \mu(Q)^{\an \cdot \frac{1}{r}} \\
  \leq & \prod_{i=1}^2\|b_i\|_{\mathrm{BMO}} \mathcal{M}_{\eta,L(\log L)^r}^{\tau^{c}}(\vec{f})(x).
  \end{align*}

Since $\mathcal{N}_{43}$ is symmetric to $\mathcal{N}_{42}$, it yields that
\begin{align*}
	\mathcal{N}_{43} \lesssim \prod_{i=1}^2\left\|b_i\right\|_{\mathrm{BMO}} \mathcal{M}_{\eta,L(\log L)^r}^{\tau^{c}}(\vec{f})(x).
\end{align*}
Finally, applying the techniques in the estimates of $\mathcal{N}_{42}$ again, together with \eqref{N41_3}, we obtain
\begin{align*}
  \mathcal{N}_{44} \leq & \prod_{i=1}^2\Big(\dashint_{Q^{\prime}}\big(\sum_{Q \subsetneq Q^{\prime}}\big|b_i(x)-b_{i, Q^{\prime}}\big|\big\langle f_i\big\rangle_{Q, r} \chi_Q(x)\big)^{\delta t_i}\Big)^{\frac{1}{\delta t_i}} \\
  & \times\Big(\dashint_{Q^{\prime}}\big(\sum_{Q \subsetneq Q^{\prime}}\big\|f_3^r\big\|_{L(\log L)^r, Q^{\prime}}^{1 / r} \chi_Q(x)\big)^{\delta t_3} d \mu\Big)^{\frac{1}{\delta t_3}} \\
  \lesssim & \prod_{i=1}^2\|b_i\|_{\mathrm{BMO}} \mathcal{M}_{\eta,L(\log L)^r}^{\tau^{c}}(\vec{f})(x).
  \end{align*}  
This completes the proof.

It remains to justify \eqref{claim:endingBound_1}.
Firstly, fix cube $Q^{\prime} \in \mathcal{D}$ containing $x$ and define
 \begin{align*}
   C_{Q'}=\sum_{\substack{Q \in \mathcal{S} \\ Q \supseteq Q^{\prime}}} \mu(Q)^{\an \cdot \frac{1}{r}} \prod_{i \in \tau}\left\langle f_i\right\rangle_{Q, r} \times \prod_{j \in \tau^{c}}\left\|f_j^r\right\|_{L(\log L)^r, Q}^{1 / r} \chi_{Q'}(x).
 \end{align*}
Since
\begin{align*}
   M^{\sharp}_{\mathcal{D}}(h)(x)&= \sup_{Q' \in \mathcal{D}}\inf_{C}\left(\frac{1}{\mu(Q')}\int_{Q'}\left||h(y)|^{\delta} - C_{Q'}^{\delta} \right|dy\right)^{\frac{1}{\delta}}\\
   & \leq \sup_{Q' \in \mathcal{D}}\inf_{C} \left(\frac{1}{\mu(Q')}\int_{Q'}\left||h(y)| - C_{Q'} \right|^{\delta}dy\right)^{\frac{1}{\delta}},
\end{align*}
 we set $h (y) = \mathcal{A}_{\eta,\mathcal{S}, L(\log L)^r}^{\tau^{c}}(\vec{f})(y)$ and deduce that 
\begin{align*}
   |h(y) - C_{Q'}| = \sum_{\substack{Q \in \mathcal{S} \\ Q' \supsetneq Q}} \mu(Q)^{\an \cdot \frac{1}{r}} \prod_{i \in \tau}\left\langle f_i\right\rangle_{Q, r} \times \prod_{j \in \tau^{c}}\left\|f_j^r\right\|_{L(\log L)^r(X), Q}^{1 / r} \chi_{Q}(x).
\end{align*}
Applying Hölder's inequality with exponents $\gamma_i \in(0,1)$ satisfying $\frac{1}{\delta_3}=\frac{1}{\gamma_1}+\frac{1}{\gamma_2}+\frac{1}{\gamma_3}$, we conclude that
\begin{align}\notag
  &\quad\bigg(\frac{1}{\mu(Q')}\int_{Q'} \big|h(y) - C_{Q'}\big|^{\delta_3} d \mu \bigg)^{\frac{1}{\delta_3}} \\ \notag
  & \leq  \bigg(\prod_{i=1}^2\bigg(\dashint_{Q^{\prime}}\bigg(\sum_{Q \subsetneq Q^{\prime}}\big\langle f_i\big\rangle_{Q, r} \chi_Q(x)\bigg)^{\gamma_i} d \mu\bigg)^{\frac{1}{\gamma_i}}\bigg) \times \mu(Q')^{\an \cdot \frac{1}{r}}\\ \label{N41_2}
  & \quad  \times \bigg(\dashint_{Q^{\prime}}\bigg(\sum_{Q \subsetneq Q^{\prime}}\big\|f_3^r\big\|_{L(\log )^r(X), Q}^{1 / r} \chi_Q(x)\bigg)^{\gamma_3} d \mu\bigg)^{\frac{1}{\gamma_3}}.
\end{align}
To estimate the first product in \eqref{N41_2}, for $i = 1,2$, we have
\begin{align}\notag
  \bigg( \dashint_{Q^{\prime}} \bigg( \sum_{Q \subsetneq Q^{\prime}} \big\langle f_i \big\rangle_{Q, r} \chi_Q(x) \bigg)^{\gamma_i} \, d \mu \bigg)^{\frac{1}{\gamma_i}} &\leq  \bigg( \sum_{Q \subsetneq Q'} \frac{\mu(Q)}{\mu(Q')} \big\langle f_i \big\rangle^{\gamma_i}_{Q, r}\bigg)^{\frac{1}{\gamma_i}}\\ \notag
&\lesssim \big\langle f_i \big\rangle_{Q, r} \frac{\mu(\bigcup_{Q \subsetneq Q'}E_Q)}{\mu(Q')} \\ \label{N41_3}
&\lesssim \big\langle f_i \big\rangle_{Q', r}.
\end{align}
Together with Lemma \ref{Cao2018:2.10}, it yields that
\begin{align*}
	\bigg(\dashint_{Q^{\prime}}\bigg(\sum_{Q \subsetneq Q^{\prime}}\big\|f_3^r\big\|_{L(\log )^r(X), Q}^{1 / r} \chi_Q(x)\bigg)^{\gamma_3} d \mu\bigg)^{\frac{1}{\gamma_3}} \lesssim \big\|f_3^r\big\|_{L(\log L)^r, Q^{\prime}}^{1 / r}.
\end{align*}
 We complete the proof of \eqref{claim:endingBound_1}.\qedhere

\end{proof}

\vspace{1cm}
\noindent{\bf Acknowledgements } 
The authors would like to thank the editors and reviewers for careful reading and valuable comments, which lead to the improvement of this paper. 

\medskip 

\noindent{\bf Data Availability} Our manuscript has no associated data.

\medskip 
\noindent{\bf\Large Declarations}
\medskip 

\noindent{\bf Conflict of interest} The authors state that there is no conflict of interest.

\end{document}